\documentclass[12pt,reqno]{amsart}
\usepackage[top=0.8in, bottom=0.8in, left=1.25in, right=1.25in]{geometry}
\usepackage{amsmath,amsfonts,amssymb,amsmath,latexsym,mathrsfs,tabu}
\usepackage{setspace}
\usepackage{caption}
\usepackage{subcaption}
\usepackage{graphicx}
\usepackage{xcolor}
\usepackage{amssymb,amsmath,amsthm,eucal,mathrsfs,array,graphicx,xcolor,verbatim}
\usepackage{amsfonts}
\usepackage{mathrsfs}
\usepackage{amsmath}
\usepackage{amssymb}
\usepackage{amsthm}
\usepackage{amsmath,amsfonts,amssymb,amsmath,latexsym,mathrsfs,tabu}
\usepackage[all, cmtip]{xy}
\usepackage{xcolor}
\usepackage{ucs}
\usepackage{bbm}
\usepackage{caption}
\usepackage{subcaption}
\usepackage{graphicx, import}
\usepackage{tikz}
\def\C{\mathbb{C}}
       
\def\R{\mathbb R}
\def\Q{\mathbb Q}
\def\Z{\mathbb Z}

\def\P{\mathbb P}

\def\1{\mathbbm 1}

\def\T{\mathcal{T}}

\newcommand{\interior}[1]{\raise0.2ex\hbox{$\displaystyle{\mathop{#1}^{\circ}}$}}

\renewcommand\phi{\varphi}
\renewcommand\emptyset{\varnothing}
\newtheorem{theorem}{Theorem}[section]

\newtheorem{proposition}[theorem]{Proposition}
\newtheorem{definition}[theorem]{Definition}
\newtheorem{corollary}[theorem]{Corollary}

\numberwithin{equation}{section}

\newcommand{\Ext}{\mathrm{Ext}}

\newcommand{\ind}{\mathrm{ind}}

\newcommand{\vir}{\mathrm{vir}}

\newcommand{\Hilb}{\mathrm{Hilb}}

\newcommand{\Sym}{\mathrm{Sym}}
\newcommand{\F}{\mathcal{F}}
\def\I{\mathcal{I}}
\newcommand{\K}{\mathcal{K}}
\newcommand{\Def}{\mathrm{Def}}
\newcommand{\Obs}{\mathrm{Obs}}
\newcommand{\Pic}{\mathrm{Pic}}
\newcommand{\diag}{\mathrm{diag}}
\def\V{\mathcal{V}}

\title{K-theoretic Donaldson-Thomas theory and the Hilbert scheme of points on a surface}
\author{Noah Arbesfeld}
\address{Imperial College London, Department of Mathematics, London, UK}
\email{n.arbesfeld@imperial.ac.uk}

\begin{document}

\begin{abstract}
Integrals of characteristic classes of tautological sheaves on the Hilbert scheme of points on a surface frequently arise in enumerative problems. We use the K-theoretic Donaldson-Thomas theory of certain toric Calabi-Yau threefolds to study K-theoretic variants of such expressions.

We study limits of the K-theoretic Donaldson-Thomas partition function of a toric Calabi-Yau threefold under certain one-parameter subgroups called slopes, and formulate a condition under which two such limits coincide. We then explicitly compute the limits of components of the partition function under so-called preferred slopes, obtaining explicit combinatorial expressions related to the refined topological vertex of Iqbal, Kos\c{c}az and Vafa.

Applying these results to specific Calabi-Yau threefolds, we deduce dualities satisfied by a generating function built from tautological bundles on the Hilbert scheme of points on $\C^2$. We then use this duality to study holomorphic Euler characteristics of exterior and symmetric powers of tautological bundles on the Hilbert scheme of points on a general surface.

\text{} \newline

\end{abstract}
\maketitle


\section{Introduction}\label{introduction}

\subsection{Tautological classes on the Hilbert scheme of points on a surface}
Let $S$ be a complex smooth quasi-projective surface and let $S^{[n]}$ denote the Hilbert scheme of $n$ points on $S$. The moduli space $S^{[n]}$ is a smooth quasi-projective scheme of dimension $2n$ parametrizing zero-dimensional length $n$ subschemes $Y\subset S$, or, equivalently, ideal sheaves $\mathcal{I}_{Y}\subset \mathcal{O}_S$ of colength $n$.

Given a line bundle $\mathcal{L}$ on $S$, one can use the universal family $\Sigma_n\subset S^{[n]}\times S$ as a correspondence to produce a rank $n$ vector bundle $\mathcal{L}^{[n]}$ on $S^{[n]}$ . Namely, if $p_{S^{[n]}}$ and $p_{S}$ are the projection maps
 $$
\xymatrix{
& \ar[ld]_{p_{S^{[n]}}} S^{[n]}\times S \ar[rd]^{p_S} & \\
S^{[n]} & & S
}
 $$
we set $$\mathcal{L}^{[n]}=(p_{S^{[n]}})_*(\mathcal{O}_{\Sigma_n}\otimes p_S^{*}\mathcal{L});$$ the vector bundle $\mathcal{L}^{[n]}$ is called a \emph{tautological bundle}, and its fiber over a subscheme $Y\in S^{[n]}$ is $$H^0(\mathcal{O}_Y\otimes \mathcal{L}).$$
Characteristic classes of tautological bundles often encode geometric information. For example, series of the form \begin{align}\label{cohint} \sum_{n\geq 0} m^iy^jz^n\int_{S^{[n]}} c_i(\mathcal{L}^{[n]})s_j(\mathcal{L}^{[n]})\end{align}
 arise in \cite{MOP1} in the study of the cohomology ring of the moduli space of $K3$ surfaces. The special case \begin{align} \label{cohseg}
\sum_{n\geq 0} z^n \int_{S^{[n]}}s_{2n}(\mathcal{L}^{[n]}) 
\end{align} 
 has been the subject of recent activity; its terms count secants to the projective embedding of $S$ given by $\mathcal{L}$ and are relevant to the computation of certain Donaldson invariants on $S$.  The precise form of this series was conjectured by Lehn in \cite{L} and was recently proved, first for $S$ with numerically trivial $\mathcal{K}_S$ by Marian, Oprea and Pandharipande in \cite{MOP1} and subsequently for general $S$ by Voisin,  and Marian, Oprea and Pandharipande in \cite{V,MOP2}. 
  

We study K-theoretic variants of series like (\ref{cohint}). Given a vector bundle $\mathcal{V}$ on a scheme $X$, set \begin{align}\label{lambsym}
\Lambda^{\bullet}_m \mathcal{V}=\sum_{i\geq 0}^{\infty} (-m)^i \Lambda^{i}\mathcal{V} \in K(X)[[m]],\ \ \Sym^{\bullet}_y\mathcal{V}=\sum_{i\geq 0}^{\infty} y^i\Sym^i\mathcal{V}\in K(X)[[y]].
\end{align}
 Then, a K-theoretic version of (\ref{cohint}) is 
 \begin{align} \label{keul} &\sum_{n\geq 0} z^n \chi\bigg (S^{[n]},\Lambda^{\bullet}_m \Big(\mathcal{L}^{[n]}\Big)\otimes \Sym^{\bullet}_y \Big(\mathcal{L}^{[n]}\Big)\bigg )\nonumber \\&= \sum_{i,j,n\geq 0} (-m)^iy^jz^n \chi\bigg (S^{[n]},\Lambda^{i} \Big(\mathcal{L}^{[n]}\Big)\otimes \Sym^{j} \Big(\mathcal{L}^{[n]}\Big)\bigg) .
 \end{align}

Given an Euler characteristic $\chi(\mathscr{M},\mathcal{F})$ of a torus equivariant coherent sheaf $\mathcal{F}$ on a quasi-projective scheme $\mathscr{M}$ (equipped with a torus action), we identify $\chi(\mathscr{M},\mathcal{F})$ with its character as a torus representation; in all examples we consider, such characters are rational functions on the torus. 
 
\subsubsection{Universal series} A general framework is developed in \cite{EGL} for analyzing expressions of a form similar to (\ref{cohint}) and  (\ref{keul}) using algebraic cobordism. To be precise, let $$\psi_1(x)\in 1+x\Q[y^1_j][[x]],\ \psi_2(x)\in \Q[y^2_j][[x]]$$ be fixed power series (with coefficients in some polynomial rings $\Q[y^1_j],\Q[y^2_j]$), and for a vector bundle $\mathcal{V}$ with Chern roots $v_i$, set $$\Psi_1(\mathcal{V})=\prod_i \psi_1(v_i),\ \ \Psi_2(\mathcal{V})=\prod_i \psi_2(v_i).$$


By \cite[Thm 4.2]{EGL}, the series \begin{align}\label{bigseries} \Psi(S,\mathcal{L}) =\sum_{n\geq 0} z^n \int_{S^{[n]}} \Psi_1(\mathcal{L}^{[n]})\Psi_2(\mathcal{T}S^{[n]})
\end{align}
can be written as $$\Psi(S,\mathcal{L})= \exp(c_1(\mathcal{L})^2 A_1+ c_1(\mathcal{L})c_2(S)A_2+ c_1(S)^2A_3+c_2(S)A_4),$$ where the $A_i$ are universal series; that is, they do not depend on the choice of $(S,\mathcal{L})$.

\subsubsection{Localization expressions} The series $A_i$ can be determined by the values of $\Psi(S,\mathcal{L})$ for toric $S$. One approach to studying such values is to use equivariant localization to further reduce an \emph{equivariant} analog of (\ref{bigseries}), for which $S$ is $\C^2,$ equipped with the action of a two-dimensional torus $T$, and $\mathcal{L}$ is a twist of the trivial bundle by a $T$-character. 


The scaling action of $T={\rm diag}(t_1,t_2)$ on $\C^2$ lifts to an action on $(\C^2)^{[n]}$. The $T$-fixed points of $(\C^2)^{[n]}$ correspond to monomial ideals $I_{\lambda}\subset \C[z_1,z_2]$ of colength $n$, indexed by (two-dimensional) partitions $\lambda$ of size $n$. Here $$I_{\lambda}=\mathrm{Span}\{ z_1^{b_1}z_2^{b_2}\ |\ (b_1,b_2)\not\in \lambda\};$$ our notation for partitions is set in Section \ref{part}.



It was shown in \cite[3.2]{ES} that as a $T$-module, the tangent space to $I_{\lambda}\in (\C^2)^{[n]}$ admits a combinatorial formula in terms of the arm lengths $a(\square)$ and leg lengths $l(\square)$ of boxes in the corresponding partition: \begin{align}\label{tanchar} T_{I_{\lambda}}(\C^2)^{[n]}=\sum_{\square\in\lambda} t_1^{-l(\square)}t_2^{a(\square)+1}+t_1^{l(\square)+1}t_2^{-a(\square)}.\end{align} The definition of arm and leg lengths are recalled in Section \ref{part}. We abbreviate the $T$-character $T_{I_{\lambda}}(\C^2)^{[n]}$ by $T_{\lambda}$.

The $T$-characters of fibers of tautological bundles over the $T$-fixed points of the Hilbert scheme are also described combinatorially. Let $\mathcal{O}(l)$ denote the trivial line bundle on $\C^2$ twisted by the $T$-weight $l$. We have $$(\mathcal{O}(l))^{[n]}|_{I_{\lambda}} = \sum_{\square=(b_1,b_2)\in\lambda} lt_1^{-b_1}t_2^{-b_2}.$$
By equivariant localization, expressions of the form (\ref{bigseries}) may be written in terms of the combinatorics of partitions. This approach has been used in similar settings to compute universal series up to some given order; see for example \cite{GK1, GK2}. However, some structures of such universal series may be intractable from this perspective; for example, it is not known how to use this approach to determine the series (\ref{cohseg}) studied in \cite{MOP1, V, MOP2}.

\subsubsection{Symmetries of generating functions} In this paper, we introduce a framework to study certain tautological classes over the Hilbert scheme of points on $\C^2$. Introduce formal variables $z, y, m_1, m_2$ and $m_3$, and consider the following expression assembled from exterior and symmetric powers of tautological bundles over the Hilbert schemes $(\C^2)^{[n]}$. 
\begin{definition}
Set $F(z,m_1,m_2,m_3,y)(t_1,t_2)$ to be the series of $T$-equivariant Euler characteristics \begin{align}\label{Fdef}\sum_{n\geq 0} (-z)^{n}\chi\bigg ((\C^2)^{[n]},\Lambda^{\bullet}_{m_1}(\mathcal{O}^{[n]})^{\vee} \otimes \Lambda^{\bullet}_{m_2}(\mathcal{O}^{[n]})^{\vee} \otimes \Lambda^{\bullet}_{m_3y}\mathcal{O}^{[n]} \otimes \Sym^{\bullet}_y\mathcal{O}^{[n]}\otimes \Lambda^n\mathcal{O}^{[n]}\bigg).
\end{align}
\end{definition}


By K-theoretic equivariant localization (\cite[Thm 3.5]{T}), \begin{align*} F&(z,m_1,m_2,m_3,y)(t_1,t_2)\\&=\sum_{\lambda} \frac{z^{|\lambda|}}{\Lambda^{\bullet}(T_{\lambda}^{\vee})}\prod_{(b_1,b_2)\in\lambda} \frac{(1-m_1t_1^{b_1}t_2^{b_2})(1-m_2t_1^{b_1}t_2^{b_2})(1-m_3yt_1^{-b_1}t_2^{-b_2})}{y-t_1^{b_1}t_2^{b_2}};\end{align*}
here, given a $T$-character $V=\sum_i u_i,$ where each $u_i$ is a $T$-weight, we set $$\Lambda^{\bullet} V=\prod_i (1-u_i).$$ We regard $F$ as a power series in the variables $y$ and $z$ whose coefficients are rational functions in the remaining variables $m_1, m_2, m_3, t_1$ and $t_2$.

We show that after normalization, the series $F$ enjoys two types of symmetries: one between the ``box-counting variable'' $z$ and the ``Segre variable'' $y$, and another between the ``Chern variables'' $m_1, m_2, m_3$.

\begin{theorem}\label{symmetry}
We have \begin{align}\label{swap} \frac{F(z,m_1,m_2,m_3,y)(t_1,t_2)}{F(z,m_1,m_2,m_3,0)(t_1,t_2)}=\frac{F(z,m_1,m_3,m_2,y)(t_1,t_2)}{F(z,m_1,m_3,m_2,0)(t_1,t_2)}=\frac{F(y,m_1,m_2,m_3,z)(t_1,t_2)}{F(y,m_1,m_2,m_3,0)(t_1,t_2)}.\end{align}
\end{theorem}

Moreover, we show that the denominators appearing in (\ref{swap}) can be characterized using the plethystic exponential.

\begin{proposition}\label{denominator}
We have \begin{align}\label{denom} F(z,m_1,m_2,m_3,0)(t_1,t_2)=\exp\bigg(\sum_{n\geq 1} -\frac{z^n}{n} \frac{(1-m_1^{n})(1-m_2^{n})}{(1-t_1^{-n})(1-t_2^{-n})}\bigg).\end{align}
\end{proposition}

Note that the $n=1$ term of the sum in the argument on the right-hand side of (\ref{denom}) is the $z^1$-term of the left-hand side. We remark that Proposition \ref{denominator} is also a consequence of \cite[Thm 3.2]{WZ}. Anton Mellit has observed that certain specializations of Theorem \ref{symmetry} can also be deduced from the combinatorics of Macdonald polynomials; see \cite[Sec. 7]{M}, where such expressions are studied in connection with mixed Hodge polynomials of character varieties.

\subsubsection{Applications to tautological bundles} Theorem \ref{symmetry} and Proposition \ref{denominator} can be used to study the holomorphic Euler characteristics of exterior powers, and certain symmetric powers, of tautological bundles $\mathcal{L}^{[n]}$ on $S^{[n]}$ for general surfaces $S$. For example, we prove the following.

\begin{corollary}
If $\chi(\mathcal{O}_S)=1$, then for $n\geq k$, we have $$\chi(S^{[n]},\Sym^{k}(\mathcal{L}^{[n]}))=\binom{\chi(\mathcal{L})+k-1}{k}.$$
\end{corollary}

Theorem \ref{symmetry} and Proposition \ref{denominator} controls holomorphic Euler characteristics of exterior powers of tautological bundles associated to rank two vector bundles on $S$.

Theorem \ref{symmetry} also encodes families of nontrivial combinatorial identities. For example, it is obvious that, for $l>n$, the coefficient of $z^nm_1^l$ is $0$ on the left-hand side of (\ref{swap}); this is not obvious on the right-hand side.


\subsection{K-theoretic enumerative geometry}

We deduce Theorem \ref{symmetry} and Proposition \ref{denominator} from equalities of particular limits of K-theoretic Donaldson-Thomas partition functions of certain Calabi-Yau threefolds. 

These results follow from a more general framework that produces identities involving generating series called K-theoretic Nekrasov partition functions; these are series formed from Euler characteristics of tautological bundles over instanton moduli spaces. Many invariants of geometric interest can be phrased in terms of K-theoretic Nekrasov functions and their cohomological analogs; examples include the K-theoretic Vafa-Witten invariants studied in \cite{Th2} and the elliptic genera of moduli spaces of sheaves studied in \cite{GK2}.

Our approach to studying such functions is based on two observations. First, it is shown in \cite[Sec 8]{NO} that special limits of the K-theoretic DT partition function can be expressed in terms of the refined topological vertex of \cite{IKV}. Second, it is shown in the physics literature on ``geometric engineering'' that particular computations using this refined topological vertex can be matched with the K-theoretic Nekrasov partition functions of certain gauge theories; some examples of this phenomenon are computed in \cite{IKV, Ta}. 

Let us outline the approach in more detail.

Let $X$ be a complex smooth threefold and let $DT(X)$ denote the Hilbert scheme of curves in $X$. In \cite{NO}, Nekrasov and Okounkov introduce the K-theoretic Donaldson-Thomas partition function $Z_{DT}(X)$, a generating series formed from Euler characteristics $$\chi(DT(X,\beta,n),\tilde{\mathcal{O}}^{\vir}).$$ Here $DT(X,\beta,n)$ is a component of the Hilbert scheme of curves equipped with the symmetric perfect obstruction theory arising from its realization as a moduli space of ideal sheaves and $\tilde{\mathcal{O}}^{\vir}$ is a certain modification of the virtual structure sheaf. The precise definition will be recalled in Section \ref{prelim}.

Our applications will be in the case when $X$ is toric Calabi-Yau and the partition function $Z_{DT}(X)$ is formed from torus-equivariant Euler characteristics.

\subsubsection{Slope independence} Given a complex torus $\mathbb{T}$, a power series $Z$ with coefficients in $\Q(\mathbb{T})$ and a one-parameter subgroup $\sigma:\C^{\times}\to \mathbb{T}$, define $Z^{\sigma}$ to be the power series (again with coefficients in $\Q(\mathbb{T})$) given by \begin{align}\label{genericlimit} Z^{\sigma}(\bar{t})=\lim_{z\to 0} Z({\sigma(z)\overline{t}}),\end{align} for $\overline{t}\in \mathbb{T}$, if the limit in question exists. 

Given such a one-parameter subgroup $\sigma$ and a weight $w:\mathbb{T}\to \C^{\times}$, there exists some integer $r$ such that $w(\sigma(z)\overline{t})=z^rw(\overline{t})$ for all $z\in \C^{\times}, \overline{t}\in \mathbb{T}$. By restricting our attention to generic $\sigma$ (whose precise meaning for our purposes is given in Section \ref{secindependence}), we may assume that this $r$ is nonzero; we then say that $w$ is \emph{attracting} with respect to $\sigma$ if $r>0$, and \emph{repelling} with respect to $\sigma$ if $r<0$. 

Limits of the form (\ref{genericlimit}) are easiest to study when the terms of the power series are given by Euler characteristics over proper spaces. For example, it is shown in \cite[Prop 7.4]{NO} that, if $M_0$ is a compact moduli space equipped with a $\mathbb{T}$-equivariant symmetric perfect obstruction theory such that (1) the associated virtual canonical bundle $\mathcal{K}^{\vir}$ admits a (non-equivariant) square root and (2) the deformation and obstruction spaces are dual up to some $\mathbb{T}$-character $\kappa_0$, then for generic one-parameter subgroups $\sigma:\C^{\times}\to \ker(\kappa_0)$, the limit $$\chi(M_0,\tilde{\mathcal{O}}^{\vir})^{\sigma}$$ exists and is independent of the choice of $\sigma$. See the beginning of Section \ref{secindependence} for an example.

However, we are interested in such limits when the power series in question are Donaldson-Thomas partition functions. As the Hilbert scheme of curves in a non-compact threefold has non-compact components, our first step is to formulate a version of \cite[Prop 7.4]{NO} that holds for non-compact geometries. 

In particular, let $M$ be a component $DT(X,\beta,n)$ of the Hilbert scheme of curves in a smooth toric Calabi-Yau threefold $X$. Let $T$ now denote the three-dimensional torus acting on $X$, and let $\kappa$ be the $T$-weight of the anti-canonical bundle $\mathcal{K}^{\vee}_X$.  

A {\em slope} is a one-parameter subgroup $\sigma:\C^{\times}\to \ker(\kappa)\subset T$. Call a weight $w\in T^{\vee}$ a \emph{non-compact} direction of $X$ if the fixed locus $(DT(X))^{\ker(w)}$ has a non-proper component. We prove the following.

\begin{theorem}\label{independence}
If two generic slopes $\sigma_1$ and $\sigma_2$ share the same attracting/repelling behavior for each non-compact direction of $X$, then $$\chi(M,\tilde{\mathcal{O}}^{\vir})^{\sigma_1}=\chi(M,\tilde{\mathcal{O}}^{\vir})^{\sigma_2}\in \Q[\kappa^{1/2}].$$
\end{theorem}


\subsubsection{Vertex and edge contributions} The next step is to explicitly compute limits of $Z_{DT}(X)^{\sigma}$ for particular $X$ and $\sigma$. For toric threefolds $X$, the partition function $Z_{DT}(X)$ can be studied using equivariant localization. As shown in \cite{MNOP1}, the $T$-fixed points of $DT(X)$ are given by configurations of three-dimensional partitions along the 1-skeleton of the toric polytope $\Delta(X)$; the partition function can be written in terms of contributions from the vertices and edges of the 1-skeleton. The edge contributions admit an explicit description in terms of $(\C^2)^{[n]}$; the contribution of the relevant edge configurations are given in Propositions \ref{edge1} and \ref{edge2}. 

The vertex contributions $V(\lambda,\mu,\nu)$ are more complicated.
Their cohomological analogs are studied using the topological vertex; see \cite{AKMV, ORV}. In K-theory, it is difficult to analyze $V(\lambda,\mu,\nu)$ in full generality. However, for generic $\sigma$, the resulting limit $V(\lambda,\mu,\nu)^{\sigma}$ simplifies: the equivariant parameters enter this limit only through the character $\kappa^{1/2}$. 

Moreover, it was explained in \cite[Sec. 8]{NO} that for certain slopes $\sigma$ called \emph{preferred slopes} (roughly, those that nearly preserve a coordinate direction in each toric chart of $X$), the the limit $V(\lambda,\mu,\nu)^{\sigma}$ is related to the refined topological vertex of \cite{IKV}. Such limits are well-suited to computation. In Proposition \ref{vertexterm}, we explicate the precise relationship between such $V(\lambda,\mu,\nu)^{\sigma}$ and the refined topological vertex. In particular, the choice of direction of preferred slope corresponds to a choice of preferred direction of the refined topological vertex. 

Theorem \ref{independence} may therefore be regarded as a mathematical formulation of the independence of the refined topological partition function of toric Calabi-Yau threefolds $X$ under certain choices of preferred direction; examples of this independence have been computed in \cite{AK, IKV}. We note, however, that limits $Z_{DT}(X)^{\sigma}$ often exhibit some dependence on $\sigma$; in other words, the refined topological partition function of a threefold $X$ can depend on the choice of preferred direction. This phenomenon can be seen, for example, in \cite[Sec. 5.3]{IKV}.

We remark that in the physics literature, computations using the refined topological vertex incorporate different so-called ``framing factors''; see for instance \cite{IKV} and \cite{Ta}. The computation using DT invariants provides one consistent way to choose such factors. 

 \begin{figure}[htb] 
\begin{subfigure}{1in}
\includegraphics[height=1.5in]{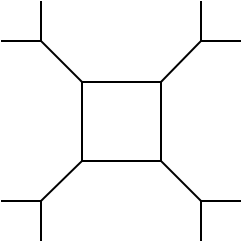}
\end{subfigure}
\hspace*{1.5in} 
\begin{subfigure}{1in}
\includegraphics[height=1.5in]{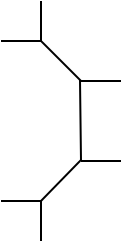}
\end{subfigure}
 \caption{The toric geometries that give rise to Theorem \ref{symmetry} and Proposition \ref{denominator}, respectively} \label{firstfig}
\end{figure}

\subsubsection{Limits of DT partition functions} We use these vertex and edge computations to compute $Z_{DT}(X)^{\sigma_1}$ and $Z_{DT}(X)^{\sigma_2},$ where $X$ is one of the Calabi-Yau threefolds whose toric skeleta are given in Figure \ref{firstfig}, and $\sigma_1$ and $\sigma_2$ are two preferred slopes satisfying the hypotheses of Theorem \ref{independence}.
After a certain specialization, we deduce Theorem \ref{symmetry} and Proposition \ref{denominator}.

When applied to threefolds with more complicated toric diagrams, this approach will produce more complicated identities involving higher-rank Nekrasov functions. For example, a new feature that should emerge is a relationship between tautological integrals over $\Hilb(\C^2)$ and those over moduli spaces of higher rank sheaves. 

\subsection{Outline} In Section \ref{prelim} we define K-theoretic Donaldson-Thomas invariants of Calabi-Yau threefolds. In Section \ref{secindependence}, we prove a general relationship between non-compact directions in a moduli space and the limits of Euler characteristics of coherent sheaves under slopes in order to deduce Theorem \ref{independence}. 
In Section \ref{secbuildingblocks}, we recall from \cite{MNOP1,NO,Ot} the descriptions of the vertex and edge contributions to K-theoretic Donaldson-Thomas invariants of toric Calabi-Yau threefolds in terms of partitions, and write down explicit combinatorial descriptions of their limits under preferred slopes. In Section \ref{threefold}, we apply these results to specific Calabi-Yau threefolds to deduce Theorem \ref{symmetry} and Proposition \ref{denominator}.
We then use Theorem \ref{symmetry} and Proposition \ref{denominator} in Section \ref{tautological} to study certain tautological classes over the Hilbert scheme of points on a general surface.

\subsection{Acknowledgments}

I am grateful to my thesis advisor Andrei Okounkov for his guidance and support, and Rahul Pandharipande for asking the question that led to the present work. I thank Chiu-Chu Melissa Liu, Alina Marian, Davesh Maulik, Anton Mellit, Andrei Negu\unichar{539}, Dragos Oprea, Petr Pushkar, Andrey Smirnov and Richard Thomas for helpful conversations and feedback. I also thank the anonymous referees for their careful attention and thoughtful suggestions.

This work was partially supported by the Department of Defense (DoD) through the National Defense Science and Engineering Graduate Fellowship (NDSEG) Program, by the National Science Foundation under Grant No. DMS-1440140 while the author was in residence at the Mathematical Sciences Research Institute in Berkeley, CA during the Spring 2018 semester, and by EPSRC grant EP/R013349/1.

\subsection{Notation}\label{part}

We fix some notation concerning partitions. 

A {\em two-dimensional partition} $\lambda$ is either the empty sequence $\emptyset$ or an ordered finite sequence $(\lambda_i)=(\lambda_1,\lambda_2,\cdots)$ of nonincreasing positive integers. Equivalently, it is a finite collection of ordered pairs $(b_1,b_2)\in \Z_{\geq 0}^2$ such that if $(b_1,b_2)\in \lambda$, then for any $b'_1, b'_2$ satisfying $0\leq b'_1\leq b_1$ and $0\leq b'_2\leq b_2,$ we also have $(b'_1,b'_2)\in \lambda$; the bijection between these two notions is given by $$\{(\lambda_i)\}\mapsto \{ (b_1,b_2)\ |\ b_2 \leq \lambda_{b_1+1}-1\}.$$ We define the {\em size} $|\lambda|$ of a partition $\lambda$ to be $\sum_{i} \lambda_i$, the {\em length} of $\lambda$ to be the length of the sequence $(\lambda_i)$, and we set $$||\lambda||^2=\sum_{i} |\lambda_i|^2.$$

We sometimes abbreviate an ordered pair $(b_1,b_2)$ by the symbol $\square$; we refer to such a square as a \emph{box} of $\lambda.$ Two-dimensional partitions are often drawn as top-left justified configurations of boxes, with $\lambda_i$ boxes in the $i$th row from the top; see Figure \ref{2dpar}.

We let $\lambda^t$ denote the conjugate partition of $\lambda$; that is $$\lambda^{t}=\{ (b_2,b_1)\ |\ (b_1,b_2)\in \lambda\}.$$ For a box $\square=(b_1,b_2)\in \lambda$, we let $a(\square)$ and $l(\square)$ denote the arm and leg lengths of $\square$, that is $$a(\square)= \lambda_{b_1+1}-b_2-1,\ l(\square)=\lambda^t_{b_2+1}-b_1-1;$$ see Figure \ref{2dpar}. When our meaning is clear, we abbreviate $a(\square)$ and $l(\square)$ by $a$ and $l$, respectively. 

\begin{figure}[htb]
 \begin{subfigure}{1in}
\includegraphics[height=1.7in]{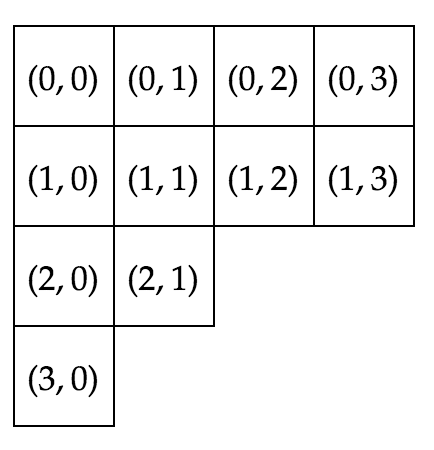}
\end{subfigure}
\hspace*{1.5in}
 \begin{subfigure}{1in}
\includegraphics[height=1.7in]{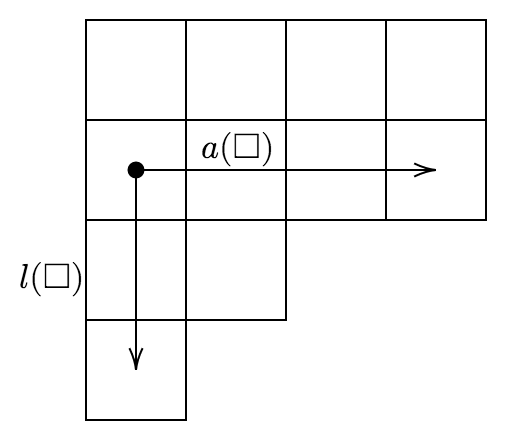}
\end{subfigure}
\caption{The two-dimensional partition $\lambda=(4,4,2,1)$ with boxes $(b_1,b_2)$ labelled. The box $\square=(1,0)\in \lambda$ has $a(\square)=3$ and $l(\square)=2$.} 
\label{2dpar}
\end{figure}

For our purposes, a {\em three-dimensional partition} $\pi$ is a (possibly infinite) collection of points of $\Z_{\geq 0}^3$ such that:

\begin{enumerate}
\item If $(b_1,b_2,b_3)\in \pi$, and $0\leq b'_i\leq b_i$ for $i=1,2,3$, then $(b_1',b'_2,b'_3)\in \pi.$ 
\item There exists some integer $B_{\pi}$ such that, for any $(b_1,b_2,b_3)\in \pi$, at most one of the coordinates $b_1, b_2, b_3$ satisfies $b_i>B_{\pi}$; in other words, the asymptotics of $\pi$ along the coordinate directions of $\R^3_{\geq 0}$ are given by (finite, possibly empty) two-dimensional partitions.
\end{enumerate}

One can visualize such a partition as a collection of unit boxes in the octant $\R_{\geq 0}^3$ centered at the points $(b_1+\frac{1}{2}, b_2+\frac{1}{2}, b_3+\frac{1}{2})$; see for example Figure \ref{3dpar}. Again, when clear, we abbreviate $(b_1,b_2,b_3)$ by the symbol $\square$. 

 \begin{figure}[htb]
\centering{
\def\svgwidth{2.5in}
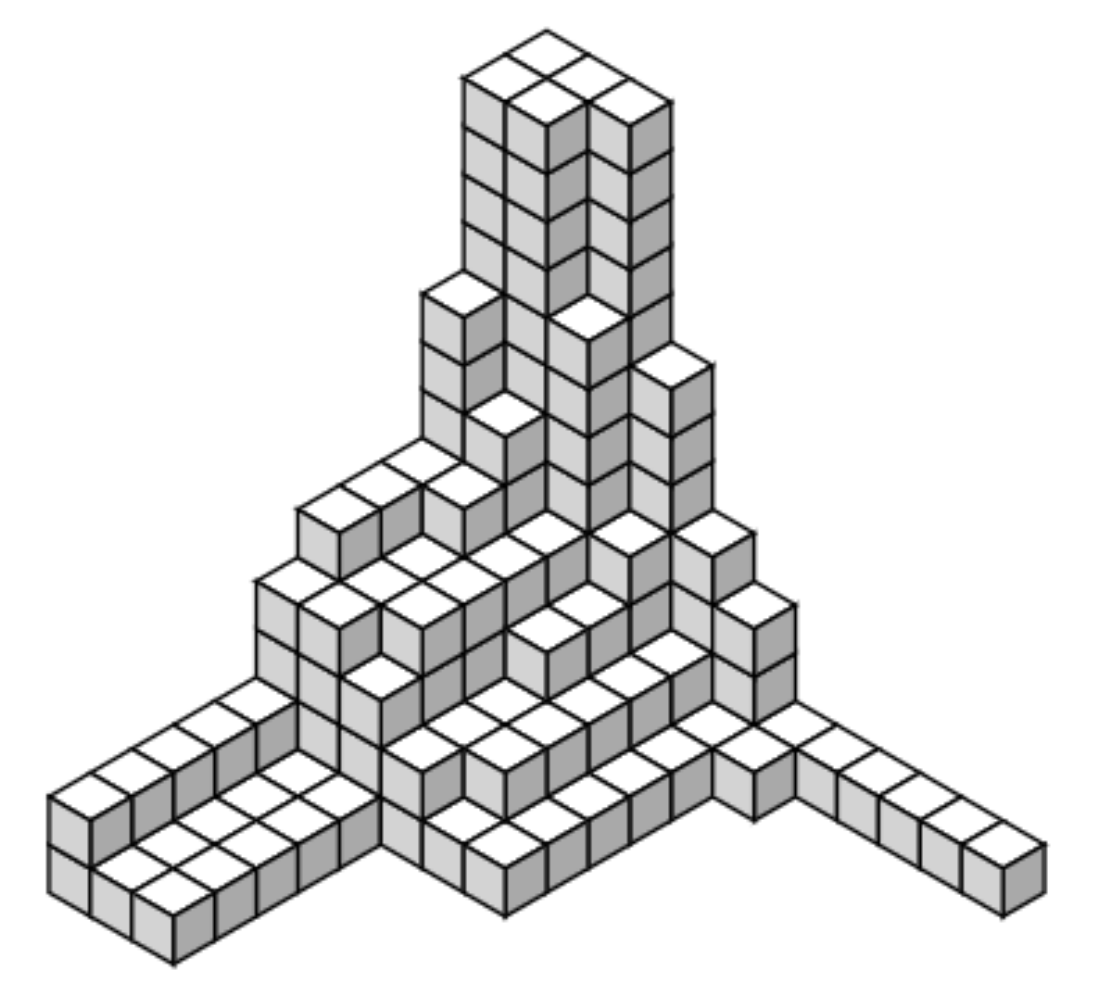
\caption{A three-dimensional partition $\pi$ with asymptotics ${\pi^{(1)}=(2,1,1)},$  $\pi^{(2)}=(1)$ and $\pi^{(3)}=(3,2)$} \label{3dpar}
}
\end{figure}

Given a three-dimensional partition $\pi$, let $\pi^{(k)}$ denote the two-dimensional partition describing the asymptotics of $\pi$ along the $k$-th coordinate axis; formally, we set $$\pi^{(1)}= \{(b_1,b_2)\ |\ (b, b_1, b_2) \in \pi\  \ {\rm for\ all}\ b\geq 0\},$$
$$\pi^{(2)}= \{(b_1,b_2)\ |\ (b_2, b, b_1) \in \pi\  \ {\rm for\ all}\ b\geq 0\},$$
$$\pi^{(3)}= \{(b_1,b_2)\ |\ (b_1, b_2, b) \in \pi\ \ {\rm for\ all}\ b\geq 0\}.$$ For example, if $\pi$ denotes the three-dimensional partition in Figure \ref{3dpar}, we have $\pi^{(1)}=(2,1,1), \pi^{(2)}=(1),$ and $\pi^{(3)}=(3,2).$ Here, we have chosen the labellings so that the two-dimensional partitions are compatible with a fixed choice of orientation on $\R_{\geq 0}^{3}$.

We let $\pi\langle k\rangle$ denote the three dimensional partition consisting of boxes belonging to infinite leg of the partition $\pi$ along the $k$-th coordinate axis; formally, we set 
$$\pi\langle 1\rangle= \{ (b,b_1,b_2)\ |\ (b_1, b_2) \in \pi^{(1)}, b\in \Z_{\geq 0} \},$$
$$\pi\langle 2\rangle= \{ (b_2,b,b_1)\ |\ (b_1, b_2) \in \pi^{(2)} , b\in \Z_{\geq 0} \},$$
$$\pi\langle 3\rangle= \{(b_1,b_2,b)\ |\ (b_1, b_2) \in \pi^{(3)}, b\in \Z_{\geq 0}\}.$$

\section{K-theoretic DT invariants}\label{prelim}

We recall the relevant aspects of K-theoretic Donaldson-Thomas invariants as introduced in \cite{NO}.

Let $X$ be a smooth threefold. We impose the further assumption $X$ is toric Calabi-Yau; in particular, $X$ is not projective. 

We remark that while this assumption has the advantage of simplifying the definition of K-theoretic Donaldson-Thomas invariants, it also obscures the generality in which these invariants can be formulated. In particular, a main source of inspiration for their definition is a conjectural connection with curve counting in associated Calabi-Yau fivefolds (see \cite{NO, Ot}); such a relationship is expected for any smooth threefold $X$.

In this section, we let $T$ denote the three-dimensional torus $\mathrm{diag}(t_1,t_2,t_3)$ acting on $X$.

\subsection{Donaldson-Thomas moduli space}\label{DTdef}

Let $DT(X)$ denote the Hilbert scheme of curves in $X$. The moduli space $DT(X)$ is a union of components $DT(X,\beta,n)$ parametrizing projective subschemes $Y\subset X$ whose components are of dimension at most 1 with $[Y]=\beta$ and $\chi(\mathcal{O}_Y)=n$; here $[Y]\in H_2(X,\Z)^{\mathrm{eff}}$ denotes the class given by the one-dimensional components of $Y$ weighted by their multiplicities. 

It is shown in \cite{Th} that $DT(X)$ admits another description as the moduli space of \emph{ideal sheaves}: rank $1$, torsion-free sheaves $\mathcal{I}_Y$ on $X$ with trivial determinant. On the level of points, one identifies a subscheme $\mathcal{O}_Y$ with its ideal sheaf $$\I_Y= {\rm ker} (\mathcal{O}_X \twoheadrightarrow \mathcal{O}_Y);$$ conversely, one passes from an ideal sheaf $\I_Y$ to the subscheme $$\mathcal{O}_Y={\rm coker}(\I_Y \hookrightarrow (\I_Y)^{\vee\vee}\cong \mathcal{O}_X).$$ 

In contrast to Hilbert schemes of points on smooth surfaces, which are smooth, Hilbert schemes of curves on threefolds are generally of unknown dimension and highly singular. However, in \cite{Th}, Thomas used the ideal sheaf description to equip $DT(X)$ with a $T$-equivariant symmetric perfect obstruction theory $E^{\bullet} \to \mathbb{L}_{DT(X)}$ in the sense of Behrend-Fantechi (see \cite{BF}). By the constructions of, for example, \cite[5.4]{BF}, \cite{CFK}, and \cite[Sec 2]{Lee}, the obstruction theory gives rise to the $T$-equivariant virtual sheaves needed for K-theoretic enumerative computations:  the \emph{virtual structure sheaf} $$\mathcal{O}^{\vir}\in K_T(DT(X)),$$ the \emph{virtual tangent sheaf} $$\mathcal{T}^{\vir} \in K_T(DT(X)),$$ and the \emph{virtual canonical bundle} $$\mathcal{K}^{\vir}=\det (\mathcal{T}^{\vir})^{\vee} \in \Pic_T(DT(X)).$$ 

In particular, from the obstruction theory of \cite{Th}, one concludes $$\mathcal{T}^{\vir}_{{\mathcal{I}}}= \Def_{\I}-\Obs_{\I} = \Ext^1(\I,\I)-\Ext^2(\I,\I),$$ and hence $$\K^{\vir}_{{\mathcal{I}}}= \frac{\det\Obs_{\mathcal{I}}}{\det\Def_{\mathcal{I}}}.$$ As $X$ is Calabi-Yau, the anticanonical bundle $\K_{X}^{\vee}$ is a trivial bundle twisted by some $T$-weight $\kappa$; by Serre duality, we have \begin{align}\label{obssym} \Def_{\mathcal{I}}=(\Obs_{\mathcal{I}})^{\vee}\cdot \kappa,\end{align} as $T$-characters.

\subsection{Modified virtual structure sheaf}\label{virtstruct}
 
The definition of K-theoretic Donaldson-Thomas invariants incorporates a twist of the virtual structure sheaf by a square root of the virtual canonical bundle. This modification endows the resulting invariants with certain symmetries that make for easier analysis; it also aligns the invariants with the indices of certain Dirac operators of interest in physics (see for example, \cite[Sec. 3.2.7]{O}). 

The existence of such a square root is ensured by the following proposition.

\begin{proposition}[{\cite[Sec 6]{NO}}]
The line bundle $\K^{\vir}$ admits a square root $(\K^{\vir})^{1/2}$ in $\Pic(DT(X))$. Moreover, if $\tilde{T}$ is a minimal cover of $T$ on which the square root $\kappa^{1/2}$ is defined, the line bundle $(\K^{\vir})^{1/2}$ carries a canonical $\tilde{T}$-equivariant structure.
\end{proposition}
 
We remark that in the special case where $X$ is Calabi-Yau, this proposition follows from an argument using Serre duality given in \cite[6.2.1]{NO}; such a statement is expected for general threefolds $X$, and is proven in \cite[Sec 6]{NO} when $DT(X)$ is replaced by the moduli space $PT(X)$ of stable pairs.

\begin{definition} The {\em modified virtual structure sheaf} is $$\tilde{\mathcal{O}}^{\vir}=\mathcal{O}^{\vir}\otimes (\K^{\vir})^{1/2}.$$
\end{definition}

\subsection{K-theoretic Donaldson-Thomas invariants}

K-theoretic Donaldson-Thomas invariants are constructed as $\tilde{T}$-equivariant holomorphic Euler characteristics of $\tilde{\mathcal{O}}^{\vir}$. Define the  \emph{K-theoretic DT partition function} to be $$Z_{DT}(X) = \sum_{\substack{\beta\in H_2(X,\Z)^{\rm eff}\\ n\in \Z}}Q^{n}u^{\beta}\chi(DT(X,\beta,n),\tilde{\mathcal{O}}^{\vir})\in\Q(t_1,t_2,t_3,\kappa^{1/2})[[Q,Q^{-1}]][u^{\beta}].$$
Given some fixed $\beta$, the space $DT(X,\beta,n)$ is empty for $n\ll 0$.

This definition does not depend on the choice of square root $(\K^{\vir})^{1/2};$ we will justify this independence after introducing equivariant localization.

It is convenient to normalize the Donaldson-Thomas partition function by the contribution from the Hilbert scheme of points on $X$; to this end, we define the \emph{reduced K-theoretic DT partition function} to be $$Z_{DT}^{'}(X) =\frac{Z_{DT}(X)}{\displaystyle\sum_{n\in \Z}Q^{n}\chi(DT(X,0,n),\tilde{\mathcal{O}}^{\vir})}\in \Q(t_1,t_2,t_3,\kappa^{1/2})[[Q,Q^{-1}]][u^{\beta}].$$ 

We remark that, when $X$ is not Calabi-Yau, the general definition of $K$-theoretic DT invariants of \cite{NO} also takes as input the total space $\mathcal{L}_1\oplus \mathcal{L}_2$ of two line bundles over $X$ satisfying $\mathcal{L}_1\otimes \mathcal{L}_2\simeq \K_X$. In this setting, the definition of $\tilde{\mathcal{O}}^{\vir}$ incorporates a tautological bundle obtained from $\mathcal{L}_1$ and $\mathcal{L}_2$. For Calabi-Yau $X$ and trivial $\mathcal{L}_1, \mathcal{L}_2$, this more general construction specializes to our construction of $\tilde{\mathcal{O}}^{\vir}$ above; moreover, the boxcounting parameter $Q$ arises geometrically, as the torus weight of $\mathcal{L}_1$.

\subsection{The plethystic exponential}

Roughly speaking, the plethystic exponential takes as input the character of a virtual, possibly infinite-dimensional, representation $V$ of some group and, in the spirit of (\ref{lambsym}), outputs the character of $$\sum_{n\geq 0}\Sym^{n}V,$$ when this character is well-defined. To simplify the exposition, we give a precise definition for the two settings in which we will use the plethystic exponential and refer the reader to \cite[Sec. 3]{M} or \cite[Sec. 2.1]{O} for an expanded discussion. 

First, given formal variables $x_1,\ldots,x_m$ and finitely many Laurent monomials $u_i(x_1,\cdots,x_m)$ and  $v_j(x_1,\cdots,x_m)$ such that no $u_i$ is equal to $1$, we define \begin{align}\label{smallsym} \Sym^{\bullet}\Big(\sum_i u_i-\sum_j v_j\Big)=\frac{\prod_{j} 1-v_j}{\prod_{i} 1-u_i}\in \Q(x_1,\ldots,x_n).\end{align} For example, $$\Sym^{\bullet}(2x_1-x_1x_2^{-3})=\frac{1-x_1x^{-3}_2}{(1-x_1)^{2}},\ \ \ \ \ \ \Sym^{\bullet}(3x_1-1)=0.$$

Second, this definition can be generalized so that $\Sym^{\bullet}$ takes as arguments certain series whose coefficients are rational functions. The following class of examples will be sufficient for our purposes. Let $z, x_1,\cdots,x_m, t_1, t_2$ be formal variables and $w_1, w_2$ be monomials in $t_1, t_{2}$. Given Laurent polynomials $f_i(x_1,\cdots,x_m,t_1,t_2)$ for $i>0$, define \begin{align} \label{ratsym}&\Sym^{\bullet}\Big(\sum_{i\geq 1}z^i\frac{f_i(x_1,\cdots,x_m,t_1,t_2)}{(1-w_1)(1-w_2)}\Big)\nonumber \\&=\mathrm{exp}\Big(\sum_{n\geq 1} \frac{(z^i)^n}{n} \frac{f_i(x_1^n,\cdots,x_m^n,t_1^n,t_{2}^n)}{(1-w_1^{n})(1-w_2^{n})}\Big)\in \Q[x_1^{\pm},\cdots,x_m^{\pm}](t_1,t_2)[[z]].\end{align}

For example, the right hand side of (\ref{denom}) is equal to $$\Sym^{\bullet}\Big(-z\frac{(1-m_1)(1-m_2)}{(1-t_1^{-1})(1-t_2^{-1})}\Big).$$ For the purposes of stating (\ref{nek}) later, we mention that the definition (\ref{ratsym}) can be extended in a straightforward manner when there are arbitrarily many variables $t$ and the product $(1-w_1)(1-w_2)$ appearing in the denominator of the argument of $\Sym^{\bullet}$ is replaced by an arbitrary finite product of terms of the form $(1-w)$.

With some care, each of the definitions (\ref{smallsym}) and (\ref{ratsym})  can be formulated in the spirit of the other. As in \cite[Ex. 2.1.3]{O}, we explain how to pass from (\ref{ratsym}) to the form of (\ref{smallsym}). Pick a consistent direction in which to expand the rational functions $1/(1-w_1)$ and $1/(1-w_2)$ as power series in $t_1$ and $t_2$. Then, one may write \begin{align}\label{expansion}\sum_{i\geq 1}z^i\frac{f_i(x_1,\cdots,x_m,t_1,t_2)}{(1-w_1)(1-w_2)}=\sum_{i\geq 1}\Big(\sum_{j} z^{i}u_{i,j}-\sum_{k} z^{i}v_{i,k}\Big),\end{align} where the $u_{i,j}$ and $v_{i,k}$ are Laurent monomials in $x_1,\cdots,x_m$ and $t_1,t_2$. In contrast to the argument of (\ref{smallsym}), there may be infinitely many nonzero $u_{i,j}$ and $v_{i,k}$. Then, we have \begin{align}\label{compsym}\Sym^{\bullet}\Big(\sum_{i\geq 1}z^i\frac{f_i(x_1^n,\cdots,x_m^n,t_1^n,t_{2}^n)}{(1-w_1)(1-w_2)}\Big)=\prod_{i}\frac{\prod_{k} 1-z^iv_{i,k}}{\prod_{j} 1- z^i u_{i,j}},\end{align} where the denominator of the right hand side is expanded in positive powers of $z$. Matching up coefficients of powers of $z$, the equality (\ref{compsym}) can be regarded either as series of equalities of rational functions (to be precise, as an equality in $\Q[x_1^{\pm},\cdots,x_m^{\pm}](t_1,t_2)[[z]]$), or, expanding the denominators of the left-hand side as power series in $t_1,t_2$ in the same direction selected in (\ref{expansion}), as an equality of power series (that is, in $\Q[x_1^{\pm},\cdots,x_m^{\pm}][[t_1,t_2]][[z]]$.) It will be helpful to pass between these two interpretations for the calculation of Section \ref{eval}. In particular, when regarded as series with coefficients in rational functions, the right hand side of (\ref{compsym}) does not depend on the choice of direction for $t_1$ and $t_2$.

For example, expanding in positive powers of $t$, we have \begin{align*}\Sym^{\bullet} (\frac{z}{1-t})=\prod_{i\geq 0}\frac{1}{1-zt^{i}}=\prod_{i\geq 0}(1+zt^{i}+z^2t^{2i}+\cdots)=\sum_{j,k\geq 0} p(j,k)z^jt^k,\end{align*} where $p(j,k)$ is the number of two-dimensional partitions $\lambda$ of size $k$ and length at most $j$. Using the map $\lambda\mapsto \lambda^t$ one sees that $p(j,k)$ is also the number of two-dimensional partitions $\mu$ of size $k$ with $\lambda_1\leq j$. Thus $$\sum_{k\geq 0} p(j,k)t^k=\prod_{i=1}^{j} \frac{1}{1-t^i}.$$

  On the other hand, expanding in negative powers of $t$ we obtain \begin{align*} \Sym^{\bullet}(\frac{z}{1-t})=\Sym^{\bullet}(\sum_{i\geq 1} -zt^{-i})=\prod_{i\geq 1}(1-zt^{-i})=\sum_{j,k\geq 0} (-1)^jq(j,k)z^jt^{-k},\end{align*} where $q(j,k)$ is the number of partitions $\lambda$ of size $k$ and length $j$ such that $\lambda_1>\lambda_2>\cdots>\lambda_j$. Using the map $$\lambda\mapsto (\lambda_1-(j-1),\lambda_2-(j-2),\cdots,\lambda_j)^{t}$$ one sees that $q(j,k)$ also counts the number of partitions $\mu$ of size $k-j(j-1)/2$ with $\mu_1=j$. Thus $$\sum_{k\geq 0} (-1)^jq(j,k) t^{-k}=(-1)^jt^{-j(j-1)/2}\cdot t^{-j}\cdot\prod_{i=1}^{j}\frac{1}{1-t^{-j}}.$$ 
  
As rational functions in $t$, we have $$(-1)^jt^{-j(j-1)/2}\cdot t^{-j}\prod_{i=1}^{j}\frac{1}{1-t^{-j}}=\prod_{i=1}^{j} \frac{1}{1-t^i}.$$
So, whether expanding in $t$ or $t^{-1}$, we obtain $$\Sym^{\bullet} (\frac{z}{1-t})=\sum_{j\geq 1} \frac{z^j}{\prod_{i=1}^j (1-t^i)}=\mathrm{exp}(\sum_{n\geq 1} \frac{z^n}{1-t^n}).$$

\subsection{Equivariant localization}
We study K-theoretic invariants using virtual equivariant localization as developed in \cite{FG, GP}, \cite{CFK} and \cite{Q}; a concise exposition of the application we use can be found in \cite{Th2}. In this subsection, let $\mathscr{M}$ be a quasi-projective scheme equipped with the action of a torus $\mathbb{T}$ and a $\mathbb{T}$-equivariant perfect obstruction theory. Suppose that the fixed locus $\mathscr{M}^\mathbb{T}$ is compact (and nonempty). Let $\F$ be a $\mathbb{T}$-equivariant coherent sheaf on $\mathscr{M}$. Equivariant localization asserts first, that $$\chi(\mathscr{M},\mathcal{F}\otimes \mathcal{O}^{\vir})\in \Q[t_k^{\pm}]\bigg[\frac{1}{1-u_i}\bigg],$$ where the $u_i$ range over the set of $\mathbb{T}$-weights appearing in $N^{\vir}_{\mathscr{M}/\mathscr{M}^\mathbb{T}}.$ 

Moreover, it expresses $\chi(\mathscr{M},\mathcal{F}\otimes \mathcal{O}^{\vir})$ as an Euler characteristic over the fixed locus $\mathscr{M}^\mathbb{T}$. Let us further assume that the fixed locus $\mathscr{M}^\mathbb{T}$ consists of isolated points and that the restriction of the obstruction theory to $\mathscr{M}^{\mathbb{T}}$ has no $\mathbb{T}$-fixed part. These assumptions will hold for DT moduli space by \cite[Sec. 4.5]{MNOP1}. Under these assumptions, the formula reads  $$\chi(\mathscr{M},\F\otimes \mathcal{O}^{\vir})= \sum_{p\in \mathscr{M}^\mathbb{T}} \chi(p,\F|_{p} )\otimes \Sym^{\bullet}\big(\chi(p,(\mathcal{T}^{\vir}_{p})^{\vee})\big),$$ where $\Sym$ is as defined in (\ref{smallsym}). By assumption, none of the monomials appearing in the character of any $\mathcal{T}^{\vir}_p$ are trivial. So,  the operation $\Sym$ is well-defined. 

In our applications to Hilbert scheme of points on smooth surfaces in Section \ref{tautological}, we also use the more classical variant \cite[Thm 3.5]{T} of K-theoretic equivariant localization where $\mathscr{M}$ is a smooth scheme; in this case the ordinary structure sheaves and tangent bundles replace their virtual analogs.


Let us now specialize to the situation where $\mathscr{M}$ is a component of $DT(X),$ the torus $\mathbb{T}$ is the double cover $\tilde{T}$ of $T=\mathrm{diag}(t_1,t_2,t_3)$, and $\F$ is $(\K^{\vir})^{1/2}.$ The symmetry of the perfect obstruction theory of $DT(X)$ ensures that resulting localization expression enjoys a particularly nice form. Given $\mathcal{I}\in DT(X)^{\tilde{T}}$, we may write the $\tilde{T}$-character of $\mathcal{T}^{\vir}_{\mathcal{I}}$ as $${T}^{\vir}_{\mathcal{I}}=\Def_{{\mathcal{I}}}-\Obs_{{\mathcal{I}}}=\sum_{i} u_{i}-v_{i},$$ where each $u_i$ and $v_i$ is a $\tilde{T}$-weight; we also have $$(\mathcal{K}^{\vir})^{1/2}_{\mathcal{I}}= \prod_i \frac{v_i^{1/2}}{u_i^{1/2}}.$$ By (\ref{obssym}), we may reorder the $u_i$ and $v_i$ so that that $$v_i=\frac{\kappa}{u_i}.$$ By virtual localization, we have the following equality of $\tilde{T}$-characters: \begin{align}\label{dtpar} Z_{DT}(X)&=\sum_{\substack{\beta\in H_2(X,\Z)^{\rm eff} \\ n\in \Z}}Q^nu^{\beta} \chi\bigg(DT(X,\beta,n)^{{T}}, \big((\mathcal{K}^{\vir})^{1/2}\otimes \Sym^{\bullet} \mathcal{T}^{\vir}\big)\big|_{DT(X,\beta,n)^T} \bigg) \nonumber\\&=\sum_{\substack{\beta\in H_2(X,\Z)^{\rm eff} \\ n\in \Z}} Q^nu^{\beta} \sum_{\mathcal{I}\in DT(X,\beta,n)^{{T}}} \prod_{i} \frac{ (\frac{\kappa}{u_i})^{1/2}-(\frac{\kappa}{u_i})^{-1/2}  }{u_i^{1/2}-u_i^{-1/2}}.
\end{align}

Motivated by this expression, we adopt the following notation from \cite{O,Ot}: given a $T$-character $V=\sum_i u_i-\sum_j v_j$ where $u_i,v_j\in T^{\vee}$, we set $$\hat{a}(V)=\frac{\prod_{j} v_j^{1/2}-v_{j}^{-1/2}}{\prod_{i} u_i^{1/2}-u_{i}^{-1/2}} \in \Q(t_k^{\pm 1/2}).$$

\subsection{Independence of choice of $(\K^{\vir})^{1/2}$}
If $\mathcal{L}$ is a $\tilde{T}$-equivariant line bundle on $DT(X)$ with $\mathcal{L}^{\otimes 2}\cong \K^{\vir}$ and $\mathcal{I}\in DT(X)$ is a torus fixed sheaf, then the $\tilde{T}$-weight of the fiber $\K^{\vir}|_{\mathcal{I}}$ is the square of the $\tilde{T}$-weight of the fiber $\mathcal{L}|_{\mathcal{I}}$. It follows from (\ref{dtpar}) that $Z_{DT}(X)$ does not depend on the choice of square root $(\K^{\vir})^{1/2}$. We remark that this is a special instance of an observation of Thomas (see \cite[Prop 2.6]{Th2}), who exhibits a canonical choice of square root of the restriction of $\mathcal{K}^{\vir}$ to the fixed locus of a quasi-projective scheme equipped with a torus action, a torus-equivariant symmetric perfect obstruction theory, and proper torus-fixed locus. 

As any two square roots of $\K^{\vir}$ differ by 2-torsion in $\Pic(DT(X)),$ the independence of $Z_{DT}(X)$ on the choice of square root can also be deduced from equivariant localization and virtual Hirzebruch-Riemann-Roch (\cite[Cor 3.4]{FG}).

\section{Slope independence}\label{secindependence}

For this section, we let $\mathscr{M}$ denote a quasi-projective scheme equipped with an action of a torus $\mathbb{T}$ and a $\mathbb{T}$-equivariant perfect obstruction theory. Examples include smooth schemes, DT moduli space and the Pandharipande-Thomas stable pair moduli space of a threefold (see \cite{PT1}). We impose the additional assumption that $\mathscr{M}^\mathbb{T}$ is proper and nonempty. We let $\F$ denote a $\mathbb{T}$-equivariant K-theory class of coherent sheaves on $\mathscr{M}$.

One assumption used in the analysis of K-theoretic enumerative invariants of Calabi-Yau threefolds in \cite[Sec. 7]{NO} is that the underlying moduli space has proper components. If $\mathscr{M}$ is proper, then the $\mathbb{T}$-character $\chi(\mathscr{M},\F\otimes \mathcal{O}^{vir})$ is a Laurent polynomial in $\mathbb{T}$. Let $w\in \mathbb{T}^{\vee}$ be a primitive weight. As $\chi(\mathscr{M},\F \otimes \mathcal{O}^{vir})$ is a Laurent polynomial, it follows that if the limit $\chi(\mathscr{M},\F  \otimes \mathcal{O}^{vir})^{\sigma}$ exists for any generic one-parameter subgroup $\sigma:\C^{\times}\to \ker(w),$ then $\chi(\mathscr{M},\F\otimes  \mathcal{O}^{vir})$ is a function of $w$ and, moreover, that \begin{align}\label{properind}\chi(\mathscr{M},\F \otimes \mathcal{O}^{vir})=\chi(\mathscr{M},\F \otimes \mathcal{O}^{\vir})^{\sigma}\in \Q[w]\end{align} for generic $\sigma$.

For example, suppose that $\mathscr{M}=\P^1\times \P^1$ (regarded as a smooth scheme) and that $\mathbb{T}=\C^{\times}_{t_1}\times \C^{\times}_{t_2}$ where each factor acts in the standard way on the corresponding factor of $\P^1\times \P^1$. Then, let $\mathcal{F}$ be the cotangent bundle $\mathcal{T}^*(\P^1\times \P^1)$ with the induced torus action. By equivariant localization, \begin{align}\label{P1P1ex}\chi(\P^1\times \P^1,\mathcal{T}^*(\P^1\times \P^1) )=&\frac{t_1^{-1}+t_2^{-1}}{(1-t_1^{-1})(1-t_2^{-1})}+\frac{t_1^{-1}+t_2}{(1-t_1^{-1})(1-t_2)} \nonumber\\&+\frac{t_1+t_2^{-1}}{(1-t_1)(1-t_2^{-1})}+\frac{t_1+t_2}{(1-t_1)(1-t_2)}.\end{align} While the right hand side of (\ref{P1P1ex}) can be computed by putting the rational functions over a common denominator, it can also be efficiently computed using (\ref{properind}). Namely, for any nonzero $r_1$ and $r_2$, among the four limits \begin{align*}&\lim_{z\to 0}\frac{z^{-r_1}t_1^{-1}+z^{-r_2}t_2^{-1}}{(1-z^{-r_1}t_1^{-1})(1-z^{-r_2}t_2^{-1})}, \lim_{z\to 0}\frac{z^{-r_1}t_1^{-1}+z^{r_2}t_2}{(1-z^{-r_1}t_1^{-1})(1-z^{r_2}t_2)}\\ &\lim_{z\to 0} \frac{z^{r_1}t_1+z^{-r_2}t_2^{-1}}{(1-z^{r_1}t_1)(1-z^{-r_2}t_2^{-1})}, \lim_{z\to 0} \frac{z^{r_1}t_1+z^{r_2}t_2}{(1-z^{r_1}t_1)(1-z^{r_2}t_2)},\end{align*} exactly two are equal to $-1$, and exactly two are equal to $0$. So, for any one-parameter subgroup $\sigma(r_1,r_2):\C^{\times}\to \ker(1)=\mathbb{T}$ given by $z\mapsto(z^{r_1},z^{r_2})$ with $r_1$ and $r_2$ nonzero, one has $$\chi(\P^1\times \P^1,\mathcal{T}^*(\P^1\times \P^1))^{\sigma(r_1,r_2)}=-2.$$ By (\ref{properind}), we conclude that $$\chi(\P^1\times \P^1,\mathcal{T}^*(\P^1\times \P^1))=-2$$ as a $\mathbb{T}$-equivariant Euler characteristic. In particular, the Euler characteristic exhibits no dependence on $t_1$ and $t_2$.

Our interest is in Euler characteristics over DT moduli spaces, whose components are not generally proper. Our aim in this section is to formulate Proposition \ref{genslopeind}, a statement in the spirit of (\ref{properind}) that holds for non-proper moduli spaces $\mathscr{M}$.

The approach is as follows. Using localization, we characterize the possible denominators of rational functions $\chi(\mathscr{M},\F\otimes\mathcal{O}^{\vir})$ in terms of the equivariant geometry of $\mathscr{M}$.

We then show that if the limit $\chi(\mathscr{M},\F\otimes\mathcal{O}^{\vir})^{\sigma}$ exists for all one-parameter subgroups $\sigma$ landing in a fixed subtorus of $\mathbb{T}$, then the value of this limit depends only on the attracting behavior, with respect to $\sigma$, of the weights appearing in the denominator of $\chi(\mathscr{M},\F\otimes\mathcal{O}^{\vir})$. 

When reading this section, it may be useful to keep in mind the eventual application. Namely, we will apply the results when $\mathscr{M}$ is a component $DT(X,\beta,n)$ of DT moduli space, the torus $\mathbb{T}$ is $\tilde{T}$, the double cover of the torus $T=\diag(t_1,t_2,t_3)$ acting on $X$ for which the square root of the character of $\mathcal{K}^{\vir}$ is defined, and the sheaf $\mathcal{F}$ is $(\mathcal{K}^{\vir})^{1/2}.$  As in Sections \ref{DTdef} and \ref{virtstruct}, set $\kappa=t_1t_2t_3$. In Section \ref{threefold}, we will be interested in limits $$\chi(DT(X,\beta,n),\F\otimes\mathcal{O}^{\vir})^{\sigma}$$ for one-parameter subgroups $\sigma:\C^{\times}\to \ker(\kappa)$.

\subsection{Denominators in localization}


Given a weight $w\in \mathbb{T}^{\vee}$, we let $\mathbb{T}_{w}\subset \mathbb{T}$ denote the maximal torus inside the subgroup $\ker(w)$ of $\mathbb{T}$. In particular, for nonzero $n\in \Z$ we have $\mathbb{T}_{w^n}=\mathbb{T}_{w}$.

\begin{definition}
A weight $w\in \mathbb{T}^{\vee}$ is said to be a \emph{compact} weight of $\mathscr{M}$ if the fixed locus $\mathscr{M}^{\mathbb{T}_{w}}$ is proper, and \emph{noncompact} otherwise.
\end{definition}

While many possible denominators can occur in any particular term of a localization expression for $\chi(\mathscr{M},\F\otimes\mathcal{O}^{\vir})$, the following result restricts the possible denominators that may appear in the total sum.

\begin{proposition}\label{denomchar}
The $\mathbb{T}$-character $\chi(\mathscr{M},\mathcal{F}\otimes\mathcal{O}^{\vir})$ can be written as a quotient of two Laurent polynomials $p(t)/q(t)$ such that $q(t)$ is of the form $\prod (1-w),$ where each $w$ is a noncompact weight of $\mathscr{M}$.
\end{proposition}
\begin{proof}
Write $\chi(\mathscr{M},\mathcal{F}\otimes\mathcal{O}^{\vir})$ as a quotient $f(t)=p(t)/q(t)$, where $q(t)=\prod (1-w)$ and each $w$ is a weight in $\mathcal{N}^{\vir}_{\mathscr{M}/\mathscr{M}^\mathbb{T}};$ if $p(t)$ is zero there is nothing to show. 

If $w$ is a compact weight, then equivariant localization with respect to $\mathbb{T}_{w}$ implies that as $\mathbb{T}_{w}$-characters, we have $$\chi(\mathscr{M},\F\otimes\mathcal{O}^{\vir}) \in \Q\Big[t_i^{\pm},\frac{1}{1-v}\Big],$$ where $v$ ranges over the $\mathbb{T}$-weights occuring in $\mathcal{N}^{\vir}_{\mathscr{M}/
\mathscr{M}^{\mathbb{T}_{w}}}$. No such $v$ vanishes on $\mathbb{T}_w$, so we conclude that $f(t)$ may be written in a form whose denominator is a product of terms of the form $(1-v)$ where no $v$ is a power of $w$. In particular, the rational function $f(t)$ has no poles at $1=\zeta w$ for any root of unity $\zeta$.
\end{proof}

For example, consider the Hilbert scheme $DT(\C^3,0,n)$ of $n$ points on $\C^3$ with the standard action of a three-dimensional torus $\mathrm{diag}(t_1,t_2,t_3)$. Then the fixed locus $${DT(\C^3,0,n)}^{\mathrm{diag}(t_1,t_2,t_3)_{w}}$$ is proper unless $w$ is a nontrivial power of some $t_k$, so the only noncompact directions are of the form $t_k^{i}$. Nekrasov's formula (\cite[Thm 3.3.6]{O}) furnishes an explicit formula for the K-theoretic DT partition function: \begin{align}\label{nek} \sum_{n=0}^{\infty} &(-Q)^n\chi(DT(\C^3,0,n),\tilde{\mathcal{O}}^{\vir})=\Sym^{\bullet} \bigg(\frac{-Q}{(1-Q\kappa^{1/2})(1-Q\kappa^{-1/2})}\prod_{k=1}^3 \frac{\kappa^{1/2}t_k^{-1}-\kappa^{-1/2}}{1-t_k^{-1}}\bigg)\end{align} (here, the argument of the right-hand side should be expanded in positive powers of $Q$). In particular, note that $\chi(DT(\C^3,0,n),\tilde{\mathcal{O}}^{\vir})$ can be written as a Laurent polynomial whose denominator is a product of terms of the form $(1-t_k^{m})$.

\subsection{Limit independence}\label{limind}

Let  $A\subset \mathbb{T}$ be a subtorus, and let $f(t) \in \C(t_k)$ be a rational function that can be written in the form $$\frac{p(t)}{\prod_i (1-w_i)}, $$ where $p(t)$ is a Laurent polynomial in $t_k$ and $w_i\in \mathbb{T}^{\vee}$ are nontrivial weights. Set $W=\{w_i\}$. Via the map $\mathbb{T}^{\vee}\to A^{\vee},$ the elements of $W$ can be regarded as $A$-weights.

We say that the rational function $f(t)$ is \emph{$A$-balanced} if $$\lim_{z\to 0} f(\sigma(z)t)$$ exists for all one-parameter subgroups $\sigma:\C^{\times}\to A$, or, equivalently (for nonzero $f$), if the set of $A$-weights appearing in the numerator of $f$ is the same as the set of $A$-weights appearing in the denominator of $f$. 

Given an $A$-balanced rational function $f$, the weights in $W$ impose a wall and chamber structure on the lattice of one-parameter subgroups $\C^{\times}\to A$. We show that the limit of $f$ under a one-parameter subgroup $\sigma:\C^{\times}\to A$ depends only on the chamber containing $\sigma$. 

To simplify the exposition, we restrict our focus to one-parameter subgroups lying in the interior of such a chamber: we say that a one-parameter subgroup $\sigma:\C^{\times}\to A$ is \emph{$f$-generic} if, for all $w_i\in W$, one has $w_i(\sigma(z)t)\not=w_i(t)$.  

Given an $f$-generic $\sigma$ and $w_i\in W$ we have $w_i(\sigma(z)t)=z^rw_i(t)$ for some $r\not=0$. We say that $w$ is \emph{attracting} with respect to $\sigma$ if $r>0$ and is \emph{repelling} with respect to $\sigma$ if $r<0$.

\begin{proposition}\label{genindependence}
Let $$f(t)=\frac{p(t)}{\prod_i (1-w_i)}$$ be $A$-balanced and let $\sigma_1$ and $\sigma_2$ be two $f$-generic one-parameter subgroups sharing the same attracting/repelling behavior for each weight $w_i$. Then $$\lim_{z\to 0} f(\sigma_1(z)t)=\lim_{z\to 0} f(\sigma_2(z)t).$$
\end{proposition}
\begin{proof}
Multiplying the numerator and denominator of $f$ by a Laurent monomial if necessary, we may assume that each $w_i$ is attracting with respect to both $\sigma_1$ and $\sigma_2$. Set $$q(t)=\prod_i (1-w_i);$$
then we may write $q(\sigma_1(z)t)= 1+ zq_1(t,z)$ where $q_1\in \C[t_k^{\pm},z].$ 

Via the map $\phi:\mathbb{T}^{\vee}\to A^{\vee}$, we regard the monomials appearing in the numerator and denominator of $f$ as $A$-weights. Write $p(t)=\sum c_ju_j$ where each $u_j$ is a monomial, and set $$p_0(t) = \Big\{\sum c_ju_j \ \big|\  \phi(u_j)=1\in A^{\vee}\Big \}.$$ 

The set of $A$-weights appearing in the numerator of $f$ is the same as the set of $A$-weights appearing in the denominator of $f$. Therefore we may write $$p(\sigma_1(z)t)=p_0(t)+zp_1(t,z)$$ where $p_1\in \C[t_k^{\pm},z]$ and $p_0$ is regarded as an element of $\C[t_k^{\pm},z]$ via the inclusion $\C[t_{k}^{\pm}]\hookrightarrow \C[t_k^{\pm},z]$.  It follows that $$\lim_{z\to 0} f(\sigma_1(z)t)=\frac{p_0(t)+zp_1(t,z)}{1+zq_1(t)}=p_0(t);$$ the same is true of $\sigma_2$.
\end{proof}

We remark that analogous results hold for one-parameter subgroups $\sigma$ lying on a wall; that is, for $\sigma$ such that $w_i(\sigma(z)t)=w_i(t)$ for some $i$.

\subsection{Limit independence for Euler characteristics}\label{limind2}

Again, let $A\subset \mathbb{T}$ be a subtorus. We say that a one-parameter subgroup $\sigma:\C^{\times}\to A$ is {\em $\mathscr{M}$-generic} if for any noncompact weight $w$ of $\mathscr{M}$, we have $w(\sigma(z)t)\not=w(t)$. Combining Propositions \ref{denomchar} and \ref{genindependence}, we obtain the following.

\begin{proposition}\label{genslopeind}
If $\chi(\mathscr{M},\F\otimes\mathcal{O}^{\vir})$ is $A$-balanced and $\sigma_1$ and $\sigma_2$ are two $\mathscr{M}$-generic one-parameter subgroups $\C^{\times}\to A$ with the same attracting/repelling behavior for each non-compact weight of $\mathscr{M}$, then $$\chi(\mathscr{M},\F\otimes\mathcal{O}^{\vir})^{\sigma_1}=\chi(\mathscr{M},\F\otimes\mathcal{O}^{\vir})^{\sigma_2}.$$
\end{proposition}
 
Now, specialize $\mathscr{M}$ to be a component of DT moduli space, the torus $\mathbb{T}$ to be the double cover $\tilde{T}$ of the torus $T=\diag(t_1,t_2,t_3)$ as in the previous section, the sheaf $\mathcal{F}$ to be $(\mathcal{K}^{\vir})^{1/2}$ and $A$ to be the Calabi-Yau torus $\ker(\kappa),$ where $\kappa$ is the character of $(\K^{\vir})^{\vee}$. By (\ref{dtpar}), the character $\chi(DT(X,\beta,n),\tilde{\mathcal{O}}^{\vir})$ is $A$-balanced. We obtain Theorem \ref{independence}.

For example, let us explicitly verify the proposition for $DT(\C^3,0,1)$. Then \begin{align}\label{C3exp}\chi(DT(\C^3,0,1),\mathcal{O}^{\vir}\otimes( \mathcal{K}^{\vir})^{1/2} )=\prod_{k=1}^3 \frac{\kappa^{1/2}t_k^{-1}-\kappa^{-1/2}}{1-t_k^{-1}},\end{align} which is $A$-balanced. The non-compact weights of $DT(\C^3,0,1)$ are of the form $t^i_k$. So, a one-parameter subgroup $\sigma(r_1,r_2,r_3):\C^{\times}\to A$ given by $z\mapsto (z^{r_1},z^{r_2},z^{r_3})$ with $$r_1+r_2+r_3=0$$  is $\mathscr{M}$-generic if each of $r_1, r_2$ and $r_3$ are nonzero. Moreover, two one-parameter subgroups $\sigma(r_1,r_2,r_3)$ and $\sigma(s_1,s_2,s_3)$ share the same attracting/repelling behavior for each non-compact weight of $\mathscr{M}$ if the signs of $r_1,r_2$ and $r_3$ are the same as the signs of $s_1, s_2$ and $s_3$, respectively. 

By (\ref{C3exp}), we compute \begin{align}\chi(DT(\C^3,0,1),(\mathcal{K}^{\vir})^{1/2}\otimes \mathcal{O}^{\vir})^{\sigma(r_1,r_2,r_3)}&=\lim_{z\to 0}\Big(\prod_{k=1}^3 \frac{z^{-r_k}\kappa^{1/2}t_k^{-1}-\kappa^{-1/2}}{1-z^{-r_k}t_k^{-1}}\Big)\nonumber\\&=\begin{cases} -\kappa^{1/2}\ \text{if 2 of the}\ r_k\ \text{are positive} \\  -\kappa^{-1/2}\ \text{if 2 of the}\ r_k\ \text{are negative}. \label{indexex}\end{cases} \end{align}

So, the limit under $\sigma(r_1,r_2,r_3)$ only depends on the signs of the $r_k$.

\section{Preferred limits of the DT partition function}\label{secbuildingblocks}

Let $X$ be a smooth toric Calabi-Yau threefold. 

For this section, we let $T$ denote the torus ${\rm diag}(t_1,t_2,t_3)$ acting on $X$. Choose this action such that the anticanonical bundle $\mathcal{K}_X^{\vee}$ is scaled with weight $\kappa=t_1t_2t_3$.

We begin by recalling from \cite{MNOP1,NO,Ot} a procedure for writing the partition function
$Z_{DT}(X)$ and its reduced analog $Z'_{DT}(X)$ in terms of combinatorial expressions associated to the vertices and edges of the toric diagram $\Delta(X)$.

 Using results of \cite[Sec 8]{NO}, we explicitly compute the limits of these vertex and edge contributions under preferred slopes, and, as a result, deduce the precise relationship between the vertex contribution to $Z_{DT}(X)$ and the refined topological vertex of \cite{IKV}.


\subsection{K-theoretic building blocks}

\subsubsection{Torus-fixed locus}

The fixed locus $DT(X)^{T}$ consists of isolated points, with finitely many in any component $DT(X,\beta,n).$ We recall from \cite[4.1-4.2]{MNOP1} a combinatorial description of this fixed-point locus. 

The points of $X^T$ correspond to the vertices of the toric polytope $\Delta(X)$ and the $T$-fixed rational curves of $X$ correspond to the bounded edges of $\Delta(X)$. For $x_i\in X^T$, let $U_i$ be the toric chart centered at $x_i$.  Suppose that $x_i,x_j\in X^T$ are connected by a $T$-fixed rational curve $C_{ij}\subset X$; then $U_i\cap U_j$ is the total space of the direct sum of two line bundles $\mathcal{O}(l_{ij})\oplus \mathcal{O}(l'_{ij})$ over $C_{ij}.$ 

Choose coordinates on each $U_i$ such that the torus $T$ scales the coordinate directions by $T$-weights  $w_{i_1},w_{i_2},w_{i_3}$;  after fixing an orientation on the 1-skeleton of $\Delta(X)$, these coordinates can be compatibly ordered at each fixed point. The edges in the toric polytope emanating from the fixed point $x_i$ correspond to coordinate axes. Let $z_{i_1},z_{i_2},z_{i_3}$ denote the corresponding coordinate functions on $U_i$ so that $U_i={\rm Spec}\ \C[z_{i_1},z_{i_2},z_{i_3}]$. The torus $T$ scales these coordinate functions by $w_{i_1}^{-1},w_{i_2}^{-1},w_{i_3}^{-1}$, respectively. After cyclically permuting coordinates and exchanging $l_{ij}$ with $l'_{ij}$, if necessary, the weights of coordinate directions on $U_i, U_j$ are identified by \begin{align}\label{weightshift} w_{j_1}=w_{i_1}^{-1}, w_{j_2}=w_{i_1}^{-l'_{ij}}w_{i_3}, w_{j_3}=w_{i_1}^{-l_{ij}}w_{i_2};\end{align} see Figure \ref{genbun}. The convention is chosen so that the first coordinate corresponds to the edge of the polytope connecting the two vertices, and the order of the remaining coordinates is determined by the orientation on $\Delta(X)$. 

\begin{figure}[htb]
\centering{
\scalebox{.5}{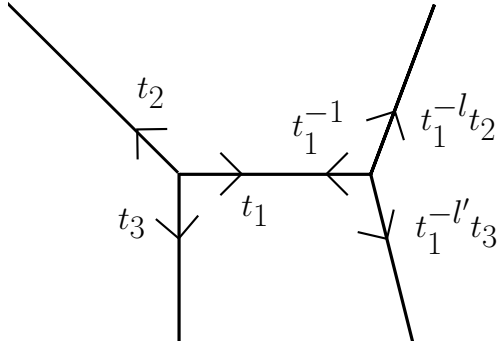}
\caption{The $T$-weights along a $T$-fixed rational curve with normal bundle $\mathcal{O}(\l)\oplus \mathcal{O}(l')$ } \label{genbun} 
}
\end{figure}

A $T$-fixed ideal sheaf $\mathcal{I}\in DT(X)$ is determined by a collection of $T$-fixed ideals $\{\mathcal{I}_i\subset \mathcal{O}_{U_i}\}$ whose restrictions $\mathcal{I}_i|_{U_{i}\cap U_j}$ are compatible, and glue to an ideal sheaf that cuts out a proper subscheme of $X$.

A $T$-fixed ideal $$\mathcal{I}_{i}\subset U_i$$ is given by a monomial ideal $$I_i\subset \C[z_{i_1},z_{i_2},z_{i_3}].$$ There is a one-to-one correspondence between monomial ideals and three-dimensional partitions $\pi_i$ given by $$I_i\leftrightarrow \pi_i=\{(b_1,b_2,b_3) | z_{i_1}^{b_1}z_{i_2}^{b_2}z_{i_3}^{b_3}\not\in I_i\};$$  we pass freely between an ideal and its associated partition.

The condition that $\mathcal{I}$ cuts out a proper subscheme restricts the possible asymptotics of the three-dimensional partitions $\pi_i$. The asymptotics of each such $\pi_i$ along each coordinate direction must be finite (as described in Section \ref{part}); more restrictively, the partition $\pi_i$ at a given fixed point $x_i$ must satisfy the following condition: \begin{align}\label{comp0} \pi_i^{(k)}\not=\emptyset\ {\rm only\ if\ the\ edge\ of}\ \Delta(X)\ {\rm in\ the\ direction\ of\ coordinate}\ z_{i_k}\ {\rm is\ bounded.} \end{align}

The asymptotics of the partitions must also agree in the direction of any $T$-fixed rational curve. To be precise, given $x_i, x_j\in X^T$ connected by a $T$-fixed rational curve $C_{ij}$, the restrictions $\mathcal{I}_i|_{U_{i}\cap U_j}$ and $\mathcal{I}_j|_{U_i\cap U_j}$ are compatible if the two localizations $$(I_i)_{z_{i_1}}\subset \C[z_{i_1}^{\pm 1},z_{i_2},z_{i_3}],\ (I_j)_{z_{i_1}^{-1}} \subset \C[z_{i_1}^{\pm 1}, z_{i_2}, z_{i_3}]$$ coincide. In terms of three-dimensional partitions, this compatibility translates to the condition that \begin{align}\label{comp1}\pi_{i}^{(1)}=\big(\pi_j^{(1)}\big)^t;\end{align} 
for example, in Figure \ref{comppair}, if the horizontal directions of $\pi_i$ and $\pi_j$ correspond to the coordinates $z_{i_1}$ and $z_{j_1}$, and the coordinates are oriented counter-clockwise,  we have ${\pi_{i}^{(1)}=(3,2)}$ and $\pi_{j}^{(1)}=(2,2,1)$.

 \begin{figure}[htb]
\centering{
\scalebox{.5}{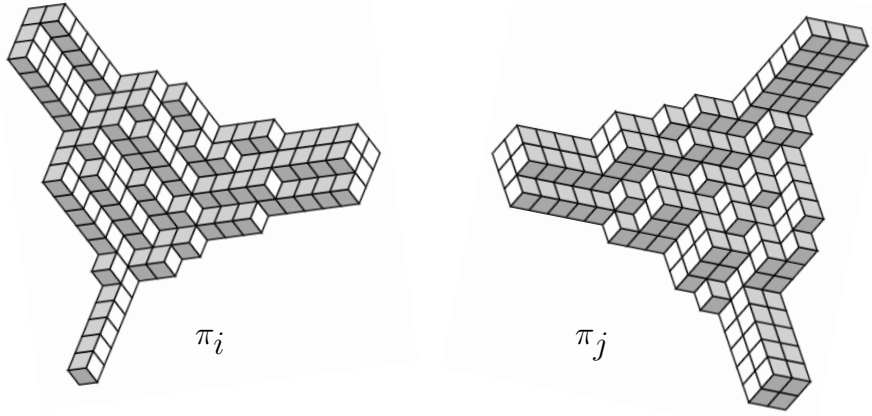}
\caption{Two compatible partitions $\pi_i$ and $\pi_j$ at vertices connected by an edge of $\Delta(X)$} \label{comppair}
}
\end{figure}

Conversely, any collection of three-dimensional partitions $\pi_i$ satisfying (\ref{comp0}) and (\ref{comp1}) corresponds to a $T$-fixed ideal sheaf $\mathcal{I}$. It follows that a point of $DT(X)^T$ is specified by a collection $\{\lambda_{ij}, \pi_i\},$ consisting of first, a choice of two-dimensional partition $\lambda_{ij}$ for each bounded edge of $\Delta(X)$, and then, a choice of three-dimensional partition $\pi_i$ for each vertex of $\Delta(X)$ whose asymptotics coincide with the chosen $\lambda_{ij}$; by convention we set $$\lambda_{ji}=\lambda_{ij}^t.$$

\subsubsection{Obstruction theory at fixed points}

Recall from \cite[Sec. 4.6]{MNOP1} that the virtual tangent space to $DT(X)$ at an ideal sheaf $\mathcal{I}$ is given by $${\mathcal{T}}^{\vir}_{{\mathcal{I}}}=\Ext^1(\mathcal{I},\mathcal{I})-\Ext^2(\mathcal{I},\mathcal{I}).$$

Recall that given a character $V$, the character of the dual representation is denoted $V^{\vee}$. In other words, $V^{\vee}$ is the character obtained from $V$ by inverting all torus variables. For example $$(2t_1^2+t^{-1}_1t_2)^{\vee}=2t_1^{-2}+t_1t_2^{-1}.$$

As $X$ is Calabi-Yau, the obstruction theory on $DT(X)$ is symmetric, so that \begin{align}\label{symobs} {T}^{\vir}_{{\mathcal{I}}}=-(T^{\vir}_{{\mathcal{I}}})^{\vee}\cdot \kappa.
\end{align} as $T$ characters.

The obstruction theory at a torus-fixed point ${\mathcal{I}}\in DT(X)^T$ can be calculated from the corresponding combinatorial data $\{\lambda_{ij},\pi_i\}.$ We recall this procedure from \cite[Sec. 4.7-4.9]{MNOP1}; see also \cite[Sec. 8.2]{NO} and \cite[Sec. 6.3]{Ot}.


Now suppose that $T={\rm diag}(t_1,t_2,t_3)$ acts on $\C^3={\rm Spec}\ \C[z_1,z_2,z_3]$ so $T$ scales the coordinate $z_k$ by $t_k^{-1}$. Given a partition $\pi$ corresponding to a $T$-fixed ideal sheaf $\mathcal{I}\subset \mathcal{O}_{\C^3}$, we set $O_{\pi}$ to be the $T$-character $\chi(\mathcal{O}_{\C^3}/\mathcal{I})$, that is $$O_{\pi} = \sum_{\square\in \pi} t_1^{-b_1}t_2^{-b_2}t_3^{-b_3}.$$ 

Then for $\pi$ with finitely many boxes we have \begin{align}\label{tangentterm}T^{\vir}_{\pi}(t_1,t_2,t_3)=O_{\pi}-{O^{\vee}_{\pi}}t_1t_2t_3-O_{\pi}O^{\vee}_{\pi}(1-t_1)(1-t_2)(1-t_3);
\end{align}
this is a Laurent polynomial in the variables $t_k$. For $\pi$ with infinitely many boxes, we first write the infinite sum $O_{\pi}$ as a rational function, and define the tangent space $T^{\vir}_{\pi}$ to be the \emph{rational function} in $t_k$ given by the same formula (\ref{tangentterm}). 

Then, for a $T$-fixed ideal sheaf $\mathcal{I}\in DT(X)^T$ given by $\{\lambda_{ij},\pi_i\}$ with weights $w_{i_1},w_{i_2},w_{i_3}$ as above, equivariant localization implies
\begin{align}\label{tanloc}T_{{\mathcal{I}}}^{\vir} = \sum_{i} T^{\vir}_{\pi_i}(w_{i_1},w_{i_2},w_{i_3}).
\end{align}

While each term on the right hand side of (\ref{tanloc}) is a rational function, the total sum is a Laurent polynomial in $t_k$. It is explained in \cite[4.9]{MNOP1} (see also \cite[8.2.2-8.2.3]{NO}, \cite[6.3.1]{Ot}) how to write $T_{{\mathcal{I}}}^{\vir}$ as a sum of Laurent polynomials by regularizing each term of (\ref{tanloc}) by contributions associated to the asymptotics of the corresponding partition $\pi_i$. Recalling (\ref{tanchar}), set \begin{align}\label{edgesub} T_{\lambda_{ij}}(w_{i_{k+1}},w_{i_{k+2}})=\sum_{\square\in\lambda_{ij}} w_{i_{k+1}}^{-l(\square)}w_{i_{k+2}}^{a(\square)+1}+w_{i_{k+1}}^{l(\square)+1}w_{i_{k+2}}^{-a(\square)}.\end{align} Then, if the $T$-fixed rational curve $C_{ij}$ corresponds to coordinate direction $z_{i_k}$, define \begin{align}\label{edgeterm} E_{\lambda_{ij},(l_{ij},l'_{ij})}(w_{i_1},w_{i_2},w_{i_3}) = \frac{T_{\lambda_{ij}}(w_{i_{k+1}},w_{i_{k+2}})}{1-w_{i_k}^{-1}}+\frac{T_{\lambda^t_{ij}}(w_{i_k}^{-l'_{ij}}\cdot w_{i_{k+2}},w_{i_k}^{-l_{ij}} \cdot w_{i_{k+1}})}{1-w_{i_k}};
\end{align}
here, for $k=1,2,3,$ we set $w_{i_{k+3}}=w_{i_{k}}.$ Using \ref{edgesub}, one can write the expression (\ref{edgeterm}) in terms of arm and leg lengths: 

When clear from the geometry of $X$, we omit the pair $(l_{ij},l'_{ij})$ from the notation for $E$.

We then define \begin{align}\label{vertexterm} V_{\pi_i}(w_{i_1},w_{i_2},w_{i_3}) = T^{\vir}_{\pi_i}(w_{i_1},w_{i_2},w_{i_3}) - \sum_{k=1}^3 \frac{T_{\pi_i^{(k)}}(w_{i_{k+1}},w_{i_{k+2}})}{1-w_{i_k}^{-1}}. \end{align} 
By \cite[Lemma 9]{MNOP1}, each $V_{\pi_i}$ and $E_{\lambda_{ij}}$ is a Laurent polynomial in $t_i$. From (\ref{tanloc}), we have \begin{align}\label{laurpol} T_{{\mathcal{I}}}^{\vir} = \sum_{i} V_{\pi_i}(w_{i_1},w_{i_2},w_{i_3})+\sum_{ij} E_{\lambda_{ij},(l_{ij},l'_{ij})}(w_{i_1},w_{i_2},w_{i_3}).
\end{align}
It follows from (\ref{tangentterm}, \ref{edgeterm}, \ref{vertexterm}) that each term of (\ref{laurpol}) enjoys the symmetry (\ref{symobs}); that is \begin{align}\label{termsym}
V_{\pi_i}(t_1,t_2,t_3)= -V_{\pi_i}(t_1,t_2,t_3)^{\vee}\cdot \kappa, \ E_{\lambda_{ij}}(t_1,t_2,t_3)=-E_{\lambda_{ij}}(t_1,t_2,t_3)^{\vee}\cdot \kappa. 
\end{align}

\subsubsection{Euler characteristic}
In the same manner as the obstruction theory, the value of $\chi(\mathcal{O}_Y)$ for $T$-fixed curves $Y$ can be computed in terms of edge and regularized vertex contributions using the \v{C}ech cover $\{U_i\}$; we recall the description from \cite[Lemma 5]{MNOP1}. 

 Set 
\begin{align}\label{vertexeul1} \chi(\pi)=\sum_{\square\in \pi} \bigg(1-\ \Big({\rm number\ of\ legs }\ \{ \pi\langle 1\rangle,  \pi\langle 2\rangle, \pi\langle 3 \rangle\}\ {\rm that\ contain}\ \square\Big)\bigg); \end{align}
this is the renormalized volume of $\pi$ as defined in \cite[4.4]{MNOP1}. Only finitely many boxes of $\pi$ contribute a nonzero value to the sum (\ref{vertexeul1}).

For a two-dimensional partition $\lambda$ and integers $(l,l')$, set \begin{align}\label{edgeeul} \chi(\lambda,(l,l'))= \sum_{\square\in\lambda} 1-lb_1-l'b_2.
\end{align}
Then, for ${\mathcal{I}}\in DT(X)^T$ specified by partition data $\{\lambda_{ij},\pi_i\}$ we have \begin{align}\label{toteul} \chi(\mathcal{O}/\mathcal{I})=\sum_{i} \chi(\pi_i) + \sum_{i,j} \chi(\lambda_{ij},(l_{ij},l'_{ij})).\end{align}

\subsection{Factorization of components of the partition function}\label{vertexsection}
We rewrite (\ref{dtpar}) as \begin{align}\label{dtpar2} Z_{DT}(X)=\sum_{{\mathcal{I}}\in DT(X)^T}  Q^{\chi(\mathcal{O}/\mathcal{I})}u^{[\mathcal{O}/\mathcal{I}]}\hat{a}(T_{{\mathcal{I}}}^{\vir}).\end{align} Recall that $[\mathcal{O}/\mathcal{I}]\in H_2(X,\Z)^{\mathrm{eff}}$ is the class of the one-dimensional components of $\mathcal{O}/\mathcal{I}$ weighted by their multiplicities.

We now write $Z_{DT}(X)$ in terms of expressions associated to the vertices and edges of $\Delta(X)$. Let ${\mathcal{I}}\in DT(X)^T$ be given by the configuration $\{\lambda_{ij},\pi_i\}$. Then, by (\ref{laurpol}), we have

$$\hat{a}(T_{{\mathcal{I}}}^{\vir})=\prod_{i} \hat{a}\Big(V_{\pi_i}(w_{i_1},w_{i_2},w_{i_3})\Big)\prod_{ij} \hat{a}\Big(E_{\lambda_{ij},(l_{ij},l'_{ij})}(w_{i_1},w_{i_2},w_{i_3})\Big),$$
while $$Q^{\chi(\mathcal{O}/\mathcal{I})}=\prod_{i}Q^{\chi(\pi_i)}\prod_{ij} Q^{\chi(\lambda_{ij},(l_{ij},l'_{ij}))},\ \ u^{[\mathcal{O}/\mathcal{I}]}=\prod_{ij} u^{|\lambda_{ij}|[C_{ij}]}.$$

Given two-dimensional partitions $\lambda_1,\lambda_2,\lambda_3$, set $P(\lambda_1,\lambda_2,\lambda_3)$ to be the set of three-dimensional partitions $\pi$ whose asymptotics are given by $\lambda_1,\lambda_2,\lambda_3$; that is $$P(\lambda_1,\lambda_2,\lambda_3)=\Big\{ \pi\ \mathrm{a\ three\textendash dimensional\ partition}\ |\ \pi^{(k)}=\lambda_k,\ k\in \{1,2,3\}\Big \}.$$ Then, set $$\hat{E}(\lambda,(l,l'))(w_{i_1},w_{i_2},w_{i_3})=Q^{\chi(\lambda,(l,l'))}\hat{a}\Big(E_{\lambda,(l,l')}(w_{i_1},w_{i_2},w_{i_3})\Big),$$ and $$\hat{V}(\lambda,\mu,\nu)(w_{i_1},w_{i_2},w_{i_3})=\sum_{\pi \in P(\lambda,\mu,\nu)} Q^{\chi(\pi)}\hat{a}\Big(V_{\pi}(w_{i_1},w_{i_2},w_{i_3})\Big).$$

We may then write $Z_{DT}(X)$ as \begin{align}\label{fpart} \sum_{\{\lambda_{ij}\}} \bigg( \prod_{ij} u^{|\lambda_{ij}|[C_{ij}]}\hat{E}(\lambda_{ij},(l_{ij},l'_{ij}))(w_{i_1},w_{i_2},w_{i_3}) \prod_{i}  \hat{V}(\lambda_{i1},\lambda_{i2},\lambda_{i3})(w_{i_1},w_{i_2},w_{i_3}) \bigg),\end{align} where the sum is over all assignments of two-dimensional partitions to the bounded edges of $\Delta(X)$.

The expression $\hat{V}(\lambda,\mu,\nu)(t_1,t_2,t_3)$ is called the {\em K-theoretic equivariant vertex}. 

\subsubsection{Reduced partition function}
The reduced K-theoretic partition function admits a similar expression. From (\ref{fpart}), we see $$Z_{DT}(X)|_{u=0}=\prod_{i}\hat{V}(\emptyset,\emptyset,\emptyset)(w_{i_1},w_{i_2},w_{i_3}),$$ so that we may write $Z'_{DT}(X)$ as \begin{align} \label{redpart} \sum_{\{\lambda_{ij}\}} \bigg( \prod_{ij} u^{|\lambda_{ij}|[C_{ij}]}\hat{E}(\lambda_{ij},(l_{ij},l'_{ij}))(w_{i_1},w_{i_2},w_{i_3}) \prod_{i}  \frac{\hat{V}(\lambda_{i1},\lambda_{i2},\lambda_{i3})(w_{i_1},w_{i_2},w_{i_3})}{\hat{V}(\emptyset,\emptyset,\emptyset)(w_{i_1},w_{i_2},w_{i_3})} \bigg).
\end{align}

The expression $\hat{V}(\emptyset,\emptyset,\emptyset)$ is completely characterized by Nekrasov's formula (\ref{nek}), so that the difference between $Z_{DT}(X)$ and $Z'_{DT}(X)$ is an  easily understood prefactor; however, our formulas will be more concise when formulated in terms of $Z'_{DT}(X)$. Consequently, we set \begin{align*} \hat{V}'(\lambda_{i1},\lambda_{i2},\lambda_{i3})(w_{i_1},w_{i_2},w_{i_3})=\frac{\hat{V}(\lambda_{i1},\lambda_{i2},\lambda_{i3})(w_{i_1},w_{i_2},w_{i_3})}{\hat{V}(\emptyset,\emptyset,\emptyset)(w_{i_1},w_{i_2},w_{i_3})}
\end{align*} to be the vertex contribution normalized by the contribution of the Hilbert scheme of points.

\subsection{Preferred limits of vertex and edge contributions}
We analyze the behavior of certain edge contributions $\hat{E}(\lambda,(l,l'))$ and vertex contributions $\hat{V}(\lambda,\mu,\nu)$ under certain one-parameter subgroups. 

\subsubsection{Rigidity}

The symmetry (\ref{termsym}) enjoyed by the vertex and edge contributions ensures that each contribution behaves well under limits of generic one-parameter subgroups. Let $A\subset T$ be the Calabi Yau subtorus $\ker(\kappa).$ Recall that such one-parameter subgroups $\sigma:\C^{\times}\to A$ are called {\em slopes} (this name is used as $A$ is two-dimensional.)

It is explained in \cite{MNOP1} that no weight appearing in $T_{{\mathcal{I}}}^{\vir}$ for ${\mathcal{I}}\in DT(X)^T$ is a power of $\kappa$. Let $w$ be a weight appearing in $\Def^{\vir}$, and $\sigma:\C^{\times}\to A$ be a slope such that $w(\sigma(z)t)=z^rw(t)$ where $r\not=0$. Then, we have \begin{align}\label{rigidity} \lim_{z\to 0} \frac{ (\frac{\kappa}{w(\sigma(z)t)})^{1/2}-(\frac{\kappa}{w(\sigma(z)t)})^{-1/2}}{(w(\sigma(z)t))^{1/2}-(w(\sigma(z)t))^{-1/2}} & =\lim_{z\to 0} \frac{z^{-r/2} (\frac{\kappa}{w})^{1/2}-z^{r/2} (\frac{\kappa}{w})^{-1/2}}{z^{r/2}w^{1/2}-z^{-r/2}w^{-1/2}} \nonumber \\ &= \begin{cases} -\kappa^{1/2} & \text{if}\ r>0 \\ -\kappa^{-1/2} & \text{if}\ r<0
\end{cases}
\end{align}

Given a virtual $T$-character $V$ such that $$V=-V^{\vee}\cdot \kappa,$$ we may write $$V=\sum_{i} u_i-v_i,$$ where each $u_i$ and $v_i$ are $T$-weights and $u_iv_i=\kappa.$ Recall that a slope $\sigma:\C^{\times}\to A$ is said to be $V$-generic if $u_i(\sigma(z)t)\not= u_i(t)$ for all $u_i$; as $u_iv_i=\kappa$, it immediately follows that $v_i(\sigma(z)t)\not= v_i(t).$

Now, from (\ref{toteul}), we see that for any fixed $\beta$, the fixed point set $DT(X,\beta,n)^{{T}}$ is finite for any fixed $n$ and empty for $n\ll 0$. From (\ref{rigidity}) we conclude that for generic $\sigma:\C^{\times}\to A$, we have \begin{align}\label{livesinG} \sum_{n} Q^n\chi(DT(X,\beta,n),\tilde{\mathcal{O}}^{\vir})^{\sigma} \in \Gamma,\end{align}

where $$\Gamma\subset \Q[[Q^{\pm 1}, \kappa^{\pm 1/2}]]$$ is the integral domain \begin{align}\label{Gdef}\Big\{ \sum_{i,j\in \Z} c_{i,j} Q^i\kappa^{j/2} \ \Big | \begin{array}{l}  c_{i,j}\in \Q, \\  \text{there exists}\ N\ \text{such that}\ c_{i,j}=0\ \text{whenever}\ i<N ,\\  \text{and for any fixed}\ i \ \text{only finitely many}\ c_{i,j}\ \text{are nonzero}  \end{array}\Big\}. \end{align}

We now recall the following definition from \cite[7.3.1]{NO}.

\begin{definition}\label{defindex}
Given a virtual $T$-character $$V=\sum_i u_i-v_i$$ with $u_iv_i=\kappa$, and a $V$-generic slope $\sigma$, define $\ind_{\sigma}(V)\in \Z$ to be the number of weights of $u_i$ that are attracting with respect to $\sigma$ (in the sense of Section \ref{limind}) minus the number of $u_i$ that are repelling with respect to $\sigma$. 
\end{definition}

Given $(r_1,r_2,r_3)\in \Z$ such that $r_1+r_2+r_3=0$, we let $\sigma(r_1,r_2,r_3):\C^{\times}\to A$ denote the slope $$z\mapsto (z^{r_1},z^{r_2},z^{r_3}).$$

For an example of the index, let $\pi$ denote the three-dimensional partition consisting of a single box. Then $$T_{\pi}^{\vir}(t_1,t_2,t_3)=t_1+t_2+t_3-t_1t_2-t_1t_3-t_2t_3,$$ so that $\sigma(r_1,r_2,r_3)$ is $T_{\pi}^{\vir}(t_1,t_2,t_3)$-generic if all $r_k$ are nonzero. For such $\sigma(r_1,r_2,r_3)$, we have\begin{align} \ind_{\sigma}(T_{\pi}^{\vir}(t_1,t_2,t_3))=\begin{cases} 2-1=1\ \ \ \ \ \text{if 2 of the}\ r_k\ \text{are positive} \\ 1-2=-1\ \ \ \text{if 2 of the}\ r_k\ \text{are negative}. \end{cases} \end{align} Note the similarity with (\ref{indexex}).

It follows from (\ref{rigidity}) that, for $V$ as in Definition \ref{defindex}, we have $$(\hat{V})^{\sigma}=(-\kappa^{1/2})^{\ind_{\sigma}(V)}.$$

 Set $$t=-Q\kappa^{1/2}, q=-Q\kappa^{-1/2}.$$  Under this change of variables, we have $$\Gamma=\Big\{\sum_{\substack{i,j\in 1/2\cdot \Z \\ i+j \in \Z }}c_{i,j} t^iq^j \Big| \begin{array}{l} c_{i,j}\in \Q \\ \text{there exists}\ N\ \text{such that}\ c_{i,j}=0\ \text{whenever}\ i+j<N \\  \text{and for any fixed}\ i\ \text{only finitely many}\ c_{i-j,j}\ \text{are nonzero} \end{array} \Big\}.$$ We separately compute the limits $E^{\sigma}_{\lambda_{ij},(l,l')}$ (which are monomials in $q, t, \sqrt{qt}$) and $V^{\sigma}_{\pi}\in \Gamma$. The latter limits do not seem to admit a concise expression in terms of the partition data for arbitrary slopes $\sigma$. However, as observed in \cite[Sec. 8]{NO}, when one of the exponents $r_k$ has very small magnitude compared to the others, these limits simplify considerably. Such $\sigma$ are called {\em preferred slopes}.

\subsubsection{Edge limits}

For the two Calabi-Yau threefolds $X$ we consider in Section \ref{threefold} (and for essentially any $X$ for which the methods of this paper can produce nontrivial identities), the possible normal bundles to a $T$-fixed rational curve $C_{ij}\subset X$ are $\mathcal{O}(-1)\oplus \mathcal{O}(-1)$ and $\mathcal{O}\oplus \mathcal{O}(-2)$. We compute the limits of the edge terms $\hat{E}_{\lambda}$ associated to these two configurations. 

We first consider a $T$-fixed rational curve whose normal bundle in $X$ is $\mathcal{O}(-1)\oplus\mathcal{O}(-1)$, as depicted in Figure \ref{minus1}.

 \begin{figure}[htb]
\centering{
\scalebox{.4}{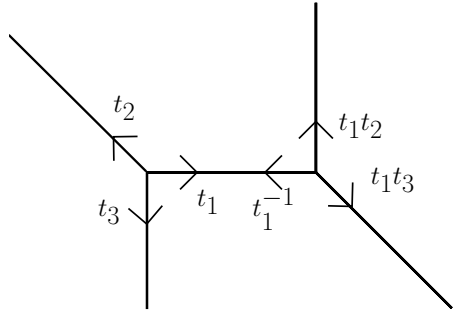}
\caption{The weights along a $T$-fixed rational curve with normal bundle $\mathcal{O}(-1)\oplus \mathcal{O}(-1)$} \label{minus1}
}
\end{figure}

By (\ref{edgesub}), we compute \begin{align*} E_{\lambda,(-1,-1)}(t_1,t_2,t_3)&= \frac{T_{\lambda}(t_2,t_3)}{1-t_1^{-1}}+\frac{T_{\lambda^t}(t_1t_3,t_1t_2)}{1-t_1}
\nonumber \\&= \sum_{\square\in\lambda} \frac{1-t_1^{-l(\square)+a(\square)}}{1-t_1^{-1}}t_2^{-l(\square)}t_3^{a(\square)+1}+\frac{1-t_1^{l(\square)-a(\square)}}{1-t_1^{-1}}t_2^{l(\square)+1}t_3^{-a(\square)} \nonumber 
\\& = \sum_{\square\in\lambda\ |\  a(\square)< l(\square)}  (1+\cdots+t_1^{-l(\square)+a(\square)+1})t_2^{-l(\square)}t_3^{a(\square)+1} \nonumber \\&\ \ \ +\sum_{\square\in\lambda \ |\  a(\square)< l(\square)} - (t_1+\cdots+t_1^{l(\square)-a(\square)})t_2^{l(\square)+1}t_3^{-a(\square)} 
\nonumber \\&\ \ \ +\sum_{\square\in\lambda \ |\  a(\square)>l(\square)} -(t_1+\cdots+t_1^{-l(\square)+a(\square)})t_2^{-l(\square)}t_3^{a(\square)+1} \nonumber \\ & \ \ \ +\sum_{\square\in\lambda \ |\  a(\square)>l(\square)} (1+\cdots+t_1^{l(\square)-a(\square)+1})t_2^{l(\square)+1}t_3^{-a(\square)} .
\end{align*}

Using this explicit expression, we compute $\ind_{\sigma}(E_{\lambda,(-1,-1)}(t_1,t_2,t_3))$ for preferred $\sigma$. Recall from Definition \ref{defindex} that this index is equal to the number of weights among \begin{align}\label{indwts}&\sum_{\square\in\lambda\ |\  a(\square)< l(\square)} (1+\cdots+t_1^{-l(\square)+a(\square)+1})t_2^{-l(\square)}t_3^{a(\square)+1}\nonumber \\&+\sum_{\square\in\lambda \ |\  a(\square)>l(\square)} (1+\cdots+t_1^{l(\square)-a(\square)+1})t_2^{l(\square)+1}t_3^{-a(\square)}\end{align} that are attracting with respect to $\sigma$ minus the number that are repelling with respect to $\sigma$.

For example, suppose that $\sigma=\sigma(r_1,r_2,r_3)$ with $r_1\gg r_2>0\gg r_3.$ For this choice of $\sigma$, the weights of the first summand of (\ref{indwts}) are repelling, while those of the second summand are attracting. So, the weights of the summand in (\ref{indwts}) associated to some $\square\in \lambda$ contribute $l(\square)-a(\square)$ to the index, hence \begin{align*}\ind_{\sigma(r_1,r_2,r_3)}(E_{\lambda,(-1,1)}(t_1,t_2,t_3))=\sum_{\square\in\lambda}\big( l(\square)-a(\square)\big)&=\frac{||\lambda^t||^2-|\lambda|}{2}-\frac{||\lambda||^2-|\lambda|}{2}\\&=\frac{||\lambda^t||^2-||\lambda||^2}{2}.\end{align*}

The indices associated to any preferred slope $\sigma$ can be similarly computed. We conclude the following.
$$\begin{tabu}{c|c}
\sigma=\sigma(r_1,r_2,r_3) & \ind_{\sigma}(E_{\lambda,(-1,-1)}(t_1,t_2,t_3)) \\
\hline
r_1\gg r_2>0\gg r_3 & (||\lambda^t||^2-||\lambda||^2)/2\\
\hline
r_2\gg r_1>0\gg r_3 &(||\lambda^t||^2-||\lambda||^2)/2 \\
\hline
r_1\gg r_3> 0 \gg  r_2 & (||\lambda||^2-||\lambda^t||^2)/2 \\
\hline
r_3\gg r_1> 0 \gg  r_2 &  (||\lambda||^2-||\lambda^t||^2)/2  \\
\hline
r_2 \gg  r_3 > 0 \gg  r_1 &  (||\lambda^t||^2-||\lambda||^2)/2 \\
\hline
r_3 \gg  r_2 > 0 \gg  r_1 &  (||\lambda||^2-||\lambda^t||^2)/2 
\end{tabu}$$

The indices for remaining preferred $\sigma$ are determined by the relation $$\ind_{\sigma(-r_1,-r_2,-r_3)}(V)=-\ind_{\sigma(r_1,r_2,r_3)}(V).$$ By (\ref{edgeeul}), we have \begin{align} \chi(\lambda,(-1,-1))=\sum_{(b_1,b_2)\in \lambda} \big(1+b_1+b_2\big)&=|\lambda|+\frac{||\lambda^t||^2-|\lambda|}{2}+\frac{||\lambda||^2-|\lambda|}{2}\nonumber\\&=\frac{||\lambda^t||^2+||\lambda||^2}{2}\label{subdef}.\end{align} We conclude the following.
\begin{proposition}
The limits of the edge contribution $\hat{E}(\lambda,(-1,-1))$ under preferred slopes $\sigma$ are as follows.
$$\begin{tabu}{c|c}\label{edge1}
\sigma=\sigma(r_1,r_2,r_3) & \hat{E}(\lambda,(-1,-1))(t_1,t_2,t_3)^{\sigma} \\
\hline 
r_1\gg r_2>0\gg r_3 & t^{\frac{||\lambda^t||^2}{2}}q^{\frac{||\lambda||^2}{2}}\\
\hline
r_2\gg r_1>0\gg r_3 & t^{\frac{||\lambda^t||^2}{2}}q^{\frac{||\lambda||^2}{2}} \\
\hline
r_1\gg r_3> 0 \gg  r_2 & t^{\frac{||\lambda||^2}{2}}q^{\frac{||\lambda^t||^2}{2}}\\
\hline
r_3\gg r_1> 0 \gg  r_2 & t^{\frac{||\lambda||^2}{2}}q^{\frac{||\lambda^t||^2}{2}}\\
\hline
r_2 \gg  r_3 > 0 \gg  r_1 &  t^{\frac{||\lambda^t||^2}{2}}q^{\frac{||\lambda||^2}{2}}\\
\hline
r_3 \gg  r_2 > 0 \gg  r_1 &  t^{\frac{||\lambda||^2}{2}}q^{\frac{||\lambda^t||^2}{2}}
\end{tabu}$$
\end{proposition}

We repeat the same procedure for a $T$-fixed rational curve in $X$ with normal bundle $\mathcal{O}\oplus\mathcal{O}(-2)$, as depicted in Figure \ref{minus2}.

 \begin{figure}[htb]
\centering{
\scalebox{.4}{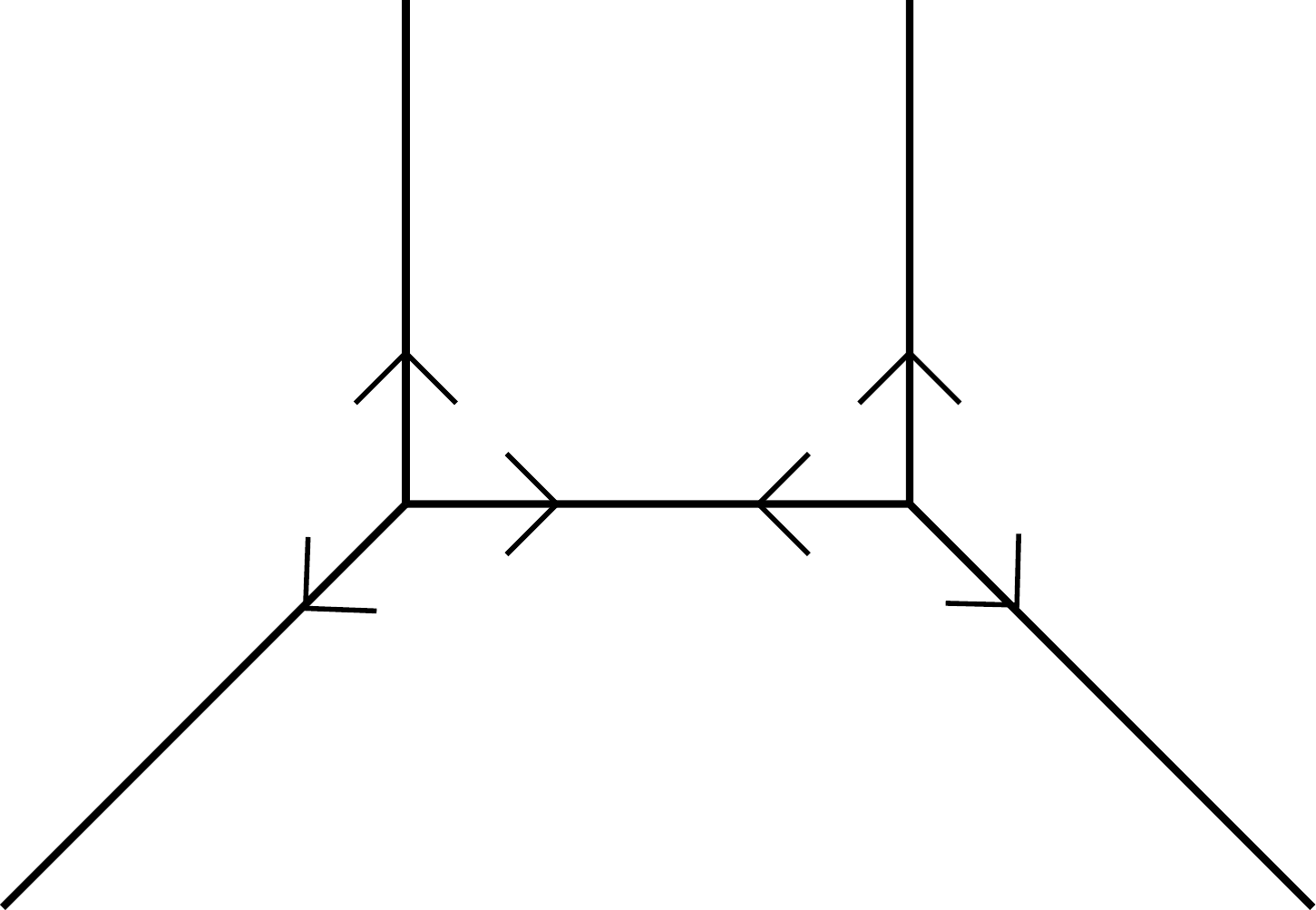}
\caption{The weights along a $T$-fixed rational curve with normal bundle $\mathcal{O}\oplus \mathcal{O}(-2)$} \label{minus2}
}
\end{figure}
We then compute 
 \begin{align*} E_{\lambda,(0,-2)}(t_1,t_2,t_3)&= \frac{T_{\lambda}(t_2,t_3)}{1-t_1^{-1}}+\frac{T_{\lambda^t}(t_1^2t_3,t_2)}{1-t_1} \nonumber
\\&= \sum_{\square\in\lambda} \frac{1-t_1^{2a(\square)+1}}{1-t_1^{-1}}t_2^{-l(\square)}t_3^{a(\square)+1}+\frac{1-t_1^{-2a(\square)-1}}{1-t_1^{-1}}t_2^{l(\square)+1}t_3^{-a(\square)} \nonumber
\\&= \sum_{\square\in\lambda} \Big((1+t_1^{-1}+\cdots + t_1^{-2a(\square)} )t_2^{l(\square)+1}t_3^{-a(\square)}\\ &\ \ \ \ \  \ \ \ \ \ -(t_1+\cdots+t_1^{2a(\square)+1})t_2^{-l(\square)}t_3^{a(\square)+1} \Big). 
\end{align*}

From the identity \begin{align}\label{edge2eq} \sum_{\square \in \lambda} (2a(\square)+1)=2\frac{||\lambda||^2-|\lambda|}{2}+|\lambda|=||\lambda||^2,\end{align} we conclude the following.
$$\begin{tabu}{c|c}
\sigma=\sigma(r_1,r_2,r_3) & \ind_{\sigma}(E_{\lambda,(0,-2)}(t_1,t_2,t_3)) \\
\hline 
r_1\gg r_2>0\gg r_3 & |\lambda| \\
\hline
r_2\gg r_1>0\gg r_3 & ||\lambda||^2\\
\hline
r_1\gg 0>r_2\gg r_3 & -|\lambda|  \\
\hline
r_2\gg 0>r_1\gg r_3 & ||\lambda||^2 \\
\end{tabu}$$

By (\ref{edgeeul}) we have $$\chi(\lambda,(0,-2))=\sum_{(b_1,b_2)\in \lambda} \big(1+2b_2\big)=|\lambda|+2\frac{||\lambda||^2-|\lambda|}{2}=||\lambda||^2.$$ We conclude the following.
\begin{proposition}\label{edge2} 
The limits of the edge contribution $\hat{E}(\lambda,(0,-2))$ under preferred slopes are as follows.
$$\begin{tabu}{c|c}
\sigma=\sigma(r_1,r_2,r_3) & \hat{E}(\lambda,(0,-2))(t_1,t_2,t_3)^{\sigma} \\
\hline 
r_1\gg r_2>0\gg r_3 & t^\frac{||\lambda||^2+|\lambda|}{2}q^\frac{||\lambda||^2-|\lambda|}{2} \\
\hline
r_2\gg r_1>0\gg r_3 & t^{||\lambda||^2}\\
\hline
r_1\gg 0>r_2\gg r_3 & t^\frac{||\lambda||^2-|\lambda|}{2}q^\frac{||\lambda||^2+|\lambda|}{2}  \\
\hline
r_2\gg 0>r_1\gg r_3 & t^{||\lambda||^2} \\
\hline
r_3\gg  0 > r_2 \gg  r_1 & t^\frac{||\lambda||^2-|\lambda|}{2}q^\frac{||\lambda||^2+|\lambda|}{2}\\
\hline
r_3\gg 0>r_1\gg r_2 & q^{||\lambda||^2}\\
\hline
r_3 \gg  r_2 > 0 \gg r_1 & t^\frac{||\lambda||^2+|\lambda|}{2}q^\frac{||\lambda||^2-|\lambda|}{2} \\
\hline
r_3\gg r_1>0\gg r_2 & q^{||\lambda||^2} \\
\end{tabu}$$

\end{proposition}

\subsubsection{Vertex limits}

The vertex $V_{\pi}(t_1,t_2,t_3)$ defined in Section \ref{vertexsection} is a (suitably renormalized) quadratic expression in the character of the three-dimensional partition $\pi$.  For an arbitrary partition $\pi$ and slope $\sigma$, there is no known concise expression for $\ind_{\sigma}(V_{\pi}(t_1,t_2,t_3))$. 

However, when $\sigma$ is a preferred slope, it is shown in \cite[Sec. 8]{NO} that the index $\ind_{\sigma}(V_{\pi}(t_1,t_2,t_3))$ is equal to the index of a different character written as a single sum over the boxes of $\pi$. This simpler expression for the index yields a combinatorial expression for preferred limits of the vertex that is well-suited to computations.

We recall this procedure from \cite[Sec. 8]{NO}. Some notation for partitions is needed. Fix a two-dimensional partition $\nu\subset \Z_{\geq 0}^2,$ and let $c^{-}_{i}$ and $c^{+}_{i}$ denote the ordered values of $b_2-b_1$ when $(b_1,b_2)$ are among the inner and outer corners of $\nu$, respectively; see Figure \ref{fig:corners}.

\begin{figure}[bht]
 \scalebox{.25}{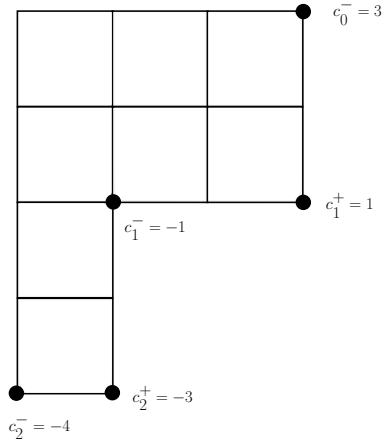}
\caption{The partition $\lambda=(3,3,1,1)$ with corners labeled}
\label{fig:corners}
\end{figure}

To be precise, suppose $n_1>\ldots>n_{d(\nu)}$ are the distinct parts occuring in $\nu$, and $\nu$ is the partition in which the part $n_i$ occurs with nonzero multiplicity $m_i.$ Then for $i=0,\cdots,d(\nu)$, we set $$c^{-}_{i}=n_{i+1}-\sum_{j=1}^i m_j,$$ (here $n_{d(\nu)}+1=0$), and for $i=1,\ldots,d(\nu)$, we set $$c^{+}_{i}=n_{i}-\sum_{j=1}^i m_j.$$ When $\nu=\emptyset$, the only corner is $c^-_0=0.$

We then define a function $\psi_{\nu}: \Z_{\geq 0}^3\to\{t^{-1}_1,t^{-1}_2,t_3,t_3^{-1}\}$ as follows:

\begin{equation*} \psi_{\nu}(b_1,b_2,b_3) = \begin{cases} t_2^{-1} & \text{if}\ b_2-b_1  > c^{-}_0\ \\& \text{or}\ c^+_i> b_2-b_1 > c^-_i\ \text{for some}\ i \\ t_1^{-1} &  \text{if}\ c^-_i>b_2-b_1> c^+_{i+1}  \ \text{for some}\ i \\&  \text{or}\ c^{-}_{d(\nu)}> b_2-b_1\\ t_3 & \text{if}\ b_2-b_1=c_i^{-} \text{for some}\ i \\ t^{-1}_3 & \text{if}\ b_2-b_1=c_i^{+} \text{for some}\ i.
\end{cases}
\end{equation*}

In particular, when $\nu$ is empty, we have
\begin{equation*} \psi_{\emptyset}(b_1,b_2,b_3) = \begin{cases} t_2^{-1} & \text{if}\ b_2-b_1 > 0 \\ t_3 &  \text{if}\ b_2-b_1=0 \\ t_1^{-1} &  \text{if}\ b_2-b_1<0. 
\end{cases}
\end{equation*}

Now, given a three-dimensional partition $\pi$, a box $\square\in \pi$ and $k=1,2$ or $3$, set 
\begin{equation*} \delta_{k}(\square)=\begin{cases} 1 & \text{if} \ \square\in \pi\langle k\rangle \\ 0 & \text{if} \ \square\not\in \pi\langle k\rangle.  \end{cases}
\end{equation*}

Then, we define the virtual character $W_{\pi}(t_1,t_2,t_3)$ to be 
 \begin{align}\label{Wchar} W_{\pi}(t_1,t_2,t_3)=  \sum_{\square\in \pi} \Big(\psi_{\pi^{(3)}}(\square)-\delta_1(\square)\psi_{\emptyset}(\square)-\delta_2(\square)\psi_{\emptyset}(\square)-\delta_3(\square)\psi_{\pi^{(3)}}(\square) \Big). \end{align} 

Note the similarity between (\ref{vertexeul1}) and (\ref{Wchar}). Given some fixed $\pi$, we have $\psi_{\pi^{(3)}}(b_1,b_2,b_3)=\psi_{\emptyset}(b_1,b_2,b_3)$ for $|b_2-b_1|\gg 0$. So, the sum in (\ref{Wchar}) is finite.

We can now recall from \cite{NO}  a formula  for $\ind_{\sigma}(V_{\pi})$ for preferred slopes $\sigma$.
\begin{theorem}\label{nekok} \cite[Thm. 2]{NO}
If $\pi$ is a three-dimensional partition and $\sigma(r_1,r_2,r_3)$ is a preferred slope with with $|r_1|,|r_2|\gg |r_3|$, then
\begin{align}\label{indexlim}\ind_{\sigma(r_1,r_2,r_3)}(V_{\pi}(t_1,t_2,t_3))=\ind_{\sigma(r_1,r_2,r_3)}\Big(W_{\pi}(t_1,t_2,t_3)-\kappa\cdot W_{\pi}^{\vee}(t_1,t_2,t_3)\Big).\end{align}
\end{theorem}


Given a preferred slope $\sigma(r_1,r_2,r_3)$ with $|r_1|,|r_2|\gg |r_3|$, and a two-dimensional partition $\nu$, define  $$\Psi^{\sigma}_{\nu}:\Z_{\geq 0}^3\to \{-Q\kappa^{1/2},-Q\kappa^{-1/2}\}=\{t,q\}$$ to be the function $$\Psi^{\sigma}_{\nu}(b_1,b_2,b_3)=\begin{cases} t & \text{if}\ \psi_{\nu}(b_1,b_2,b_3)\ \text{is attracting with respect to}\ \sigma \\ q & \text{if}\ \psi_{\nu}(b_1,b_2,b_3)\ \text{is repelling with respect to}\ \sigma.\end{cases}$$  

For example, if $r_1\gg r_3 > 0 \gg r_2$, then $$\Psi^{\sigma(r_1,r_2,r_3)}_{\emptyset}(b_1,b_2,b_3)=\begin{cases} t & \text{if}\ b_2-b_1\geq 0 \\ q & \text{if}\ b_2-b_1<0, \end{cases}$$ and more generally, $$\Psi^{\sigma(r_1,r_2,r_3)}_{\nu}(b_1,b_2,b_3)=\begin{cases} t & \text{if}\ \psi_{\nu}(b_1,b_2,b_3)\in \{t_2^{-1}, t_3\} \\ q & \text{if}\ \psi_{\nu}(b_1,b_2,b_3)\in \{t_1^{-1}, t^{-1}_3\}  .\end{cases}$$

For $N>0$, set $$B_N=\{(b_1,b_2,b_3) |\  |b_2-b_1|\leq N \}.$$ Note that for any fixed $N$, the sets $B_N\cap \pi\langle 1\rangle, B_N\cap \pi\langle 2 \rangle$ and $B_N\cap (\pi -\pi\langle 3\rangle)$ consist of finitely many boxes. Combining (\ref{vertexeul1}) and Theorem \ref{nekok}, we obtain the following.

\begin{corollary}
For a preferred slope $\sigma(r_1,r_2,r_3)$ with $|r_1|,|r_2|\gg |r_3|$ and  $N\gg 0$, we have \begin{align}\label{singlecont} Q^{\chi(\pi)}\hat{a}(V_{\pi}(t_1,t_2,t_3))^{\sigma(r_1,r_2,r_3)}=\frac{\displaystyle\prod_{\{\square \in \pi\cap B_N\ |\  \square\not\in \pi\langle 3\rangle \} } \Psi^{\sigma}_{\pi^{(3)}}(\square)}{\displaystyle\prod_{\square\in B_N\cap\pi\langle 1\rangle }\Psi^{\sigma}_{\emptyset}(\square)\displaystyle\prod_{\square\in B_N\cap\pi\langle 2\rangle}\Psi^{\sigma}_{\emptyset}(\square)}.
\end{align}
\end{corollary}

We now recall how to use expression (\ref{singlecont}) to compute the limit $\hat{V}(\lambda,\mu,\nu)^{\sigma}$ of the refined vertex under preferred slopes $\sigma(r_1,r_2,r_3)$. For this computation, we now assume that $r_1\gg r_3 >0 \gg r_2$.

We first compute the denominator of (\ref{singlecont}) for such $\sigma$. Let $\pi\in P(\lambda,\mu,\nu)$ and take $N$ to be sufficiently large (say, greater than $|\lambda|$ and $|\mu|$.) Then, if $0\leq c< l(\lambda)$, then the number of  boxes $(b_1,b_2,b_3)\in B_N\cap \pi\langle 1\rangle$  with  $b_2-b_1=c$ is $$\sum_{i=c+1}^{l(\lambda)} \lambda_i.$$ On the other hand, if $-N\leq c<0$, then the number of boxes $(b_1,b_2,b_3)\in B_N\cap \pi\langle 1\rangle$ with $b_2-b_1=c$ is $|\lambda|$. 

As $$\sum_{c=1}^{l(\lambda)} \sum_{i=c+1}^{l(\lambda)} \lambda_i = \sum_{i=1}^{l(\lambda)} (i-1)\lambda_i=\sum_{(b_1,b_2)\in \lambda} b_1 = \frac{||\lambda^t||^2-|\lambda|}{2},$$
we conclude that
$$\prod_{\square\in B_N\cap\pi\langle 1 \rangle}\Psi^{\sigma}_{\emptyset}(\square)=q^{N|\lambda|}t^{|\lambda|}t^{\frac{||\lambda^t||^2-|\lambda|}{2}},$$ where the first of the three factors on the right hand side arises from counting those $(b_1,b_2,b_3)\in B_N\cap\pi\langle 1 \rangle$ with $b_2-b_1<0$, the second factor arises from counting such boxes with $b_2-b_1=0$, and the third arises from counting such boxes with $b_2-b_1>0$.

Similarly, $$\prod_{\square\in B_N\cap\pi\langle 2 \rangle}\Psi^{\sigma}_{\emptyset}(\square)=t^{N|\mu|}t^{|\mu|}q^{\frac{||\mu||^2-|\mu|}{2}}.$$

The limit $\hat{V}(\lambda,\mu,\nu)(t_1,t_2,t_3)^{\sigma}$ can then be computed by summing (\ref{singlecont}) over all partitions $\pi\in P(\lambda,\mu,\nu)$. Note that there are no values of $(r_1, r_2, r_3)$ and $N$ for which (\ref{singlecont}) holds simultaneously for all $\pi\in P(\lambda,\mu,\nu)$. However, for any fixed $m\in \Z$ one may choose $R_m, R'_m$ and $N_m\in \Z$ such that for any $\pi\in P(\lambda,\mu,\nu),$ and triple $(r_1,r_2,r_3)\in \Z^3$ and $N>0$ satisfying $$\chi(\pi)\leq m,\ r_1>R_m>R'_m>r_3>0>r_2,\  \mathrm{and}\ N>N_m,$$ the equation (\ref{singlecont}) holds. In plainer language, any individual term of the limit $$\hat{V}(\lambda,\mu,\nu)(t_1,t_2,t_3)^{\sigma(r_1,r_2,r_3)}$$ stabilizes as $r_1\gg r_3>0\gg r_2$, and can be computed using (\ref{singlecont}).

Applying (\ref{singlecont}) (and recalling that $r_1\gg r_3>0\gg r_2$) we have \begin{align} \hat{V}&(\lambda,\mu,\nu)(t_1,t_2,t_3)^{\sigma(r_1,r_2,r_3)}\nonumber \\&=\sum_{\pi\in P(\lambda,\mu,\nu)} \lim_{N\to \infty}\bigg(\frac{t^{-\frac{||\lambda^t||^2-|\lambda|}{2}}q^{-\frac{||\mu||^2-|\mu|}{2}}}{q^{N|\lambda|}t^{|\lambda|} t^{N|\mu|}t^{|\mu|}}\displaystyle\prod_{\{ \square \in \pi\cap B_N\ |\ \square\not\in\pi\langle 3\rangle\}}  \Psi^{\sigma}_{\nu}(\square) \bigg).\label{vertexsum}\end{align} A general framework for evaluating such sums was introduced in \cite{OR}. Namely, given a partition $\pi\in P(\lambda,\mu,\nu)$, the configuration $\pi/\pi\langle 3\rangle$ can be regarded as a sequence of two-dimensional partitions whose interlacing behavior is read off from $\nu$. Sums of the form (\ref{vertexsum}) can then be evaluated using a transfer matrix approach, where each such two-dimensional partition is assigned a weight using $\Psi^{\sigma}_{\nu}$.

This procedure is carried out in \cite[Eqn (150)]{IKV} and yields the answer \begin{align*}\sum_{\pi\in P(\lambda,\mu,\nu)}& \lim_{N\to \infty}\frac{\displaystyle\prod_{\{\square \in \pi\cap B_N, \square\not\in\pi\langle 3\rangle \}} \Psi^{\sigma}_{\nu}(\square)}{q^{N|\lambda|}t^{|\lambda|} t^{N|\mu|}t^{|\mu|}}= \frac{ t^{\frac{-|\lambda|}{2}}q^{\frac{-|\mu|}{2}}\displaystyle \sum_{\eta}\big(\frac{q}{t}\big )^{\frac{|\eta|}{2}}s_{\lambda/\eta}(t^{-\rho}q^{-\nu^t})s_{\mu^t/\eta}(q^{-\rho}t^{-\nu})}{\displaystyle\prod_{i,j\geq 0}(1-q^it^{j+1})\displaystyle\prod_{\square\in\nu}( 1-q^{l(\square)}t^{a(\square)+1}) }
.\end{align*}
Here $t^{-\rho}$ denotes the sequence $$(t^{1/2},t^{3/2},t^{5/2},\cdots);$$ the symbol $t^{-\lambda}$ denotes the sequence $(t^{-\lambda_1},t^{-\lambda_2},\cdots,)$ and products of sequences are taken termwise. The symbols $s_{\lambda/\eta}$ denote skew Schur polynomials; see, for example, \cite[App. A]{IKP} for a concise introduction.

  After incorporating the remaining factor $$t^{-\frac{||\lambda^t||^2-|\lambda|}{2}}q^{-\frac{||\mu||^2-|\mu|}{2}},$$ we deduce that for $r_1\gg r_3>0\gg r_2$, we have $$ \hat{V}(\lambda,\mu,\nu)(t_1,t_2,t_3)^{\sigma(r_1,r_2,r_3)}=\frac{ t^{\frac{-||\lambda^t||^2}{2}}q^{-\frac{||\mu||^2}{2}}\displaystyle\sum_{\eta}\Big(\frac{q}{t}\Big)^{\frac{|\eta|}{2}}s_{\lambda/\eta}(t^{-\rho}q^{-\nu^t})s_{\mu^t/\eta}(q^{-\rho}t^{-\nu})}{\displaystyle\prod_{i,j\geq 0}(1-q^it^{j+1}) \displaystyle\prod_{\square\in\nu} (1-q^{l(\square)}t^{a(\square)+1}) } 
.$$

In particular, we have \begin{align}\label{deg0}\hat{V}(\emptyset,\emptyset,\emptyset)(t_1,t_2,t_3)^{\sigma(r_1,r_2,r_3)}=\frac{1}{\prod_{i,j\geq 0}(1-q^it^{j+1})};
\end{align} Note that this equality also follows directly from Nekrasov's formula  (\ref{nek}). 

So, we set $$C(\lambda,\mu,\nu)(t,q)= \frac{ t^{-\frac{||\lambda^t||^2}{2}}q^{-\frac{||\mu||^2}{2}}}{\displaystyle \prod_{\square\in\nu} (1-q^{l(\square)}t^{a(\square)+1}) } 
\sum_{\eta}\Big(\frac{q}{t}\Big)^{\frac{|\eta|}{2}}s_{\lambda/\eta}(t^{-\rho}q^{-\nu^t})s_{\mu^t/\eta}(q^{-\rho}t^{-\nu}).$$ 

Up to a prefactor that is important for the computations in the following section, the series $C(\lambda,\mu,\nu)$ is equal to the refined topological vertex used in \cite{IKV}. 

The limits of $\hat{V}(\lambda,\mu,\nu)$ under other preferred slopes can be determined similarly, or using symmetries of the vertex. To be precise, suppose that $\sigma(r_1,r_2,r_3)$ and $\sigma(s_1,s_2,s_3)$ are generic slopes with $r_1\gg r_3 > 0 \gg r_2$ and $s_2\gg s_3>0\gg s_1.$ Then, \begin{align*}&\ind_{\sigma(s_1,s_2,s_3)}(W_{\pi}(t_1,t_2,t_3)-\kappa\cdot W^{\vee}_{\pi}(t_1,t_2,t_3))\\&=\ind_{\sigma(r_1,r_2,r_3)}(W_{\pi}(t_2,t_1,t_3)-\kappa\cdot W^{\vee}_{\pi}(t_2,t_1,t_3)).\end{align*} It follows from the definition (\ref{Wchar}) that $$W_{\pi}(t_2,t_1,t_3)=W_{\mathring{\pi}}(t_1,t_2,t_3),$$ where $$\mathring{\pi}=\{(b_2,b_1,b_3) | (b_1,b_2,b_3)\in \pi\}$$ is the partition obtained by reflecting $\pi$ in the $x_1=x_2$ plane. Hence $$\ind_{\sigma(s_1,s_2,s_3)}(V_{\pi}(t_1,t_2,t_3))=\ind_{\sigma(r_1,r_2,r_3)}(V_{\mathring{\pi}}(t_1,t_2,t_3)).$$

As $$\mathring{\pi}^{(1)}=(\pi^{(2)})^{t}, \mathring{\pi}^{(2)}=(\pi^{(1)})^{t}, \mathring{\pi}^{(3)}=(\pi^{(3)})^{t},$$ we conclude that \begin{align}\label{slopewall}\hat{V}(\lambda,\mu,\nu)(t_1,t_2,t_3)^{\sigma(s_1,s_2,s_3)}=\hat{V}(\mu^t,\lambda^t,\nu^t)(t_1,t_2,t_3)^{\sigma(r_1,r_2,r_3)}.\end{align}

Finally, for any generic slope $\sigma(r_1,r_2,r_3)$ one has \begin{align}\label{sloperfn}\ind_{\sigma(r_1,r_2,r_3)}(V_{\pi}(t_1,t_2,t_3))=-\ind_{\sigma(-r_1,-r_2,-r_3)}(V_{\pi}(t_1,t_2,t_3)).\end{align}

From (\ref{slopewall}) and (\ref{sloperfn}), we obtain the following. 

\begin{proposition}\label{indref}
The limits of the normalized vertex contribution $\hat{V}'(\lambda,\mu,\nu)$ under preferred slopes $\sigma$ are as follows.

$$\begin{tabu}{c|c}
\sigma=\sigma(r_1,r_2,r_3) & \hat{V}'(\lambda,\mu,\nu)(t_1,t_2,t_3)^{\sigma}\\
\hline
 
r_1\gg r_3>0\gg r_2 & C(\lambda,\mu,\nu)(t,q) \\
\hline
r_1\gg 0>r_3\gg r_2 & C(\mu^t,\lambda^t,\nu^t)(q,t) \\
\hline
r_2\gg r_3>0\gg r_1 & C(\mu^t,\lambda^t,\nu^t)(t,q) \\
\hline
r_2\gg 0>r_3\gg r_1 & C(\lambda,\mu,\nu)(q,t)
\end{tabu}$$
\end{proposition}

Note that there are four different choices of preferred slope corresponding to a given choice of ``preferred direction'' for the refined topological vertex.

\section{Dualities for tautological classes on $(\C^2)^{[n]}$ }\label{threefold}
In this section, we prove Theorem \ref{symmetry}.

\subsection{Two toric geometries}

Let $X_1$ and $X_2$ be the smooth toric Calabi-Yau threefolds whose toric diagrams are those in Figure \ref{weights}, with the action of a three-dimensional torus $T=\mathrm{diag}(t_1,t_2,t_3)$ as indicated.

The geometry of $X_1$ and $X_2$ can be read off from the toric skeleta. The threefold $X_1$ has 8 torus fixed points and 8 torus invariant curves, while $X_2$ has 4 torus fixed points and 3 torus invariant curves. For both threefolds, the invariant curves corresponding to vertical and horizontal line segments in Figure \ref{weights} have normal bundle $\mathcal{O}\oplus \mathcal{O}(-2)$, while the remaining invariant curves (those corresponding to line segments at $45^{\circ}$ angles to the edges of the page) have normal bundle $\mathcal{O}(-1)\oplus \mathcal{O}(-1)$. The $T$-weights can be determined by first selecting a single toric chart where $T$ scales the coordinate axes by $t_1,t_2$ and $t_3,$ and using (\ref{weightshift}) to determine the weights in all other charts.

The semigroups $H_2(X_1,\Z)^{\mathrm{eff}}$ and $H_2(X_2,\Z)^{\mathrm{eff}}$ are generated by the classes $\beta$ of the torus invariant curves. The only relations in $H_2(X_1,\Z)^{\mathrm{eff}}$ are that the classes of the invariant curves given by the two vertical line segments are equal, and that the classes of the curves given by the two horizontal line segments are equal. There are no relations between classes of invariant curves in $H_2(X_2,\Z)^{\mathrm{eff}}$. In Figure \ref{kahs}, we label each torus invariant curve $C_i$ with a parameter that denotes, in the sense of (\ref{dtpar}), the variable $u^{[C_i]}$; note, in particular, that in the diagram for $X_1$, homologous curves are assigned the same parameters. These parameters are called K\"{a}hler parameters.

See Section \ref{Xorigin} for a brief explanation as to why the geometries $X_1$ and $X_2$ are selected.


\begin{figure}[htb] 
\begin{subfigure}{1in}
\scalebox{.3}{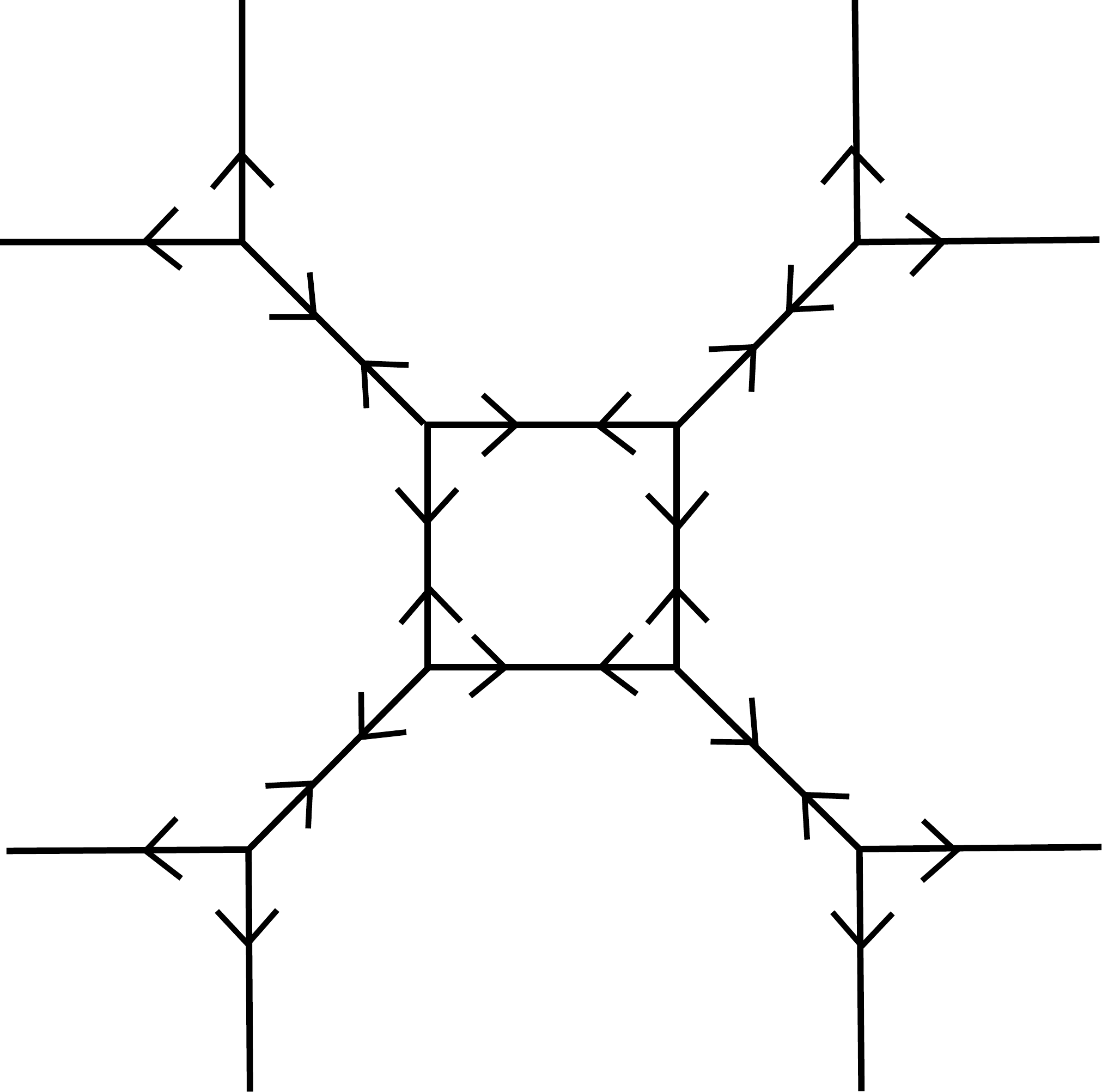}
\end{subfigure}
\hspace{2in} 
\begin{subfigure}{1in}
\scalebox{.3}{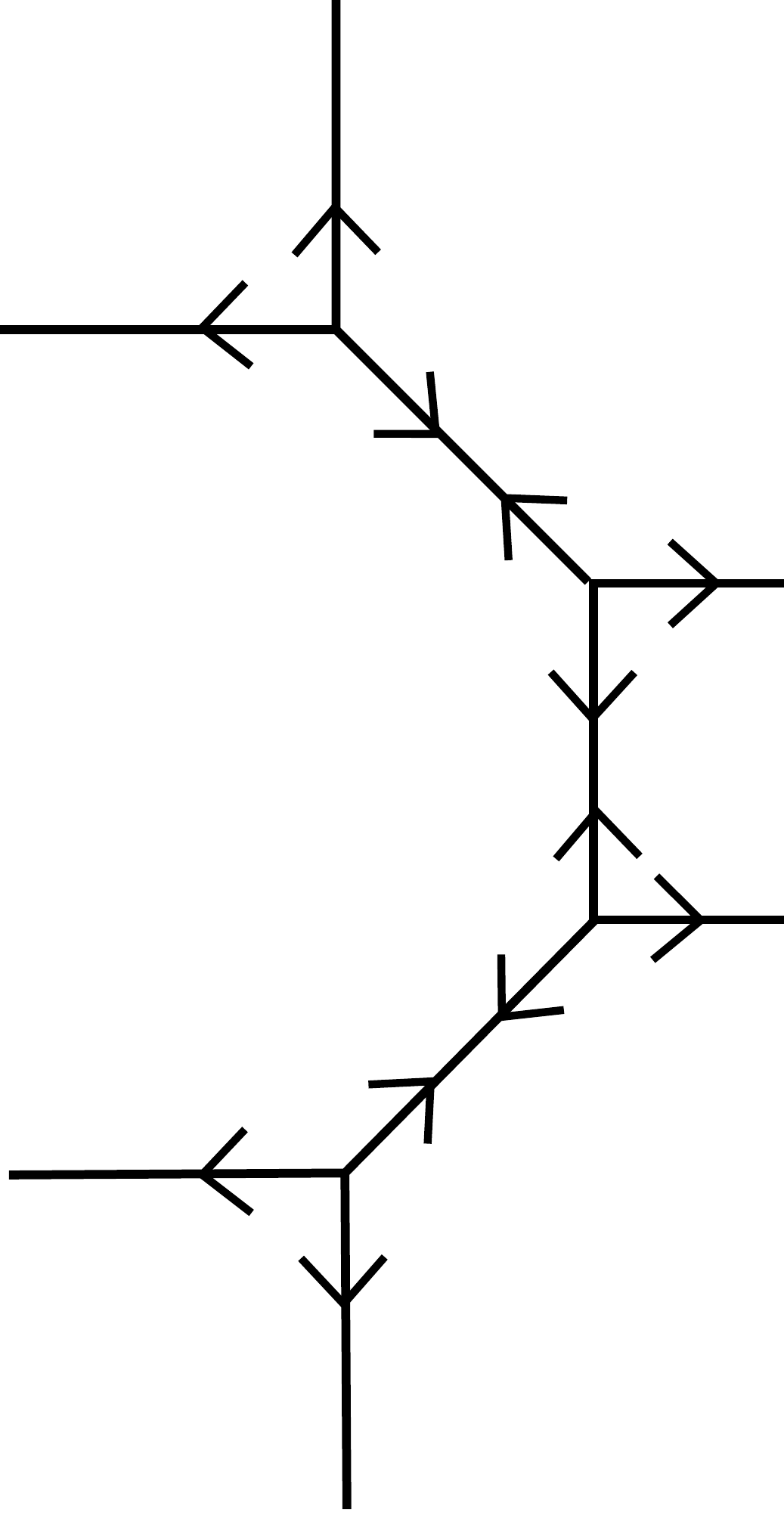}
\end{subfigure}
 \caption{The Calabi-Yau threefolds $X_1$ and $X_2$, respectively, with $T$-weights labeled} \label{weights}
\end{figure}

\begin{figure}[bth] 
\begin{subfigure}{1in}
\scalebox{.3}{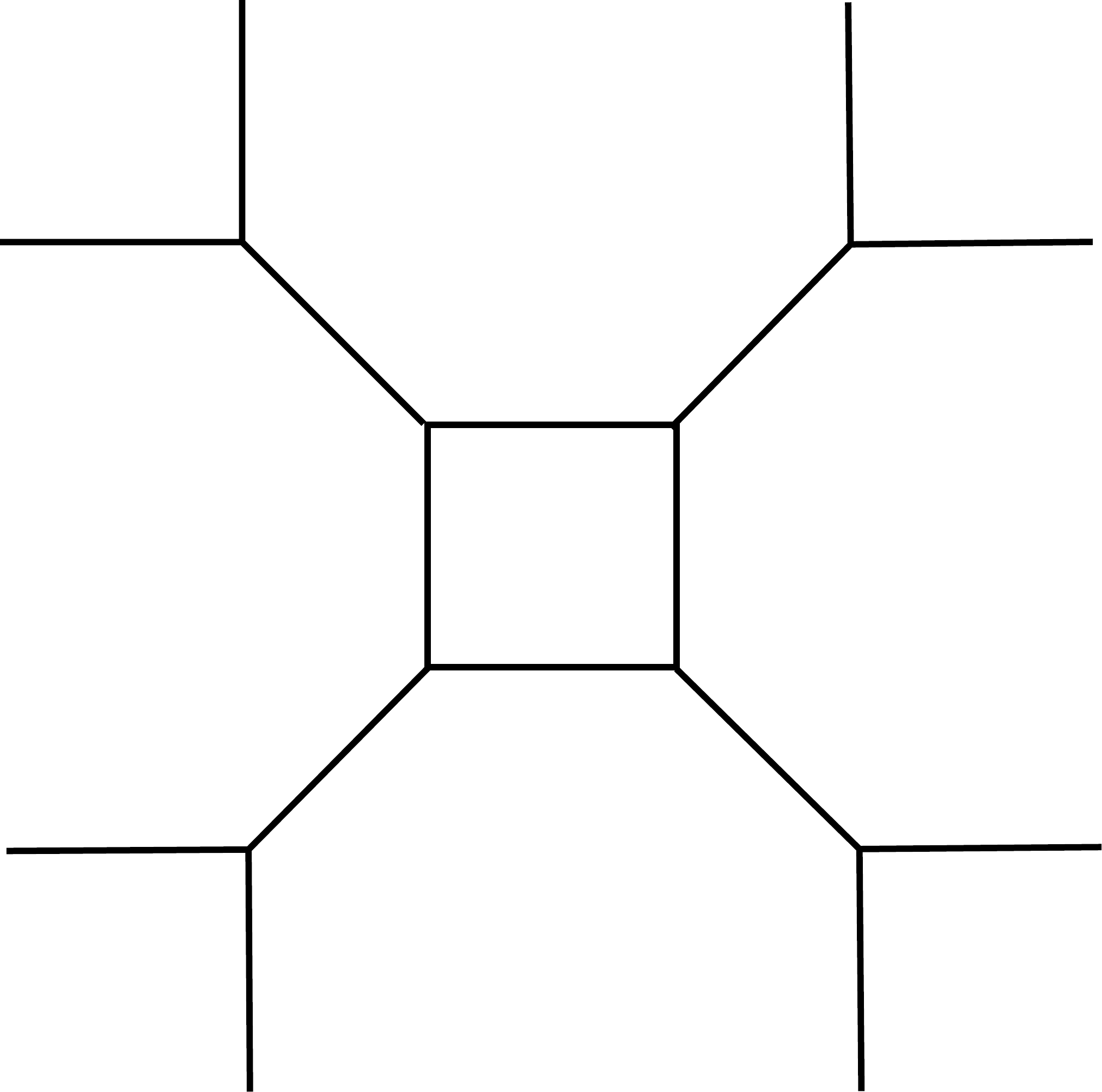}
\end{subfigure}
\hspace{2in} 
\begin{subfigure}{1in}
\scalebox{.3}{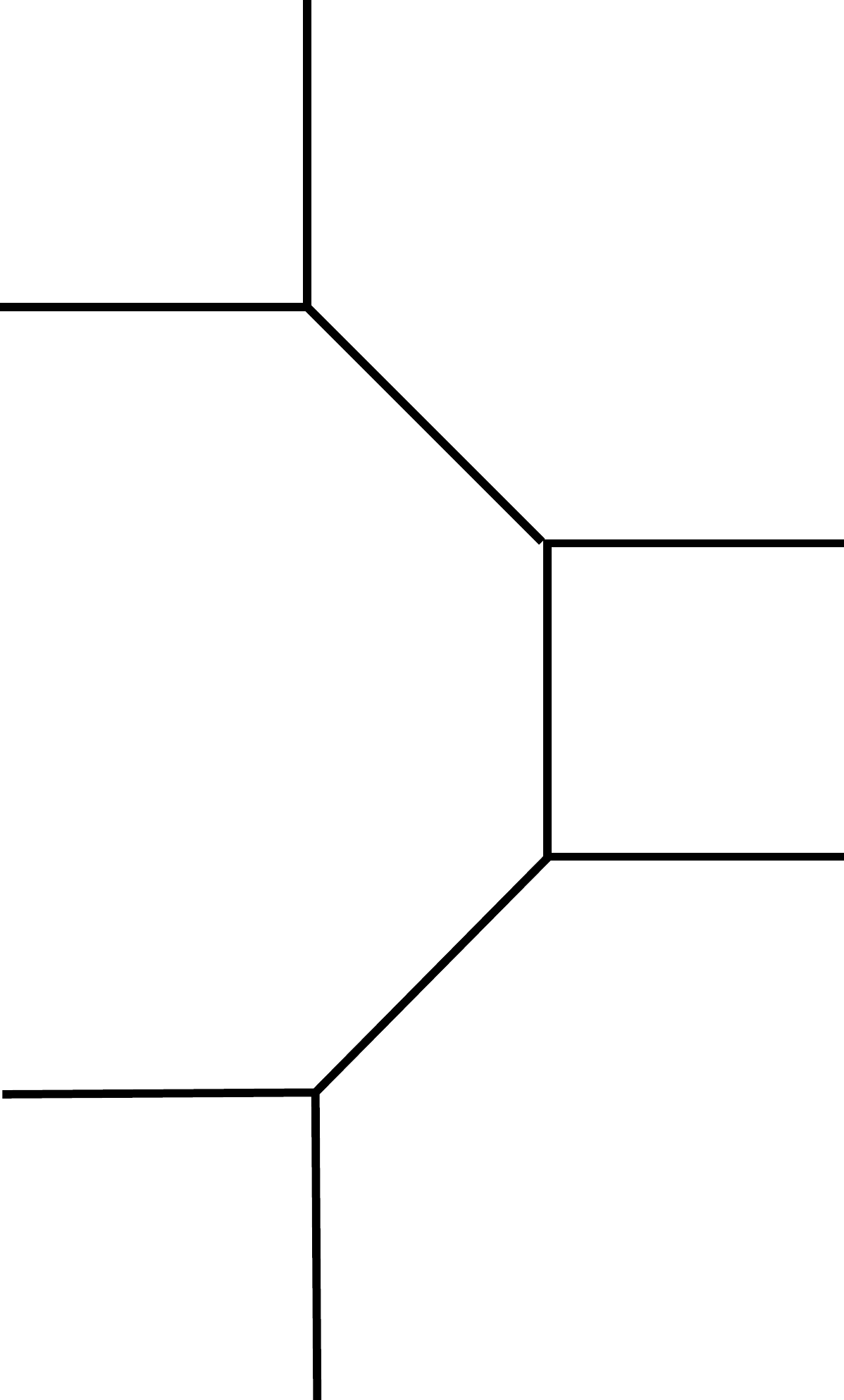}
\end{subfigure}
 \caption{The Calabi-Yau threefolds $X_1$ and $X_2$, respectively, with K\"{a}hler parameters labeled} \label{kahs}
\end{figure}

\subsection{Equality of limits}

Let $X$ denote either of $X_1$ or $X_2$. As before, let $\tilde{T}$ be the double cover of $T$ on which the character $(t_1t_2t_3)^{1/2}$ is defined and let $A\subset T$ be the Calabi-Yau subtorus $\ker(\kappa)$; then $\tilde{T}$ acts on $X$ via the map $\tilde{T}\to T$. 
For these two geometries, the possible non-compact directions of $X$ are those pointing vertically or horizontally in Figure \ref{weights}; these are all of the form $$t_1^i(t_1t_2t_3)^{j/2},\ t_2^{i}(t_1t_2t_3)^{j/2}.$$

Let $M$ be a component $DT(X,\beta,n)$ of $DT(X)$. By (\ref{rigidity}), the function $\chi(M,\tilde{\mathcal{O}}^{\vir})$ is $A$-balanced. 

Let $\sigma(r_1,r_2,r_3)$ be the one-parameter subgroup $\C^{\times}\to A$ given by $$z\mapsto (z^{r_1},z^{r_2},z^{r_3}).$$  The weight $t_1^i(t_1t_2t_3)^{j/2}$ is attracting with respect to $\sigma(r_1,r_2,r_3)$ if $i\cdot r_1>0$ and repelling if $i\cdot r_1<0$, while the weight $t_2^i(t_1t_2t_3)^{j/2}$ is attracting with respect to $\sigma$ if $i\cdot r_2>0$ and repelling if $i\cdot r_2<0$. We conclude that the attracting/repelling behavior of the noncompact directions of $X$ depends only on the signs of $r_1$ and  $r_2$.  Hence, for $M$-generic $\sigma(r_1,r_2,r_3),$  Theorem \ref{independence} implies that the value of the limit
$$\chi(M,\tilde{\mathcal{O}}^{\vir})^{\sigma(r_1,r_2,r_3)}$$ depends only on the signs of $r_1$ and $r_2$. 

The same result holds when $DT(X)$ is replaced by the Hilbert scheme of points $DT_0(X)$; this can also be seen directly from Nekrasov's formula, or its corollary (\ref{deg0}). We conclude the following.

\begin{corollary}\label{eqlim}
For $r_3\gg r_2>0\gg r_1$ and $s_2\gg 0>s_1\gg s_3,$ we have \begin{align}\label{equ}Z'_{DT}(X)^{\sigma(r_1,r_2,r_3)}=Z'_{DT}(X)^{\sigma(s_1,s_2,s_3)} \in \Q[\Gamma, u^{\beta}].
\end{align}
\end{corollary}

Here, the integral domain $\Q[\Gamma,u^{\beta}]$ consists of power series in the K\"{a}hler variables with coefficients in $\Gamma$. In Section \ref{eval} we will also write certain elements of $\Gamma$ as rational functions in $q^{1/2}$ and $t^{1/2};$ the rational functions we write can be identified with elements of $\Gamma$ by expanding in positive powers of $q$ and $t$.


While sufficient to establish our results for Hilbert schemes, we remark that our notion of non-compact weight does not produce the strongest expected results of the form of Corollary \ref{eqlim}. Namely, if one assumes the K-theoretic DT/stable pairs correspondence (\cite[(16)]{NO}), then non-compact weights should be replaced by weights corresponding to directions in which curves in $X$ may escape to infinity. Such a reformulation would imply stronger versions of Corollary \ref{eqlim} for some threefolds not considered in this paper. For example, the space of effective curves in $\mathcal{K}_{\P^1\times\P^1}$ is compact; the K-theoretic DT/stable pairs correspondence would then imply that, for generic $\sigma$, the limit $Z'_{DT}(\mathcal{K}_{\P^1\times\P^1})^{\sigma}$ does not depend on $\sigma$.

\subsection{Evaluation of limits}\label{eval}

To apply Corollary \ref{eqlim}, we evaluate the limits (\ref{equ}).

We begin by dividing the toric 1-skeleton of $X_1$ into 8 trivalent vertices and 8 bounded edges. Propositions \ref{edge1}, \ref{edge2}, and \ref{indref} enable us to evaluate the corresponding limits of (\ref{redpart}). Choose the counter-clockwise orientation for the 1-skeleton of $\Delta(X_1)$. This means that the edges corresponding to the first, second, and third coordinates in a toric chart centered at a vertex lie in a counter-clockwise direction around the vertex.  Let \begin{align}\label{2d} \lambda_1,\ \lambda_2,\ \mu_1,\ \mu_2,\ \kappa_1,\ \kappa_2,\ \kappa_3,\ \kappa_4
\end{align} be two-dimensional partitions assigned to the bounded edges as indicated in Figure \ref{fig:CYedge1}. 

\begin{figure}[htb]
 \scalebox{.25}{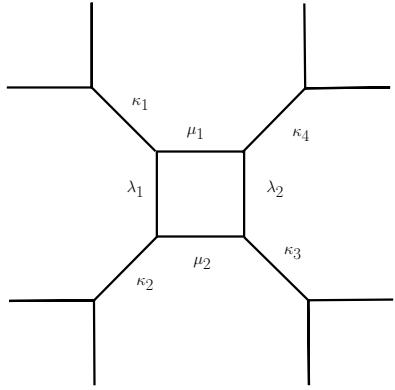}
\caption{The Calabi-Yau threefold $X_1$ with two-dimensional partitions assigned to edges}
\label{fig:CYedge1}
\end{figure}

Then, for $\sigma(r_1,r_2,r_3)$ with $r_3\gg r_2>0\gg r_1$, the corresponding contribution to the limit $Z'_{DT}(X_1)^{\sigma(r_1,r_2,r_3)}$ is the product of sixteen terms, one for each vertex, and one for each edge. Reading the diagram left to right, then top to bottom, the the edge terms are

\begin{align*}
 &m_1^{|\kappa_1|} t^{\frac{||\kappa_1^t ||^2}{2}}q^{\frac{||\kappa_1||^2}{2}}, \\ & m_4^{|\kappa_4|}t^{\frac{||\kappa_4||^2}{2}}q^{\frac{||\kappa_4^t||^2}{2}}, \\
 &v^{|\mu_1|} q^{||\mu_1||^2}, \\
 &u^{|\lambda_1|} t^{\frac{||\lambda_1||^2+|\lambda_1|}{2}}q^{\frac{||\lambda_1||^2-|\lambda_1|}{2}},\\ &u^{|\lambda_2|} t^{\frac{||\lambda_2||^2-|\lambda_2|}{2}}q^{\frac{||\lambda_2||^2+|\lambda_2|}{2}},\\
  &v^{|\mu_2|}t^{||\mu_2||^2},\\
 &m_2^{|\kappa_2|} t^{\frac{||\kappa_2^t||^2}{2}}q^{\frac{||\kappa_2||^2}{2}}, \\
 & m_3^{|\kappa_3|}t^{\frac{||\kappa_3 ||^2}{2}}q^{\frac{||\kappa_3^t||^2}{2}}, \\
\end{align*}

while the vertex terms are
\begin{align*}
&C(\kappa_1^t,\emptyset,\emptyset)(q,t)=q^{-\frac{||\kappa_1||^2}{2}}s_{\kappa_1^t}(q^{-\rho}), \\
 &C(\emptyset, \kappa_4^t,\emptyset)(t,q)=q^{-\frac{||\kappa_4^t||^2}{2}}s_{\kappa_4}(q^{-\rho}), \\
 &C(\kappa_1,\lambda_1,\mu_1^{t})(t,q)=\frac{ t^{\frac{-||\kappa_1^t||^2}{2}}q^{\frac{-||\lambda_1||^2}{2}}}{\displaystyle \prod_{\square\in\mu_1^t} 1-t^{a(\square)+1}q^{l(\square)} } \sum_{\eta_1}\Big(\frac{q}{t}\Big)^{\frac{|\eta_1|}{2}}s_{\kappa_1/\eta_1}(t^{-\rho}q^{-\mu_1})s_{\lambda_1^t/\eta_1}(q^{-\rho}t^{-\mu_1^t}),\\
 &C(\lambda_2^t,\kappa_4,\mu_1)(q,t)= \frac{ q^{\frac{-||\lambda_2||^2}{2}}t^{\frac{-||\kappa_4||^2}{2}}}{\displaystyle \prod_{\square\in\mu_1} 1-q^{a(\square)+1}t^{l(\square)} } \sum_{\eta_4}\Big(\frac{t}{q}\Big)^{\frac{|\eta_4|}{2}}s_{\lambda_2^t/\eta_4}(q^{-\rho}t^{-\mu_1^t})s_{\kappa_4^t/\eta_4}(t^{-\rho}q^{-\mu_1}). \\
 &C(\lambda_1^t,\kappa_2,\mu_2)(t,q)= \frac{ t^{\frac{-||\lambda_1||^2}{2}}q^{\frac{-||\kappa_2||^2}{2}}}{\displaystyle \prod_{\square\in\mu_2} 1-t^{a(\square)+1}q^{l(\square)} }\sum_{\eta_2}\Big(\frac{q}{t}\Big)^{\frac{|\eta_2|}{2}}s_{\lambda_1^t/\eta_2}(t^{-\rho}q^{-\mu_2^t})s_{\kappa_2^t/\eta_2}(q^{-\rho}t^{-\mu_2}), \\
 &C(\kappa_3,\lambda_2,\mu_2^t)(q,t)=\frac{ q^{\frac{-||\kappa_3^t||^2}{2}}t^{\frac{-||\lambda_2||^2}{2}}}{\displaystyle \prod_{\square\in\mu^t_2} 1-q^{a(\square)+1}t^{l(\square)} } \sum_{\eta_3}\Big(\frac{t}{q}\Big)^{\frac{|\eta_3|}{2}}s_{\kappa_3/\eta_3}(q^{-\rho}t^{-\mu_2})s_{\lambda_2^t/\eta_3}(t^{-\rho}q^{-\mu^t_2}), \\
 &C(\emptyset, \kappa_2^t, \emptyset)(q,t)=t^{-\frac{||\kappa_2^t||^2}{2}}s_{\kappa_2}(t^{-\rho}) ,\\
 &C(\kappa_3^t,\emptyset,\emptyset)(t,q)=t^{-\frac{||\kappa_3||^2}{2}}s_{\kappa_3^t}(t^{-\rho}). \\
\end{align*}

The limit  $Z'_{DT}(X_1)^{\sigma(r_1,r_2,r_3)}$ is the sum over all possible choices of the two-dimensional partitions (\ref{2d}) of the product of these 16 terms. As examples, we indicate how the contributions from the edge assigned $\mu_2$ in Figure \ref{fig:CYedge1} is determined from Proposition \ref{edge2} and how the contributions terms from the two vertices on this edge are determined from Proposition \ref{indref}.

\begin{figure}[htb]
\includegraphics[height=1.7in]{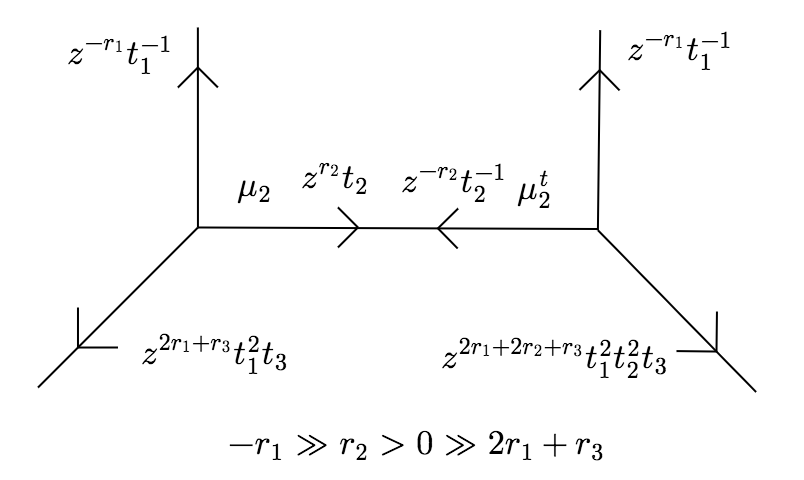}
\caption{The edge of $\Delta(X_1)$ assigned $\mu_2$ with weights at fixed points}
\label{fig:weightededge}
\end{figure}

The weights at the fixed points of the edge assigned $\mu_2$ are recalled in Figure \ref{fig:weightededge}. The corresponding torus invariant curve has normal bundle $\mathcal{O}\oplus \mathcal{O}(-2)$. In particular, at the fixed point on the left of the diagram, the slope $\sigma$ scales the torus-weight in the direction of the torus invariant curve by $z^{r_2}$, scales the fiber of $\mathcal{O}$ by $z^{-r_1}$ and scales the fiber of $\mathcal{O}(-2)$ by $2r_1+r_3$. As indicated in the diagram, we have $$-r_1\gg r_2 > 0 \gg 2r_1+r_3.$$ Comparing with Figure \ref{minus2}, we see that the contribution of this edge to the limit is given by the second row of the table in Proposition \ref{edge2}. So, the edge contributes $v^{|\mu_2|}t^{||\mu_2||^2}.$

\begin{figure}[bth] 
\begin{subfigure}{1in}
\includegraphics[height=1.7in]{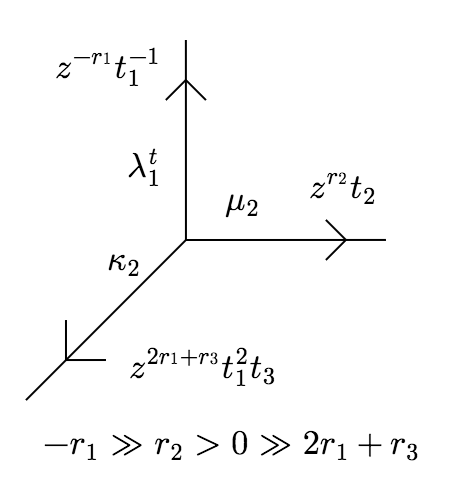}
\end{subfigure}
\hspace{1in} 
\begin{subfigure}{1.5in}
\includegraphics[height=1.7in]{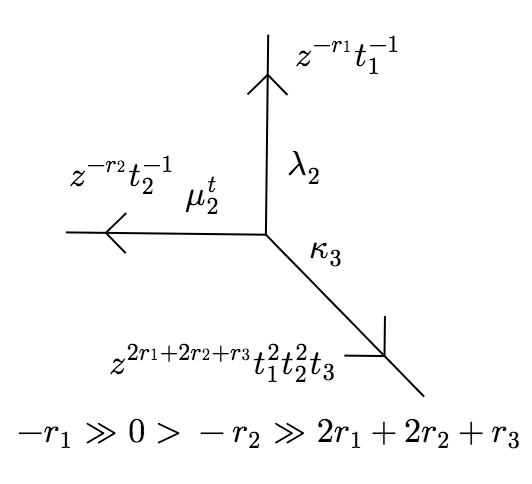}
\end{subfigure}
 \caption{Two vertices of $\Delta(X_1)$ with weights and infinite legs labelled} \label{fig:weightedvertex}
\end{figure}

The vertices are determined similarly. In Figure \ref{fig:weightedvertex}, the weights of the vertices of the edge assigned $\mu_2$ are labelled, along with the two-dimensional partitions describing the infinite legs. Recall from (\ref{comp1}) that, due to the choice of orientation, the two-dimensional partitions corresponding to the two legs (one for each vertex) lying along the bounded edge are related by transposition. Again, the exponents with which $\sigma$ scales the weights at the fixed point on the left of Figure \ref{fig:weightedvertex} are (in an order respecting the orientation) $-r_1, 2r_1+r_3$ and $r_2$. Recall that $$-r_1\gg r_2 > 0 \gg 2r_1+r_3.$$ So, the contribution of this vertex is given by the first row of the table in Proposition \ref{indref} and is therefore $$C(\lambda_1^t, \kappa_2,\mu_2)(t,q).$$ At the fixed point on the right, the slope $\sigma$ scales the weights with exponents $2r_1+r_2+r_3, -r_1,$ and $-r_2$. Again, this order of the exponents is compatible with the orientation. As $$-r_1\gg 0 > -r_2 \gg 2r_1+r_2+r_3,$$ the contribution of this vertex is given by the fourth row of the table in Proposition \ref{indref} and is therefore $$C(\kappa_3,\lambda_2,\mu_2^t)(q,t).$$

The remaining thirteen vertex and edge contributions are determined similarly. We now evaluate $Z'_{DT}(X_1)^{\sigma(r_1,r_2,r_3)}$. Note that there are many cancellations between the edge terms and prefactors in the vertex terms. 

To evaluate the sum over all choices (\ref{2d}), we use the techniques of \cite{IKP} for computing with the topological vertex, which also work for the refined topological vertex. We make repeated use of the following identities involving skew Schur polynomials: \begin{align}\label{schur1}\sum_{\lambda} u^{|\lambda|} s_{\lambda/\eta_1}\big((x_i) \big)s_{\lambda/\eta_2}\big((y_j) \big)=\frac{1}{\displaystyle\prod_{i,j}(1-ux_iy_j)} \sum_{\lambda} u^{|\lambda|} s_{\eta_1/\lambda}\big((uy_j)\big)s_{\eta_2/\lambda}\big((ux_i)\big),\end{align}
\begin{align}\label{schur2}\sum_{\lambda} u^{|\lambda|} s_{\lambda/\eta_1}\big((x_i) \big)s_{\lambda^t/\eta_2}\big((y_j) \big)=\displaystyle\prod_{i,j}(1+ux_iy_j) \sum_{\lambda} u^{|\lambda|} s_{\eta^t_1/\lambda^t}\big((uy_j)\big)s_{\eta_2^t/\lambda}\big((ux_i)\big).\end{align}
these particular identities may be found, for example, in \cite{Taka}.

Adopting notation from \cite{IKP},  for two sequences $x=(x_0,x_1,\cdots), y=(y_0,y_1,\cdots)$ and a formal variable $m$, we set $$\{x,y\}_m=\prod_{i,j\geq 0}(1+mx_iy_j).$$ 

First, use (\ref{schur2}) to sum over the partitions $\kappa_i$ labeling the outer edges, yielding

$$\sum_{\kappa_1}m_1^{|\kappa_1|}s_{\kappa_1^t}(q^{-\rho})s_{\kappa_1/\eta_1}(t^{-\rho}q^{-\mu_1})=\{t^{-\rho}q^{-\mu_1},q^{-\rho}\}_{m_1}\cdot s_{\eta_1^t}(m_1q^{-\rho}),$$

$$\sum_{\kappa_2}m_2^{|\kappa_2|}s_{\kappa_2}(t^{-\rho})s_{\kappa_2^t/\eta_2}(q^{-\rho}t^{-\mu_2})=\{q^{-\rho}t^{-\mu_2},t^{-\rho}\}_{m_2}\cdot s_{\eta_2^t}(m_2t^{-\rho}),$$

$$\sum_{\kappa_3}m_3^{|\kappa_3|}s_{\kappa_3}(t^{-\rho})s_{\kappa_3^t/\eta_3}(q^{-\rho}t^{-\mu_2})=\{q^{-\rho}t^{-\mu_2},t^{-\rho}\}_{m_3}\cdot s_{\eta_3^t}(m_3t^{-\rho}),$$

and

$$\sum_{\kappa_4}m_4^{|\kappa_4|}s_{\kappa_4}(q^{-\rho})s_{\kappa_4^t/\eta_4}(t^{-\rho}q^{-\mu_1})=\{t^{-\rho}q^{-\mu_1},q^{-\rho}\}_{m_4}\cdot s_{\eta_4^t}(m_4q^{-\rho}).$$ 

As the skew-Schur polynomial $s_{\lambda/\mu}$ is homogenous of degree $|\lambda|-|\mu|$, we incorporate the $(t/q)^{|\lambda_1|/2}$ and $(q/t)^{|\lambda_2|/2}$ factors arising from the edge contributions and the $(t/q)^{|\eta_1|/2}$ and $(q/t)^{|\eta_3|/2}$ factors arising from vertex contributions into the arguments of the skew-Schur polynomials:

$$ \Big(\frac{t}{q}\Big)^{\frac{\lambda_1}{2}} \Big(\frac{q}{t}\Big)^{\frac{\eta_1}{2}}s_{\lambda_1^t/\eta_1}(q^{-\rho}t^{-\mu_1^t})=s_{\lambda_1^t/\eta_1}\Big(\sqrt{\frac{t}{q}}q^{-\rho}t^{-\mu_1^t}\Big), $$

$$\Big(\frac{q}{t}\Big)^{\frac{\lambda_2}{2}} \Big(\frac{t}{q}\Big)^{\frac{\eta_3}{2}}s_{\lambda_2^t/\eta_3}(t^{-\rho}q^{-\mu_2^t})=s_{\lambda_2^t/\eta_3}\Big(\sqrt{\frac{q}{t}}t^{-\rho}q^{-\mu_2^t}\Big).$$

We use (\ref{schur1}) to sum over the partitions $\lambda_1$ and $\lambda_2$, yielding

\begin{align*} &\sum_{\lambda_1} u^{|\lambda_1|}s_{\lambda_1^{t}/\eta_1}\Big(\sqrt{\frac{t}{q}}q^{-\rho}t^{-\mu_1^t}\Big)s_{\lambda_1^t/\eta_2}(t^{-\rho}q^{-\mu_2^t})\\&=\frac{\displaystyle\sum_{\lambda_1}u^{|\lambda_1|}s_{\eta_2/\lambda_1^t}\Big(u\sqrt{\frac{t}{q}}q^{-\rho}t^{-\mu_1^t}\Big)s_{\eta_1/\lambda_1^t}(ut^{-\rho}q^{-\mu_2^t})}{\Big\{\sqrt{\frac{t}{q}}q^{-\rho}t^{-\mu_1^t},t^{-\rho}q^{-\mu_2^t}\Big\}_{-u}},\end{align*}
and 

\begin{align*}&\sum_{\lambda_2} u^{|\lambda_2|}s_{\lambda_2^{t}/\eta_3}\Big(\sqrt{\frac{q}{t}}t^{-\rho}q^{-\mu_2^t}\Big)s_{\lambda_2^t/\eta_4}(q^{-\rho}t^{-\mu_1^t})\\& =\frac{\displaystyle\sum_{\lambda_2}u^{|\lambda_2|} s_{\eta_4/\lambda_2^t}\Big(u\sqrt{\frac{q}{t}}t^{-\rho}q^{-\mu_2^t}\Big)s_{\eta_3/\lambda_2^t}(uq^{-\rho}t^{-\mu_1^t})}{\Big\{\sqrt{\frac{q}{t}}t^{-\rho}q^{-\mu_2^t},q^{-\rho}t^{-\mu_1^t}\Big\}_{-u}}.\end{align*}

We then use (\ref{schur2}) sum over the ancillary partitions $\eta_i$ associated to each vertex. We have 

\begin{align*} & \sum_{\eta_1} s_{\eta_1^t}(m_1q^{-\rho})s_{\eta_1/\lambda_1^t}(ut^{-\rho}q^{-\mu_2^t})=\{t^{-\rho}q^{-\mu_2^t},q^{-\rho}\}_{m_1u}\cdot s_{\lambda_1}(m_1q^{-\rho}), \\
&\sum_{\eta_2} s_{\eta_2^t}\Big (m_2\sqrt{\frac{q}{t}}t^{-\rho}\Big )s_{\eta_2/\lambda_1^t}\Big (u\sqrt{\frac{t}{q}}q^{-\rho}t^{-\mu_1^t}\Big )=\{q^{-\rho}t^{-\mu_1^t},t^{-\rho}\}_{m_2u}\cdot s_{\lambda_1}\Big(m_2\sqrt{\frac{q}{t}}t^{-\rho}\Big ), \\
&\sum_{\eta_3} s_{\eta_3^t}(m_3t^{-\rho})s_{\eta_3/\lambda_2^t}(uq^{-\rho}t^{-\mu_1^t})=\{q^{-\rho}t^{-\mu_1^t},t^{-\rho}\}_{m_3u}\cdot s_{\lambda_2}(m_3t^{-\rho}), \\
&\sum_{\eta_4} s_{\eta_4^t}\Big (m_4\sqrt{\frac{t}{q}}q^{-\rho}\Big )s_{\eta_4/\lambda_2^t}\Big (u\sqrt{\frac{q}{t}}t^{-\rho}q^{-\mu_2^t}\Big )=\{t^{-\rho}q^{-\mu_2^t},q^{-\rho}\}_{m_4u}\cdot s_{\lambda_2}\Big(m_4\sqrt{\frac{t}{q}}q^{-\rho}\Big).\end{align*}

Finally, we again use (\ref{schur1}) to sum over the $\lambda_i$ yielding $$\sum_{\lambda_1}u^{|\lambda_1|}s_{\lambda_1}(m_1q^{-\rho})s_{\lambda_1}\Big (m_2\sqrt{\frac{q}{t}}t^{-\rho}\Big )=\frac{1}{\Big \{\sqrt{\frac{q}{t}}t^{-\rho},q^{-\rho}\Big \}_{-um_1m_2}}$$ and $$\sum_{\lambda_2}u^{|\lambda_2|}s_{\lambda_2}(m_3t^{-\rho})s_{\lambda_2}\Big (m_4\sqrt{\frac{t}{q}}q^{-\rho}\Big )=\frac{1}{\Big \{\sqrt{\frac{t}{q}}q^{-\rho},t^{-\rho}\Big \}_{-um_3m_4}}.$$

We conclude that for $r_3\gg r_2>0\gg r_1,$ the limit $Z'_{DT}(X_1)^{\sigma(r_1,r_2,r_3)}$ is \begin{align}\label{firstlimit} \sum_{\mu_1,\mu_2}v^{|\mu_1|+|\mu_2|} \frac{q^{||\mu_1||^2}t^{||\mu_2||^2}}{V_T}\frac{V_N}{V_D} \end{align}

where \begin{align*}
V_T =\displaystyle \prod_{\square\in\mu_1} (1-t^{l(\square)}q^{a(\square)+1})(1-t^{l(\square)+1}q^{a(\square)} ) \displaystyle \prod_{\square\in\mu_2} (1-q^{l(\square)}t^{a(\square)+1}) (1-q^{l(\square)+1}t^{a(\square)}),
\end{align*} 
 \begin{align*} V_N =& \{t^{-\rho}q^{-\mu_1},q^{-\rho}\}_{m_1}\{q^{-\rho}t^{-\mu_2},t^{-\rho}\}_{m_2}\{q^{-\rho}t^{-\mu_2},t^{-\rho}\}_{m_3}\{t^{-\rho}q^{-\mu_1},q^{-\rho}\}_{m_4} \\ &\cdot\{t^{-\rho}q^{-\mu_2^t},q^{-\rho}\}_{m_1u}\{q^{-\rho}t^{-\mu_1^t},t^{-\rho}\}_{m_2u}\{q^{-\rho}t^{-\mu_1^t},t^{-\rho}\}_{m_3u}\{t^{-\rho}q^{-\mu_2^t},q^{-\rho}\}_{m_4u}, \end{align*}

\begin{align*} V_D=& \Big \{\sqrt{\frac{t}{q}}q^{-\rho}t^{-\mu_1^t},t^{-\rho}q^{-\mu_2^t}\Big \}_{-u}\Big \{\sqrt{\frac{q}{t}}t^{-\rho}q^{-\mu_2^t},q^{-\rho}t^{-\mu_1^t}\Big\}_{-u}\\ &\cdot\Big \{\sqrt{\frac{q}{t}}t^{-\rho},q^{-\rho}\Big \}_{-um_1m_2}\Big \{\sqrt{\frac{t}{q}}q^{-\rho},t^{-\rho}\Big \}_{-um_3m_4}.\end{align*}

Here, the factors of $V_T$ are expanded in positive powers of $t$ and $q$, and the factors of $V_D$ are expanded in positive powers of the Kahler variables. It can be checked directly from definitions that \begin{align}\label{btop} 
&\{q^{-\rho}t^{-\lambda},t^{-\rho}\}_{m}=\prod_{\square\in\lambda}(1+m\sqrt{\frac{q}{t}}q^{b_1}t^{-b_2})\prod_{i,j\geq 0}(1+mq^{i+1/2}t^{j+1/2}).
\end{align}
All infinite products appearing in (\ref{firstlimit}) are well-defined as power series in the K\"{a}hler variables with coefficients in $\Gamma$.


The limit $Z'_{DT}(X_1)^{\sigma(s_1,s_2,s_3)}$ for $s_2\gg 0>s_1\gg s_3$ can be computed by the same procedure, yielding
\begin{align}\label{secondlimit} \sum_{\lambda_1,\lambda_2} u^{|\lambda_1|+|\lambda_2|}\frac{t^{||\lambda_1||^2}q^{||\lambda_2||^2}}{U_T} \frac{U_N}{U_D}, \end{align} where 

\begin{align*}
U_T = \displaystyle \prod_{\square\in\lambda_1} (1-q^{l(\square)}t^{a(\square)+1})(1-q^{l(\square)+1}t^{a(\square)} ) \displaystyle \prod_{\square\in\lambda_2} (1-t^{l(\square)}q^{a(\square)+1}) (1-t^{l(\square)+1}q^{a(\square)}),
\end{align*} 
\begin{align*}
U_N=& \{q^{-\rho}t^{-\lambda_1},t^{-\rho}\}_{m_1}\{q^{-\rho}t^{-\lambda_1},t^{-\rho}\}_{m_2}\{t^{-\rho}q^{-\lambda_2},q^{-\rho}\}_{m_3}\{t^{-\rho}q^{-\lambda_2},q^{-\rho}\}_{m_4}\\&\cdot \{q^{-\rho}t^{-\lambda_2^t},t^{-\rho}\}_{m_1v}\{q^{-\rho}t^{-\lambda_2^t},t^{-\rho}\}_{m_2v}\{t^{-\rho}q^{-\lambda_1^t},q^{-\rho}\}_{m_3v}\{t^{-\rho}q^{-\lambda_1^t},q^{-\rho}\}_{m_4v},
\end{align*} and
\begin{align*}
U_D = & \Big\{\sqrt{\frac{q}{t}}t^{-\rho}q^{-\lambda_1^t},q^{-\rho}t^{-\lambda_2^t}\Big\}_{-v}\Big\{\sqrt{\frac{t}{q}}q^{-\rho}t^{-\lambda_2^t},t^{-\rho}q^{-\lambda_1^t}\Big\}_{-v} \\& \cdot \Big\{\sqrt{\frac{t}{q}}q^{-\rho},t^{-\rho}\Big\}_{-vm_1m_4}\Big\{\sqrt{\frac{q}{t}}t^{-\rho},q^{-\rho}\Big\}_{-vm_2m_3}.
\end{align*}

Again, all infinite products appearing in (\ref{secondlimit}) are well-defined (as power series in the K\"{a}hler variables with coefficients in $\Gamma$.)

Corollary \ref{eqlim} implies that (\ref{firstlimit}) and (\ref{secondlimit}) are equal. We perform specializations and substitutions to obtain Theorem \ref{symmetry}. 

First, we perform a specialization under which (\ref{firstlimit}) and (\ref{secondlimit}) each become sums over a single partition. Specialize \begin{align}\label{spec} m_4=-\sqrt{\frac{q}{t}}.\end{align} By inspection of the series (\ref{firstlimit}) and (\ref{secondlimit}), this specialization (\ref{spec}) is valid; that is, under this specialization, both (\ref{firstlimit}) and (\ref{secondlimit}) are well-defined as elements of $\Q[\Gamma][[u,v,m_1,m_2,m_3]].$ By (\ref{btop}), when $\lambda$ is nonempty, the infinite product $$\{t^{-\rho}q^{-\lambda},q^{-\rho}\}_{m}$$ contains $$(1+m\sqrt{\frac{t}{q}})$$ as a factor. So, for nonempty $\lambda$ we have \begin{align}\label{noedge}\{t^{-\rho}q^{-\lambda},q^{-\rho}\}_{-\sqrt{\frac{q}{t}}}=0.\end{align}

After implementing (\ref{spec}), we see from (\ref{noedge}) that only pairs $(\mu_1,\mu_2)$ where $\mu_1$ is the empty partition contribute to the sum (\ref{firstlimit}). 

Now, substitute \begin{align}\label{subs} m_1=-m_1'\sqrt{\frac{t}{q}},\ m_2=-m_2'\sqrt{\frac{t}{q}},\ m_3=-m_3'\sqrt{\frac{t}{q}},\ u=u'\frac{q}{t},\ v=v'\frac{q}{t}.\end{align}

Having implemented (\ref{spec}) and (\ref{subs}), we apply (\ref{btop}) to the products $V_T$ and $V_D$ appearing in each summand of (\ref{firstlimit}). After cancelling like terms and keeping careful track of the remaining factors, we obtain that under the substitutions (\ref{spec}, \ref{subs}), the sum (\ref{firstlimit}) becomes
\begin{align}& \bigg( \sum_{\mu_2}  (v')^{|\mu_2|}q^{|\mu_2|}t^{||\mu_2||^2-|\mu_2|}\prod_{\square\in\mu_2}\frac{(1-m_2'q^{b_1}t^{-b_2})(1-m_3'q^{b_1}t^{-b_2})(1-u'm_1'q^{-b_1}t^{b_2})}{(1-q^lt^{a+1})(1-q^{l+1}t^a)(1-u'q^{-b_1}t^{b_2})}\bigg) \nonumber \\ & \cdot \prod_{i,j} \frac{(1-u'm_1'q^{i+1}t^{j}) (1-u'm_2'q^{i+1}t^{j}) }{(1-u'q^{i+1}t^{j})(1-u'm_1'm_2'q^{i+1}t^{j})} \nonumber \\& \cdot \prod_{i,j} (1-q^{i+1}t^{j})(1-m_1'q^{i}t^{j+1})(1-m_2'q^{i}t^{j+1})(1-m_3'q^{i}t^{j+1}). \label{firstlimitsub1}
\end{align}
By a similar computation, under the substitutions (\ref{subs}), the series (\ref{secondlimit}) becomes
\begin{align}
&\bigg(\sum_{\lambda_1}  (u')^{|\lambda_1|}q^{|\lambda_1|}t^{||\lambda_1||^2-|\lambda_1|} \prod_{\square\in\lambda_1} \frac{(1-m_1'q^{b_1}t^{-b_2})(1-m_2'q^{b_1}t^{-b_2})  (1-v'm_3'q^{-b_1}t^{b_2}) }{(1-q^lt^{a+1})(1-q^{l+1}t^a) (1-v'q^{-b_1}t^{b_2}) } \bigg) \nonumber \\
&\cdot \prod_{i,j} \frac{ (1-v'm_2'q^{i+1}t^{j})(1-v'm_3'q^{i+1}t^{j})}{(1-v'q^{i+1}t^{j})(1-v'm_2'm_3'q^{i+1}t^{j})} \nonumber \\ & \cdot \prod_{i,j} (1-q^{i+1}t^{j})(1-m_1'q^{i}t^{j+1})(1-m_2'q^{i}t^{j+1})(1-m_3'q^{i}t^{j+1}). \label{secondlimitsub1}
\end{align}

Note that  \begin{align}\prod_{\square\in \lambda}\frac{1}{(1-q^{l+1}t^a)(1-u'q^{-b_1}t^{b_2})}&=\prod_{\square\in\lambda}\frac{q^{-l-1}t^{-a}q^{b_1}t^{-b_2}}{(1-q^{-l-1}t^{-a})(u'-q^{b_1}t^{-b_2})}\nonumber \\ &=\frac{q^{-|\lambda|}t^{-||\lambda||^2+|\lambda|}}{\prod_{\square\in\lambda}(1-q^{-l-1}t^{-a})(u'-q^{b_1}t^{-b_2})}.\label{swap}\end{align}

After canceling the factors common to (\ref{firstlimitsub1}) and (\ref{secondlimitsub1}), redistributing terms, and using (\ref{swap}), we conclude that the expression

\begin{align}
\label{int1}& \bigg(\sum_{\mu_2} (v')^{|\mu_2|} \prod_{\square\in\mu_2}\frac{(1-m_2'q^{b_1}t^{-b_2}) (1-m'_3q^{b_1}t^{-b_2})(1-u'm_1'q^{-b_1}t^{b_2})}{(1-q^lt^{a+1})(1-q^{-l-1}t^{-a})(u'-q^{b_1}t^{-b_2})}\bigg) \nonumber \\
 &\cdot \prod_{i,j} \frac{(1-v'q^{i+1}t^{j})(1-v'm_2'm'_3q^{i+1}t^{j})}{(1-v'm_2'q^{i+1}t^{j})(1-v'm'_3q^{i+1}t^{j})} 
\end{align}
is equal to the expression 
\begin{align}
\label{int2} & \bigg(\sum_{\lambda_1} (u')^{|\lambda_1|}
\prod_{\square\in\lambda_1}\frac{(1-m_1'q^{b_1}t^{-b_2}) (1-m_2'q^{b_1}t^{-b_2})(1-v'm'_3q^{-b_1}t^{b_2})}{(1-q^lt^{a+1})(1-q^{-l-1}t^{-a})(v'-q^{b_1}t^{-b_2})}\bigg) \nonumber \\ 
&\cdot \prod_{i,j} \frac{(1-u'q^{i+1}t^{j})(1-u'm_1'm_2'q^{i+1}t^{j})}{(1-u'm_1'q^{i+1}t^{j})(1-u'm_2'q^{i+1}t^{j})}. 
\end{align}

By (\ref{compsym}),

$$\prod_{i,j} \frac{(1-v'q^{i+1}t^{j})(1-v'm_2'm'_3q^{i+1}t^{j})}{(1-v'm_2'q^{i+1}t^{j})(1-v'm'_3q^{i+1}t^{j})} =\Sym^{\bullet} \bigg( -v'\frac{q(1-m_2')(1-m_3')}{(1-q)(1-t)}\bigg ).$$


So, we regard (\ref{int1}) and  (\ref{int2}) as power series in $u', v', m'_1, m'_2, m'_3$ whose coefficients are rational functions in $q$ and $t$. Writing (\ref{int1}) in terms of $q^{-1}$ and $t$, we obtain $$\frac{F(v',m_2',m_3',m_1',u')(q,t^{-1})}{\Sym^{\bullet}\bigg (-v'\frac{(1-m_2')(1-m_3')}{(1-q^{-1})(1-t)}\bigg)};$$ similarly (\ref{int2}) may be rewritten as $$\frac{F(u',m_1',m_2',m_3',v')(q,t^{-1})}{\Sym^{\bullet}\bigg (-u'\frac{(1-m_1')(1-m_2')}{(1-q^{-1})(1-t)}\bigg)}.$$

For now, assume Proposition \ref{denominator}, which asserts that $$\Sym^{\bullet} \bigg( -v'\frac{(1-m_2')(1-m_3')}{(1-q^{-1})(1-t)}\bigg )=F(v',m_2',m_3',m_1',0)(q,t^{-1}).$$ We conclude that
 $$\frac{F(v',m'_2,m'_3,m'_1,u)(q,t^{-1})}{F(v',m'_2,m'_3,m'_1,0)(q,t^{-1})}=\frac{F(u',m'_1,m'_2,m'_3,v')(q,t^{-1})}{F(u',m'_1,m'_2,m'_3,0)(q,t^{-1})}.$$

By definition, the expression $F(v,m_2,m_3,m_1,u)(q,t^{-1})$ is symmetric under exchanging $m_2$ and $m_3$. We conclude that

\begin{align*} \frac{F(v,m_2,m_3,m_1,u')(q,t^{-1})}{F(v,m_2,m_3,m_1,0)(q,t^{-1})}&=\frac{F(u,m_1,m_2,m_3,v)(q,t^{-1})}{F(u,m_1,m_2,m_3,0)(q,t^{-1})}\\&=\frac{F(u,m_2,m_1,m_3,v)(q,t^{-1})}{F(u,m_2,m_1,m_3,0)(q,t^{-1})}\\&=\frac{F(v,m_1,m_3,m_2,u)(q,t^{-1})}{F(v,m_1,m_3,m_2,0)(q,t^{-1})}\\&=\frac{F(v,m_3,m_1,m_2,u)(q,t^{-1})}{F(v,m_3,m_1,m_2,0)(q,t^{-1})}\\&=\frac{F(u,m_2,m_3,m_1,v)(q,t^{-1})}{F(u,m_2,m_3,m_1,0)(q,t^{-1})},
\end{align*}
so that Theorem \ref{symmetry} follows from Proposition \ref{denominator}.

Proposition \ref{denominator}, then, is the result of the same procedure applied to the Calabi-Yau threefold $X_2$ with K\"{a}hler parameters as labeled in Figure \ref{kahs}. Let $\lambda$ the the two-dimensional partition assigned to the bounded vertical edge of the diagram. Under the substitutions $$m_1=-m_1'\sqrt{\frac{t}{q}},\ \  m_2=-m_2'\sqrt{\frac{t}{q}},\ \ u=u'\frac{q}{t}$$ the limit $Z'_{DT}(X_2)^{\sigma(r_1,r_2,r_3)}$ as $r_3\gg r_2>0\gg r_1$ becomes:

\begin{align}\label{fact1} \prod_{i,j}\frac{ (1-m_1'q^{i}t^{j+1})(1-m_2'q^{i}t^{j+1})(1-m_1'u'q^{i+1}t^{j})(1-m_2'u'q^{i+1}t^{j}) }{(1-u'q^{i+1}t^{j})(1-u'm_1'm_2'q^{i+1}t^{j})},\end{align}
while the limit $Z'_{DT}(X_2)^{\sigma(s_1,s_2,s_3)}$ as $s_2\gg 0>s_1\gg s_3$ becomes

\begin{align}\label{denom4}
& \bigg( \sum_{\lambda}(u')^{|\lambda|} q^{|\lambda|} t^{||\lambda||^2-|\lambda|}\prod_{\square\in\lambda}\frac{(1-m_1'q^{b_1}t^{-b_2})(1-m_2'q^{b_1}t^{-b_2})}{(1-q^lt^{a+1})(1-q^{l+1}t^a)} \bigg)\nonumber \\&\cdot \prod_{i,j}(1-m_1'q^{i}t^{j+1})(1-m_2'q^{i}t^{j+1}).
\end{align}

By Corollary \ref{eqlim}, the series (\ref{fact1}) and (\ref{denom4}) are equal. We obtain the equality $$ \sum_{\lambda}(u')^{|\lambda|}\prod_{\square\in\lambda}\frac{(1-m_1'q^{b_1}t^{-b_2})(1-m_2'q^{b_1}t^{-b_2})}{(1-q^lt^{a+1})(1-q^{-l-1}t^{-a})(-q^{b_1}t^{-b_2})}=\prod_{i,j}\frac{(1-u'm_1'q^{i+1}t^{j})(1-u'm_2'q^{i+1}t^{j}) }{(1-u'q^{i+1}t^{j})(1-u'm_1'm_2'q^{i+1}t^{j})};$$ here, we have also applied the $u'=0$ specialization of the identity (\ref{swap}) to the series (\ref{denom4}). Both sides of this identity are power series in $u',m'_1,m_2'$ whose coefficients are rational functions in $q$ and $t$.

By (\ref{compsym}), we have $$\prod_{i,j}\frac{(1-u'm_1'q^{i+1}t^{j})(1-u'm_2'q^{i+1}t^{j}) }{(1-u'q^{i+1}t^{j})(1-u'm_1'm_2'q^{i+1}t^{j})}=\Sym^{\bullet} \bigg( u' \frac{q(1-m_1')(1-m_2')}{(1-q)(1-t)}\bigg).$$ It follows that $$F(u',m'_1,m'_2,m'_3,0)(q,t^{-1})=\Sym^{\bullet}\bigg(-u'\frac{(1-m_1')(1-m_2')}{(1-q^{-1})(1-t)}\bigg),$$ proving Proposition \ref{denominator}.

\subsection{Choosing the toric geometries}\label{Xorigin}
We conclude this section with some brief remarks on how the toric Calabi-Yau threefolds $X_1$ and $X_2$ were selected to produce Theorem \ref{symmetry} and Proposition \ref{denominator}. 

We recall a higher-rank analog of the Hilbert scheme of points on $\C^2$. Let $M(r,n)$ denote the instanton moduli space of torsion free sheaves on $\P^2$ of rank $r$ and second Chern class $n$ framed along the line at infinity $l_{\infty}$, that is $$M(r,n)=\Big\{ (\mathcal{F}, \phi) \Big | \begin{array}{l} \mathcal{F}\ \text{a torsion free sheaf on}\ \P^2\ \text{with}\ \mathrm{rk}(\mathcal{F})=r, c_2(\mathcal{F})=n \\ \phi: \mathcal{F}|_{l_{\infty}}\xrightarrow{\sim} \mathcal{O}^{\oplus r}_{l_{\infty}} \end{array} \Big\}/\sim.$$ Each space $M(r,n)$ is smooth, and in particular, $$M(1,n)\cong (\C^2)^{[n]}.$$ As is the case for Hilbert schemes, each $M(r,n)$ carries a rank $n$ tautological bundle $\mathcal{O}^{[n]}$ whose fiber over $(\mathcal{F},\phi)$ is given by $H^1(\F(-l_{\infty})).$ See \cite[Sec. 3.3]{Ot} for a good introduction to these moduli spaces and their connection with enumerative geometry.
 
Given a toric Calabi-Yau threefold $X$ and a slope $\sigma$ that is preferred in every toric chart, the computations of Section \ref{secbuildingblocks} explain how to write $Z_{DT}(X)^{\sigma}$ using the refined topological vertex. Then, there is an extensive literature (including \cite{AKMV, AK, HIV, IKP, IKV, LLZ, Ta}) containing matches, for certain examples of toric Calabi-Yau geometries $X$, of the topological string partition functions of $X$ with series of equivariant integrals of tautological classes over the spaces $M(r,n)$. See \cite{IKP} for calculations for a broad class of example geometries $X$, those arising from triangulations of ``strips.'' Given such a geometry, this literature roughly indicates the tautological classes whose Euler characteristics can be matched with $Z_{DT}(X)^{\sigma}$. 

In particular, as a consequence of \cite[Sec 5.2.2]{HIV}, for example, one expects that preferred limits of $Z_{DT}(X_1)$ can be matched, up to a prefactor and change of variables, with a particular series of Euler characteristics over the spaces $M(2,n)$ from which the series (\ref{Fdef}) can be extracted. 

Let us be more precise. The moduli spaces $M(2,n)$ carry the action of a torus $$\begin{pmatrix} t_1 & \\ & t_2\end{pmatrix}\times \begin{pmatrix} a_1 & \\ & a_2\end{pmatrix};$$ the action of first factor is induced by its action on $\P^2$ and the second factor acts on the framings $\phi$. The torus fixed points of $M(2,n)$ are classified by pairs $(\lambda_1,\lambda_2)$ of two-dimensional partitions with $|\lambda_1|+|\lambda_2|=n$. The series in question is 

$$\sum_{z\geq 0} z^n\chi(M(2,n), \Lambda^{\bullet}_{m_1} \mathcal{O}^{[n]}\otimes \Lambda^{\bullet}_{m_2} \mathcal{O}^{[n]}\otimes \Lambda^{\bullet}_{m_3} (\mathcal{O}^{[n]})^{\vee}\otimes \Lambda^{\bullet}_{m_4} (\mathcal{O}^{[n]})^{\vee})$$
which, using equivariant localization, can be rewritten as

\begin{align}\label{M2int}
&\sum_{\lambda_1,\lambda_2} \frac{z^{|\lambda_1|+|\lambda_2|}}{\Lambda^{\bullet} (T^{\vee}_{\lambda_1,\lambda_2})}\prod_{i=1}^{2}\Big( \Small\prod_{\square\in \lambda_1}\begin{array}{l} (1-m_ia^{-1}_1t_1^{-b_1}t_2^{-b_2}) \\ \cdot(1-m_{i+2}a^{-1}_1t_1^{b_1}t_2^{b_2})\end{array} \prod_{\square'\in \lambda_2}\begin{array}{l} (1-m_ia^{-1}_2t_1^{-b'_1}t_2^{-b'_2})\\ \cdot (1-m_{i+2}a_2^{-1}t_1^{b'_1}t_2^{b'_2})\end{array}\Big).
\end{align}

Here $$T_{\lambda_1,\lambda_2}\in \Z[t_1^{\pm}, t_2^{\pm}, (\frac{a_1}{a_2})^{\pm}]$$ is the character of the tangent space to the fixed point $(\lambda_1,\lambda_2)\in M(r,n)$ and $\square\in \lambda$ has coordinates $(b_1,b_2)$ while $\square' \in \lambda'$ has coordinates $(b'_1,b'_2)$. For our purposes, it will suffice to know that \begin{align}\label{lamemp}T_{\lambda,\emptyset}=T_{\lambda}+\sum_{\square\in\lambda}\Big( \frac{a_2}{a_1}t_1^{-b_1}t_2^{-b_2}+\frac{a_1}{a_2}(t_1t_2) t_1^{b_1}t_2^{b_2}\Big)\end{align} where $T_{\lambda}$ is the tangent space to the Hilbert scheme as defined in (\ref{tanchar}).  
We refer the interested reader to \cite[Sec. 3.3.7-8]{Ot} for an combinatorial formula for $T_{\lambda_1,\lambda_2}$ using arm and leg lengths.

Indeed, up to an important prefactor and change of variables, the limits (\ref{firstlimit}) and (\ref{secondlimit}) can be matched with (\ref{M2int}). The choice of $X_1$ and substitution (\ref{spec}) are motivated by the observation that under the substitution $m_2=a_2$ in (\ref{M2int}), the following two cancellations happen. First, the terms (\ref{M2int}) with nonempty $\lambda_2$ vanish so that the series becomes \begin{align}\label{M2mid}\sum_{\lambda_1} \frac{z^{|\lambda_1|}}{\Lambda^{\bullet}(T_{\lambda,\emptyset}^{\vee})}\prod_{\square\in\lambda_1}\Big((1-\frac{m_1}{a_1}t_1^{-b_1}t_2^{-b_2})(1-\frac{a_2}{a_1}t_1^{-b_1}t_2^{-b_2})(1-\frac{m_3}{a_1}t_1^{b_1}t_2^{b_2})(1-\frac{m_4}{a_1}t_1^{b_1}t_2^{b_2})\Big).\end{align} Second, by (\ref{lamemp}), we have $$\Lambda^{\bullet}(T_{\lambda,\emptyset}^{\vee})=\Lambda^{\bullet}(T_{\lambda}^{\vee})\prod_{\square\in\lambda_1}(1-\frac{a_1}{a_2}t_1^{b_1}t_2^{b_2})(1-\frac{a_2}{a_1}\frac{1}{t_1t_2}t_1^{-b_1}t_2^{-b_2}),$$ so there is cancellation between the numerator and denominator of (\ref{M2mid}). The series becomes $$\sum_{\lambda_1} \frac{z^{|\lambda_1|}}{\Lambda^{\bullet}(T^{\vee}_{\lambda_1})}\prod_{\square\in\lambda_1}\Big( \frac{(1-\frac{m_1}{a_1}t_1^{-b_1}t_2^{-b_2})(-\frac{a_2}{a_1}t_1^{-b_1}t_2^{-b_2})(1-\frac{m_3}{a_1}t_1^{b_1}t_2^{b_2})(1-\frac{m_4}{a_1}t_1^{b_1}t_2^{b_2})}{(1-\frac{a_2}{a_1}\frac{1}{t_1t_2}t_1^{-b_1}t_2^{-b_2})}\Big),$$ which,  up to a change of variables, equals (\ref{Fdef}).

Similarly as a consequence of, for example, \cite[Sec. 5.2.1]{HIV}, one expects that preferred limits of $Z_{DT}(X_2)^{\sigma}$ can be matched with \begin{align*}\label{Hilbint} \sum_{n}\chi((\C^2)^{[n]},\Lambda^{\bullet}_{m_1}\mathcal{O}^{[n]}\otimes \Lambda^{\bullet}_{m_2}(\mathcal{O}^{[n]})^{\vee}).\end{align*} The choice of geometry $X_2$ can also be motivated by the observation that (\ref{denom4}) is obtained from (\ref{secondlimit}) by extracting the summands corresponding to pairs $(\lambda_1,\emptyset)$ and setting $m_3=m_4=v=0$.

We conclude by identifying obstacles to determining precisely which tautological classes can be profitably studied using the $DT$-invariants of some threefold $X$. 

First, in order to produce an identity from a geometry $X$ using Theorem \ref{independence}, there must be two choices of sets of parallel edges that cover the 1-skeleton of $\Delta(X)$ (corresponding to two different choices of preferred slope $\sigma$.) The author is not aware of a classification of the tautological classes whose Euler characteristics arise as $Z_{DT}(X)^{\sigma}$ for such an $X$. 

Second, as the calculations of Section \ref{eval} indicate, the precise matching of enumerative invariants with tautological classes is intricate. To obtain precise statements, it is important to keep track of prefactors and changes of variables. Given some $X$, it seems difficult to predict the prefactors and changes of variables that will appear when matching a limit $Z_{DT}(X)^{\sigma}$ with tautological classes. In particular, the exact normalizations and changes of variables used in Definition \ref{keul} and Theorem \ref{symmetry} were found by starting from the geometry $X_1$ and computing the limits of Theorem \ref{independence}.

\section{Tautological bundles on $S^{[n]}$}\label{tautological}
Let $S$ be a complex projective smooth surface and $\mathcal{L}$ be a line bundle on $S$. In this section, we use Theorem \ref{symmetry} and Proposition \ref{denominator}  to study the holomorphic Euler characteristics of certain Schur functors of the tautological bundles $\mathcal{L}^{[n]}$ on $S^{[n]}$.

The approach has three steps. First, we use the cobordism techniques of Ellingsrud-G\"{o}ttsche-Lehn to reduce to the case where $S$ is a toric surface. We then use equivariant localization to write such Euler characteristics for toric $S$ in terms of their equivariant analogs for $\C^2$; similar strategies are used in \cite{CO,GK1,GK2}, for example. Finally, we analyze this special case using the results of the previous section.

\subsection{Reduction to toric surfaces}\label{toricred}

Let $m, y,$ and $z$ be formal parameters, and set  
 $$\chi_{\Lambda}(S,\mathcal{L})=\sum_{k,n \geq 0} z^n(-m)^k\chi(S^{[n]},\Lambda^k \mathcal{L}^{[n]}),\ \ \chi_{\Sym}(S,\mathcal{L})=\sum_{k,n\geq 0} z^ny^k\chi(S^{[n]},\Sym^k \mathcal{L}^{[n]}).$$ For each $n$, the Hirzebruch-Riemann-Roch Theorem implies that $\chi_{\Lambda}$ and $\chi_{\Sym}$ can be written in the form (\ref{bigseries}). To be precise, if $\{l^{(n)}_{i}\}$ are the Chern roots of $\mathcal{L}^{[n]}$ and $\{w^{(n)}_{j}\}$ are the Chern roots of $\mathcal{T} S^{[n]},$ we have $$\chi_{\Lambda}(S,\mathcal{L})= \sum_{n=0}^{\infty} z^n\int_{S^{[n]}} \prod_{i=1}^n (1-me^{l^{(n)}_i}) \prod_{j=1}^{2n} \frac{w^{(n)}_j}{1-e^{-w^{(n)}_j}}$$ and similarly  $$\chi_{\Sym}(S,\mathcal{L})= \sum_{n=0}^{\infty} z^n\int_{S^{[n]}} \prod_{i=1}^{n} \frac{1}{1-ye^{l^{(n)}_i}} \prod_{j=1}^{2n} \frac{w^{(n)}_j}{1-e^{-w^{(n)}_j}}.$$

By \cite[Thm 4.2]{EGL}, there exist universal series $A_i\in \Q[m,z],$ for $i=1,\ldots, 4$ such that $$\chi_{\Lambda}(S,\mathcal{L}) = \exp(c_1(\mathcal{L})^2 A_1+ c_1(\mathcal{L})c_1(S)A_2+ c_1(S)^2A_3+c_2(S)A_4).$$ 

Set $$\gamma(S,\mathcal{L})=(c_1(\mathcal{L})^2, c_1(\mathcal{L})c_1(S), c_1(S)^2, c_2(S))\in \Z^4.$$ If $$\gamma(S,L)=\sum_{i} a_i \gamma(S_i,\mathcal{L}_i)$$ for $a_i\in \Z$, then 
\begin{equation} 
\label{redtoric}
\chi_{\Lambda}(S,\mathcal{L})=\prod_{i} (\chi_{\Lambda}(S_i,\mathcal{L}_i))^{a_i}. 
\end{equation} 
 Analagous results hold for $\chi_{\Sym}.$

We conclude that $\chi_{\Lambda}$ and $\chi_{\Sym}$ can be determined for arbitrary $(S,\mathcal{L})$ by their values on any quadruple $\{(S_i,\mathcal{L}_i)\}_{i=1,\ldots,4}$ for which the vectors $\{\gamma(S_i,\mathcal{L}_i)\}$ are $\Q$-linearly independent. Any set of generators $\{(S_i,\mathcal{L}_i)\}$ of the algebraic cobordism ring (see \cite{LP}) of surfaces equipped with a line bundle will satisfy this condition. In particular, such $(S_i,\mathcal{L}_i)$ can be chosen to be toric. For example, for the 4 generators \begin{align}\label{cobordismgenerators} (\P^2, \mathcal{O}), (\P^2,\mathcal{O}(1)), (\P^1\times \P^1, \mathcal{O}), (\P^1\times \P^1, \mathcal{O}(1))\end{align} considered in \cite{EGL,Tz}, we have \begin{center}\begin{tabular}{ll} $\gamma(\P^2,\mathcal{O})=(0,0,9,3)$ & $\gamma(\P^1\times \P^1, \mathcal{O})=(0,0,8,4)$ \\ $\gamma(\P^2, \mathcal{O}(1))=(1,3,9,3)$ & $\gamma(\P^1\times \P^1,\mathcal{O}(1,0))=(0,2,8,4).$\end{tabular} \end{center} 



\subsection{Reduction to equivariant $\C^2$}\label{eqred}
For the remainder of this section, let  $T$ denote the two-dimensional torus $\mathrm{diag}(t_1,t_2).$ Let $S$ be a toric surface with $T$-action such that $S^T$ consists of finitely many points and let $\mathcal{L}$ be an equivariant line bundle on $S$. Then, the series $\chi_{\Lambda}(S,\mathcal{L})$ and $\chi_{\Sym}(S,\mathcal{L})$ can be read off from the equivariant geometry near the fixed locus $S^T$. 

Namely, let $\{s_i\}$ be the points of $S^T$, and let $U_i$ be a toric chart centered at $s_i$. Let $w_{i_1}$ and $w_{i_2}$ denote the tangent weights at $s_i$, and let $l_i$ denote the $T$-weight of the fiber $\mathcal{L}_{s_i}.$

The $T$-action on $S$ induces an action on $S^{[n]}$, and in particular, an action of $T$ on $U_{i}^{[n]}$ for each $U_i$. The $T$-fixed points of $S^{[n]}$ are given by products $\prod_i \I_{\lambda^i}$ of $T$-fixed ideal sheaves $\I_{\lambda^i}\subset \mathcal{O}_{U_i}$ with $\sum |\lambda^i|=n$. The $T$-character of the Zariski tangent space to $\prod_i \I_{\lambda^i}$ is $$\sum_{i} T_{\lambda^{i}}(w_{i_1},w_{i_2})$$ while the fiber of $\mathcal{L}^{[n]}$ over ${\prod_i \I_{\lambda^i}}$ has torus character $$\sum_{i} \sum_{\square\in\lambda^i} l_i w_{i_1}^{-b_1}w_{i_2}^{-b_2}.$$ 

Now, suppose $T$ acts on $\C^2$ by scaling the coordinate axes by weights $w_{1}, w_{2}$, and let $\mathcal{O}(l)$ denote the trivial bundle over $\C^2$ twisted by a torus weight $l$. Set $\chi_{\Lambda}(\C^2,\mathcal{O}(l))(w_1,w_2)$ to be the character of $$\sum_{n,k\geq 0} z^n(-m)^k \chi({\C^2}^{[n]},\Lambda^{k}\mathcal{O}(l)^{[n]})$$ and define $\chi_{\Sym}(\C^2,\mathcal{O}(l))(w_1,w_2)$ analogously. Equivariant localization implies $$\chi_{\Lambda}(\C^2,\mathcal{O}(l))(w_1,w_2) = \sum_{\lambda}\frac{z^{|\lambda|}}{\Lambda^{\bullet}\Big (T^{\vee}_{\lambda}(w_1,w_2)\Big )}\prod_{\square\in \lambda} (1-mlw_1^{-b_1}w_2^{-b_2}).$$ 

Applying equivariant localization to $S^{[n]}$, we obtain \begin{align*} \chi_{\Lambda}(S,\mathcal{L}) &= \sum_{\{\lambda^i\}} \prod_{i}  \frac{z^{|\lambda^i|}}{\Lambda^{\bullet}\Big (T^{\vee}_{\lambda^i}(w_{i_1},w_{i_2})\Big )}\prod_{\square\in\lambda^i} (1-ml_iw_{i_1}^{-b_1}w_{i_2}^{-b_2}) \\&=\prod_{i} \sum_{\lambda^i} \frac{z^{|\lambda^i|}}{\Lambda^{\bullet}\Big (T^{\vee}_{\lambda^i}(w_{i_1},w_{i_2})\Big )}\prod_{\square\in\lambda^i} (1-ml_iw_{i_1}^{-b_1}w_{i_2}^{-b_2}) \\& = \prod_{i} \chi_{\Lambda}(\C^2,\mathcal{O}(l_i))(w_{i_1},w_{i_2}).\end{align*}

The resulting identity is an equality of equivariant Euler characteristics; to recover the nonequivariant series, one specializes $t_1=t_2=1$.

Again, analogous results hold for $\chi_{\Sym}$.  

\subsection{Exterior powers}
Using Proposition \ref{denominator}, we compute $\chi_{\Lambda}(S,\mathcal{L})$ for an arbitrary surface. 
\begin{corollary}
For a projective surface $S$ with line bundle $\mathcal{L}$, we have $$\chi(S^{[n]},\Lambda^k \mathcal{L}^{[n]})=\binom{n-k+\chi(\mathcal{O}_S)-1}{n-k}\binom{\chi(\mathcal{L})}{k}.$$
\end{corollary}
We remark that this result is a specialization of \cite{Sc2}[Thm 5.2.1], which specifies the dimension of each of the individual cohomology groups $H^i(S^{[n]},\Lambda^k \mathcal{L}^{[n]})$.

\begin{proof}
Let $F$ be as defined in (\ref{Fdef}). We have $$\chi_{\Lambda}(\C^2,\mathcal{O}(l))(w_1,w_2) = F(zml,\frac{1}{ml},0,0,0)(w_1,w_2).$$ Using Proposition \ref{denominator} and recalling (\ref{ratsym}), we compute  \begin{align*} F(zml,\frac{1}{ml},0,0,0)(w_1,w_2) &= \exp\Big(\sum_{n>0} -\frac{(zml)^n}{n} \frac{1-(\frac{1}{ml})^n}{(1-w_1^{-n})(1-w_2^{-n})}\Big)\\&=\exp\Big(\sum_{n>0}  \frac{z^n-(zml)^{n}}{n(1-w_1^{-n})(1-w_2^{-n})}\Big ) \\ &=\Sym^{\bullet}\Big(\frac{z-zml}{(1-w_1^{-1})(1-w_2^{-1})}\Big), \end{align*} where the last equality is (\ref{ratsym}). 

Section \ref{toricred} implies that it suffices to check the corollary when $S$ is toric and $\mathcal{L}$ is torus-equivariant. As in the previous section, let $\{s_i\}$ be the torus-fixed points of $S$; let $w_{i_1}, w_{i_2}$ be the tangent weights at $s_i$, and let $l_i$ be the weight of $\mathcal{L}_{s_i}$. Then, as equivariant series, we have \begin{align}\label{altsym} \chi_{\Lambda}(S,\mathcal{L})& =\prod_{i} \chi_{\Lambda}(\C^2,\mathcal{O}(l_i))(w_{i_1},w_{i_2})\nonumber \\&=\prod_{i} \Sym^{\bullet}\Big(\frac{z-zml_i}{(1-w_{i_1}^{-1})(1-w_{i_2}^{-1})}\Big)\nonumber \\&=\Sym^{\bullet}\Big (\sum_i \frac{z-zml_i}{(1-w_{i_1}^{-1})(1-w_{i_2}^{-1})}\Big ).\end{align} By equivariant localization, under the specialization $t_1=t_2=1$, the expression (\ref{altsym}) becomes \begin{align*} \Sym^{\bullet} \Big(z\chi(\mathcal{O}_S)-zm\chi(\mathcal{L})\Big )=\frac{(1-zm)^{\chi(\mathcal{L})}}{(1-z)^{\chi(\mathcal{O}_S)}}.\end{align*}
\end{proof}

\subsection{Symmetric powers}
 In contrast to exterior powers, the author is not aware of a reasonable closed form for Euler characteristics of arbitrary symmetric powers of a tautological bundle. Nonetheless, Theorem \ref{symmetry} and Proposition \ref{denominator} are well-suited to studying such symmetric powers. 

We use Theorem \ref{symmetry} to rewrite $\chi_{\Sym}(S,\mathcal{L})$ as a fraction where the Segre variable $y$ enters the denominator in a controllable manner. Note that the series $\chi_{\Sym}(\C^2,\mathcal{O}(l))(w_1,w_2)$ can be obtained by extracting all terms of the series $F(z,m,0,0,yl)(w_1,w_2)$ of equal $z$-degree and $m$-degree. By Theorem \ref{symmetry} we have \begin{align}\label{symseries} F(z,m,0,0,yl)(w_1,w_2)= \frac{F(z,m,0,0,0)(w_1,w_2)}{F(yl,0,0,m,0)(w_1,w_2)}\cdot F(yl,0,0,m,z)(w_1,w_2).\end{align} The first term on the right hand side can be described using the plethystic exponential; Proposition \ref{denominator} implies \begin{align}\label{sym1} &F(z,m,0,0,0)(w_1,w_2)= \Sym^{\bullet}\Big(\frac{zm-z}{(1-w_1^{-1})(1-w_2^{-1})}\Big)\\ &F(yl,0,0,m,0)(w_1,w_2)=\Sym^{\bullet}\Big(\frac{-yl}{(1-w_1^{-1})(1-w_2^{-1})}\Big ).\end{align} By definition, \begin{align}\label{rightpart} F(yl,0,0,m,z)(w_1,w_2)=\sum_{\lambda} \frac{(yl)^{|\lambda|}}{\Lambda^{\bullet}\Big(T^{\vee}_{\lambda}(w_1,w_2)\Big )} \prod_{\square\in\lambda} \frac{1-mzw_1^{-b_1}w_2^{-b_2}}{z-w_1^{b_1}w_2^{b_2}}.\end{align} Let $G(\C^2,\mathcal{O}(l))(w_1,w_2)$ denote the series obtained by first extracting the terms of (\ref{rightpart}) of equal $z$-degree and $m$-degree, and then setting $m=1$; we have \begin{align}\label{Gseries} G(\C^2,\mathcal{O}(l))(w_1,w_2)=\sum_{\lambda} \frac{(yl)^{|\lambda|}}{\Lambda^{\bullet}\Big (T^{\vee}_{\lambda}(w_1,w_2)\Big)} \prod_{\square\in\lambda} \frac{1-zw_1^{-b_1}w_2^{-b_2}}{-w_1^{b_1}w_2^{b_2}}.\end{align} Consequently, extracting the terms of both sides of (\ref{symseries}) of equal $z$-degree and $m$-degree, and then setting $m=1$ yields $$\chi_{\Sym}(\C^2,\mathcal{O}(l))(w_1,w_2)= \Sym^{\bullet}\Big(\frac{z+yl}{(1-w_1^{-1})(1-w_2^{-1})} \Big)G(\C^2,\mathcal{O}(l))(w_1,w_2).$$

Suppose $S$ is toric with equivariant line bundle $\mathcal{L}$ and let $\{s_i\}, \{w_{i_1},w_{i_2}\}$ and $\{\l_i\}$ be the torus-fixed points, tangent weights at fixed points, and weights of $\mathcal{L}$ at fixed points, as before. As $T$-equivariant Euler characteristics, we have \begin{align*} \chi_{\Sym}(S,\mathcal{L})& =\prod_{i} \chi_{\Sym}(\C^2,\mathcal{O}(l_i))(w_{i_1},w_{i_2}) \\ &= \prod_{i} \Sym^{\bullet}\Big (\frac{z+yl_i}{(1-w_{i_1}^{-1})(1-w_{i_2}^{-1})} \Big)G(\C^2,\mathcal{O}(l_i))(w_{i_1},w_{i_2}) \\ & = \Sym^{\bullet}\Big(\sum_{i} \frac{z+yl_i}{(1-w_{i_1}^{-1})(1-w_{i_2}^{-1})} \Big) \prod_{i} G(\C^2,\mathcal{O}(l_i))(w_{i_1},w_{i_2}). \end{align*}  

 Set \begin{align}\label{Gdef}G(S,\mathcal{L})=\Big(\prod_{i} G(\C^2,\mathcal{O}(l_i))(w_{i_1},w_{i_2})\Big)\bigg|_{t_1=t_2=1}.\end{align}

By equivariant localization, \begin{align*} &\Big(\sum_{i} \frac{z}{(1-w_{i_1}^{-1})(1-w_{i_2}^{-1})}\Big )\bigg|_{t_1=t_2=1}=z\chi(\mathcal{O}_S),\\ &\Big(\sum_{i} \frac{yl_i}{(1-w_{i_1}^{-1})(1-w_{i_2}^{-1})}\Big)\bigg|_{t_1=t_2=1}=y\chi(\mathcal{L});\end{align*} we conclude that \begin{align}\label{simplesym}\chi_{\Sym}(S,\mathcal{L}) =\frac{1}{(1-z)^{\chi(\mathcal{O}_S)}(1-y)^{\chi(\mathcal{L})}} G(S,\mathcal{L}).\end{align} 

We use (\ref{simplesym}) to compute Euler characteristics of symmetric powers of tautological line bundles. 

In the special case when $\chi(\mathcal{O}_S)=1$, we have the following stability result.

\begin{corollary}\label{chi1}
Let $S$ be a projective surface with $\chi(\mathcal{O}_S)=1$, and let $\mathcal{L}$ be a line bundle on $S$. Then, for $n\geq k$, we have $$\chi(S^{[n]},\Sym^k \mathcal{L}^{[n]})=\binom{\chi(\mathcal{L})+k-1}{k};$$ in particular, the right hand side does not depend on $n$.
\end{corollary}

\begin{proof}
It suffices to check the result for toric $S$ and equivariant $\mathcal{L}$. As $\chi(\mathcal{O}_S)=1$, by (\ref{simplesym}) we have $$\chi_{\Sym}(S,\mathcal{L})=\frac{1}{(1-z)(1-y)^{\chi(\mathcal{L})}} G(S,\mathcal{L}).$$ Now, fix $k\geq 0$, and set $G_k(S,\mathcal{L})$ to be the sum of all terms of $G(S,\mathcal{L})$ of $y$-degree less than or equal to $k$. As any monomial in $G(S,\mathcal{L})$ has $z$-degree no larger than $y$-degree, we may write $$G_k(S,\mathcal{L})=\sum_{j=0}^{k}g_j(S,\mathcal{L})z^j,$$ where $g_j(S,\mathcal{L})\in \Z[y](t_1,t_2).$ Then, for any $n\geq k$, the $z^ny^k$-term of $\chi_{\Sym}(S,\mathcal{L})$ is the $y^k$-term of the expression $$\frac{1}{(1-y)^{\chi(\mathcal{L})}} \sum_{j=0}^k g_j(S,\mathcal{L}) =\frac{1}{(1-y)^{\chi(\mathcal{L})}} G_k(S,\mathcal{L})|_{z=1}.$$ As every two-dimensional partition $\lambda$ contains the box $(0,0)$, it follows from (\ref{Gseries}) that for any $T$-weights $l, w_1$ and $w_2$, we have $$G_k(\C^2,\mathcal{O}(l))(w_1,w_2)|_{z=1}=0$$ and therefore $$G_k(S,\mathcal{L})|_{z=1}=1.$$ We conclude that $\chi(S^{[n]},\Sym^k \mathcal{L}^{[n]})$ is the $y^k$-coefficient of $1/(1-y)^{\chi(\mathcal{L})}$.
\end{proof}

Even when $\chi(\mathcal{O}_S)\not=1$, the computation in Corollary \ref{chi1} implies that for for fixed $k$, the series $$\sum_{n=1}^{\infty} z^n \chi(S^{[n]},\Sym^k \mathcal{L}^{[n]})$$ is determined by $G_k(S,\mathcal{L}).$ We conclude the following. 

\begin{corollary}
The series $\sum_{n=1}^{\infty} z^n \chi(S^{[n]},\Sym^k \mathcal{L}^{[n]})$ is determined by its first $k$ coefficients.
\end{corollary}

We conclude with a computation of the series 

$$\sum_{n=0}^{\infty} z^ny^k \chi(S^{[n]},\Sym^k \mathcal{L}^{[n]})$$ up to $y$-degree 3. The same approach can be used up to any fixed $y$-degree.

\begin{corollary}
We have 
\begin{align}
\label{Sym0}&\chi(S^{[n]},\mathcal{O}_{S^{[n]}})=\binom{\chi(\mathcal{O}_S)+n-1}{n}\\
\label{Sym1} &\chi(S^{[n]},\mathcal{L}^{[n]})=\binom{\chi(\mathcal{O}_S)+n-2}{n-1}\chi(\mathcal{L}) \\
\label{Sym2}&\chi(S^{[n]},\Sym^2\mathcal{L}^{[n]})=\binom{\chi(\mathcal{O}_S)+n-3}{n-1}\chi(\mathcal{L}^{\otimes 2})+\binom{\chi(\mathcal{O}_S)+n-3}{n-2}\binom{\chi(\mathcal{L})+1}{2} \\
\label{Sym3}&\chi(S^{[n]},\Sym^3\mathcal{L}^{[n]})=\binom{\chi(\mathcal{O}_S)+n-3}{n-1}\chi(\mathcal{L}^{\otimes 3}) \nonumber\\
&\hspace{1.2in} + \binom{\chi(\mathcal{O}_S)+n-4}{n-2}\big(\chi(\mathcal{L}^{\otimes 2})\chi(\mathcal{L})-\chi(\mathcal{L}^{\otimes 3}\otimes \T S^{\vee})\big)\nonumber \\ & \hspace{1.2in}+\binom{\chi(\mathcal{O}_S)+n-4}{n-3}\binom{\chi(\mathcal{L})+2}{3} 
\end{align}
\end{corollary}

\begin{proof}

It again suffices to prove the result for toric $S$ and equivariant $\mathcal{L}$. We compute the series $G(S,\mathcal{L})$ up to $y$-degree $3$. Again, let $\{s_i\}, \{w_{i_1},w_{i_2}\}$ and $\{\l_i\}$ be the torus-fixed points, tangent weights at fixed points, and weights of $\mathcal{L}$ at fixed points, and recall (\ref{Gdef}).

 Computing $G(\C^2,\mathcal{O}(l))$ explicitly up to $y$-degree $3$, one obtains \begin{align*}G(\C^2,\mathcal{O}(l))(t_1,t_2)=\Sym^{\bullet}&\Big( \frac{-yl(1-z)+y^2l^2(z-z^2)+y^3l^3(z-z^2)(1-z(t_1^{-1}+t_2^{-1}))}{(1-t_1^{-1})(1-t_2^{-1})}  \Big)\\&+O(y^4).\end{align*} Rewriting the right-hand side in terms of Euler characteristics of equivariant bundles over $\C^2$, we have \begin{align}\label{Gdeg3} G(\C^2,\mathcal{O}(l))(t_1,t_2)=\Sym^{\bullet}\Big(&-y(1-z)\chi(\C^2,\mathcal{O}(l))+y^2(z-z^2)\chi(\C^2,\mathcal{O}(l)^{\otimes 2})\nonumber\\&+y^3(z-z^2)\big(\chi(\C^2,\mathcal{O}(l)^{\otimes 3})-z\chi(\C^2,\mathcal{O}(l)^{\otimes 3}\otimes \T {\C^2}^{\vee})\big)\Big)\nonumber \\&+O(y^4),\end{align} where, in all Euler characteristics on the right hand side, the torus $T$ scales the coordinate axes of $\C^2$ by $t_1$ and $t_2$.
 Though not necessary for our argument, we remark that one expects an equality of the form (\ref{Gdeg3}) on general grounds. To be more precise, it is a consequence of \cite[Lemma 5.3.4]{O} that there exist $K$-theory classes $\mathcal{F}_{k,n}\in K_T(\C^2)$ for which $$G(\C^2,\mathcal{O}(l))(t_1,t_2)=\Sym^{\bullet}\Big( \sum_{k\geq 1,n\geq 0}y^kz^n \chi(\C^2,\mathcal{F}_{k,n})\Big).$$
 
Now, recall from (\ref{Gdef}) that  \begin{align*} G(S,\mathcal{L})&=\prod_{i}G(\C^2,\mathcal{O}(l_i))(w_{i_1},w_{i_2})|_{t_1=1,t_2=1}.\end{align*} 
 By (\ref{Gdeg3}) and equivariant localization, we obtain
\begin{align*} G(S,\mathcal{L})& = \Sym^{\bullet}\Big(-y(1-z)\chi(S,\mathcal{L})+y^2(z-z^2)\chi(S,\mathcal{L}^{\otimes 2})\\& \ \ \ \  \ \ \ \  \ \ \ \ \ \ +y^3(z-z^2)\big(\chi(S,\mathcal{L}^{\otimes 3})-z\chi(S,\mathcal{L}^{\otimes 3}\otimes \T S^{\vee})\big)\Big)+O(y^4) \\ & =\big(\frac{1-y}{1-zy}\big)^{\chi(\mathcal{L})} \big(\frac{1-z^2y^2}{1-zy^2}\big)^{\chi(\mathcal{L}^{\otimes 2})}\big(\frac{1-z^2y^3}{1-zy^3}\big)^{\chi(\mathcal{L}^{\otimes 3})}\big(\frac{1-z^2y^3}{1-z^3y^3}\big)^{\chi(\mathcal{L}^{\otimes 3}\otimes \T S^{\vee})}\\ &  \ \ \ \ \ +O(y^4).\end{align*}
 By (\ref{simplesym}) we have $$\chi_{\Sym}(S,\mathcal{L})=\frac{1}{(1-z)^{\chi(\mathcal{O}_S)}(1-y)^{\chi(\mathcal{L})}}G(S,\mathcal{L}).$$  The corollary follows.

\end{proof}

The individual cohomology groups of bundles $\Sym^{k}\mathcal{L}^{[n]}$ have also been studied in connection with the strange duality conjecture for 
$\P^2$. Namely, equation (\ref{Sym1}) is shown in \cite[Prop 5.6(ii)]{EGL}, and (\ref{Sym2}, \ref{Sym3}) are consequences of \cite[Thm 5.25]{Sc}. In the cases when $n=2$ or $3$, equation (\ref{Sym2}) is also a consequence of \cite[Thm 1.1, 1.2]{Da}. We remark that the argument of this section can also be used to compute the values $\chi(S^{[n]},\Sym^k \mathcal{L}_1^{[n]}\otimes \mathcal{D}_{\mathcal{L}_2})$ studied for example in \cite{Da, Sc}; here, given a line bundle $\mathcal{L}$ on $S$, the line bundle $\mathcal{D}_{\mathcal{L}}$ denotes the pullback under the Hilbert-Chow morphism of the bundle $\mathcal{L}^{\boxtimes n}/\mathfrak{S}^n$ on $\Sym^{n}S$.

\subsection{Rank 2 vector bundles on $S$}
As for line bundles, to any rank $r$ vector bundle $\V$ on $S$ there is an associated rank $rn$ tautological vector bundle $\V^{[n]}$ on $S^{[n]}$ whose fiber over $Y\in S^{[n]}$ is $H^0(\mathcal{O}_Y\otimes \V)$.

Let $\V$ denote a rank $2$ vector bundle on $S$. In this section we study the following series: 

$$\chi_{\Lambda}(S,\V)=\sum_{n,k\geq 0} z^n(-m)^k \chi(S^{[n]},\Lambda^k \V^{[n]}).$$

As in the previous section, the author is not aware of a concise closed formula for an arbitrary term of this series. However, just as for $\chi_{\Sym}(S,\mathcal{L})$, Theorem \ref{symmetry} can be used to compute the series up to any fixed degree in $m$. For example, we show the following: 
\begin{corollary}
For a projective surface $S$ with rank 2 vector bundle $\V$, we have 
\begin{align} 
&\label{Wedgerk21}\chi(S^{[n]},\V^{[n]})= \binom{\chi(\mathcal{O}_S)+n-2}{n-1}\chi(\V) \\
&\chi(S^{[n]},\Lambda^2\V^{[n]})=\binom{\chi(\mathcal{O}_S)+n-3}{n-2}\binom{\chi(\V)}{2}+\binom{\chi(\mathcal{O}_S)+n-3}{n-1}\chi(\Lambda^2\V) \\
&\chi(S^{[n]}, \Lambda^3\V^{[n]}) \nonumber=\binom{\chi(\mathcal{O}_S)+n-4}{n-3}\binom{\chi(\V)}{3}\\ &\hspace{1.2in}+\binom{\chi(\mathcal{O}_S)+n-4}{n-2}\big(\chi(\V)\chi(\Lambda^2\V)-\chi(\V\otimes\Lambda^2\V)\big)
\end{align}
\end{corollary} 

\begin{proof}

The framework of Sections \ref{toricred} and \ref{eqred} also apply to series of the form (\ref{bigseries}) when $\mathcal{L}$ is replaced by a vector bundle of arbitrary rank. That is by \cite[Thm 4.2]{EGL}, there exist universal series $B_i\in \Q[m,z],$ for $i=1,\ldots, 5$ such that $$\chi_{\Lambda}(S,\mathcal{V}) = \exp(c_1(\mathcal{\V})^2 B_1+ c_1(\mathcal{\V})c_1(S)B_2+ c_1(S)^2B_3+c_2(S)B_4+c_2(\V)B_5).$$ So, as in Section \ref{toricred}, the series $\chi_{\Lambda}(S,\V)$ is determined for arbitrary projective $S$ and rank 2 bundle $\V$ by its values for toric $S$ and equivariant $\mathcal{V}$. 

Then, as in Section \ref{eqred}, for such $(S,\V)$, the series can be written in terms of an equivariant version when $S=\C^2$. That is, given two $T$-weights $l_1,l_2$, and an action of $T$ on $\C^2$ such that the coordinate axes are scaled by weights $w_1, w_2$, we set \begin{align}\label{rk2ser} \chi_{\Lambda}(\C^2,\mathcal{O}(l_1)\oplus \mathcal{O}(l_2))(w_1,w_2)&=\sum_{n,k\geq 0} z^n(-m)^k \chi\big((\C^2)^{[n]},\Lambda^k \big(\mathcal{O}(l_1)\oplus \mathcal{O}(l_2)\big)^{[n]}\big)\nonumber\\& =\sum_{\lambda}\frac{z^{|\lambda|}}{\Lambda^{\bullet}\big( T^{\vee}_{\lambda}(w_1,w_2)\big)}\prod_{\square\in\lambda}\begin{array}{l}(1-ml_1t_1^{-b_1}t_2^{-b_2})\\ \cdot (1-ml_2t_1^{-b_1}t_2^{-b_2})\end{array},\end{align} where the right hand side of the first line is regarded as a series of $T$-equivariant Euler characteristics, and the second equality is an application of equivariant localization. Then, if $\{s_i\}, \{w_{i_1},w_{i_2}\}$ and $\{l_{i_1}, l_{i_2}\}$ are the torus-fixed points, tangent weights at $s_i$, and weights of $\V$ at $s_i$ respectively, we have \begin{align}\label{rk2toric}\chi_{\Lambda}(S,\mathcal{V})= \prod_{i} \chi_{\Lambda}(\C^2,\mathcal{O}(l_{i_1})\oplus \mathcal{O}(l_{i_2}))(w_{i_1},w_{i_2})\Big|_{t_1=1,t_2=1}.\end{align}

Note that \begin{align}\label{r2y0}\chi_{\Lambda}(\C^2,\mathcal{O}(l_1)\oplus \mathcal{O}(l_2))(t_1,t_2)=F(zml_1,\frac{1}{ml_1}, 0, \frac{ml_2}{y}, y)(t_1,t_2)|_{y=0};\end{align} in particular, the specialization $y=0$ on the right hand side is well-defined. By Theorem \ref{symmetry}, we have \begin{align}\label{rk2sym}&F(zml_1,\frac{1}{ml_1}, 0, \frac{ml_2}{y}, y)(t_1,t_2)\nonumber\\&=\frac{F(zml_1,\frac{1}{ml_1}, 0, \frac{ml_2}{y}, 0)(t_1,t_2)}{F(y,\frac{ml_2}{y}, 0, \frac{1}{ml_1}, 0)(t_1,t_2)}F(y,\frac{ml_2}{y}, 0, \frac{1}{ml_1}, zml_1)(t_1,t_2). \end{align}

By Proposition \ref{denominator}, we have \begin{align}\label{rk2d1}F(zml_1,\frac{1}{ml_1}, 0, \frac{ml_2}{y}, 0)(t_1,t_2)=\Sym^{\bullet}\Big(\frac{z(1-ml_1)}{(1-t_1^{-1})(1-t_2^{-1})}\Big)\end{align} and 
\begin{align*}F(y,\frac{ml_2}{y}, 0, \frac{1}{ml_1}, 0)(t_1,t_2)=\Sym^{\bullet}\Big(\frac{ml_2(1-\frac{y}{ml_2})}{(1-t_1^{-1})(1-t_2^{-1})}\Big),\end{align*} so that  
\begin{align}\label{rk2d2}F(y,\frac{ml_2}{y}, 0, \frac{1}{ml_1}, 0)(t_1,t_2)\Big|_{y=0}=\Sym^{\bullet}\Big(\frac{ml_2}{(1-t_1^{-1})(1-t_2^{-1})}\Big).\end{align}
Note that \begin{align}\label{rk2swap}F(y,\frac{ml_2}{y}, 0, \frac{1}{ml_1}, zml_1)(t_1,t_2)=\sum_{\lambda} \frac{(ml_2)^{|\lambda|}}{\Lambda^{\bullet}T^{\vee}_{\lambda}}\prod_{\square\in\lambda}\frac{(1-\frac{y}{ml_2}t_1^{-b_1}t_2^{-b_2})(1-zt_1^{-b_1}t_2^{-b_2})}{(1-zml_1t_1^{-b_1}t_2^{-b_2})}.\end{align} 
In particular, each factor on the right hand side of (\ref{rk2sym}) is well-defined under the specialization $y=0$, and only finitely many terms of (\ref{rk2swap}) contribute to any fixed degree in $m$. Expanding, for example, up to degree 3 in $m$, we have \begin{align} \label{rk2swap2}F&(y,\frac{ml_2}{y}, 0, \frac{1}{ml_1}, zml_1)(t_1,t_2)\Big|_{y=0}\nonumber \\&=\Sym^{\bullet}\Big(\frac{m(1-z)l_2+m^2(1-z)zl_1l_2+m^3(1-z)z^2(l_1+l_2)l_1l_2}{(1-t_1^{-1})(1-t_2^{-1})}\Big)+O(m^4).\end{align}

Combining (\ref{rk2d1}), (\ref{rk2d2}) and (\ref{rk2swap2}), we conclude from (\ref{r2y0}) and (\ref{rk2sym}) that \begin{align}\chi_{\Lambda}&(\C^2,\mathcal{O}(l_1)\oplus \mathcal{O}(l_2))(t_1,t_2)\nonumber\\&=\Sym^{\bullet}\Big(\frac{z-mz(l_1+l_2)+m^2(1-z)zl_1l_2+m^3(1-z)z^2(l_1+l_2)l_1l_2}{(1-t_1^{-1})(1-t_2^{-1})}\Big)+O(m^4).\end{align}

Writing in terms of Euler characteristics over $\C^2$ (where $T$ acts by scaling the coordinate axes by $t_1$ and $t_2$), we have \begin{align*}\chi_{\Lambda}(\C^2,\mathcal{O}(l_1)\oplus \mathcal{O}&(l_2))(t_1,t_2)\nonumber\\=\Sym^{\bullet}\Big(&z\chi(\C^2,\mathcal{O})-mz\chi(\C^2,\mathcal{O}(l_1)\oplus\mathcal{O}(l_2))\\&+m^2(1-z)z\chi\big(\C^2,\Lambda^2\big(\mathcal{O}(l_1)\oplus\mathcal{O}(l_2)\big)\big)\\&+m^3(1-z)z^2\chi\big(\C^2,\mathcal{O}(l_1)\oplus\mathcal{O}(l_2)\otimes\Lambda^2\big(\mathcal{O}(l_1)\oplus\mathcal{O}(l_2)\big)\big)\Big)+O(m^4).\end{align*}

So, from (\ref{rk2toric}) and equivariant localization, we conclude that \begin{align*}\chi_{\Lambda}(S,\mathcal{V})&=\Sym^{\bullet}\Big(z\chi(S,\mathcal{O}_S)-mz\chi(S,\V)+m^2(1-z)z\chi(S,\Lambda^2\V)\\&\ \ \ \ \ \ \ \ \ \ \ \ \ \ +m^3(1-z)z^2\chi(S,\V\otimes\Lambda^2\V)\Big)+O(m^4)\\& =\frac{(1-zm)^{\chi(\V)}}{(1-z)^{\chi(\mathcal{O}_S)}}\Big(\frac{1-z^2m^2}{1-zm^2}\Big)^{\chi(\Lambda^2\V)}\Big(\frac{1-z^3m^3}{1-z^2m^3}\Big)^{\chi(\V\otimes\Lambda^2\V)}+O(m^4).\end{align*}

The corollary follows. Note that (\ref{Wedgerk21}) (and the same statement for $\V$ of arbitrary rank) also follows from (\ref{Sym1}).

\end{proof}

It would be interesting if the techniques of this paper could be applied to study Schur functors of other bundles over the Hilbert scheme of points on a surface $S$ and moduli of higher rank sheaves that are of interest in algebraic geometry and representation theory. Examples include the determinants of tautological bundles associated to higher rank vector bundles on $S$ (which appear in the ``Verlinde series'' studied in \cite{J,MOP2}), tangent and cotangent bundles (which are relevant to the monopole branches of Vafa-Witten invariants studied in \cite{Th2}) and the Procesi bundle on the Hilbert scheme of points on $\C^2$. 

\end{document}